\documentclass[11pt]{amsart}
\usepackage{amsmath,amssymb,amsfonts}
\usepackage{amsthm}
\usepackage{array}
\usepackage{resizegather}
\usepackage[shortlabels]{enumitem}
\usepackage{color}
\usepackage[normalem]{ulem}
\usepackage{cancel}
\usepackage{url}

\theoremstyle{plain}
\newtheorem{thm}{Theorem}[section]
\newtheorem{thmx}{Theorem}

\newtheorem{lemma}[thm]{Lemma}
\newtheorem{proposition}[thm]{Proposition}
\newtheorem{corollary}[thm]{Corollary}

\theoremstyle{definition}

\newtheorem{remark}{Remark}[section]

\numberwithin{equation}{section}

\newcommand{\bl}{\overline{\lambda}}
\newcommand{\bo}{{\rm O}}
\newcommand{\cw}{C_{W,n}}
\newcommand{\ds}{\displaystyle}

\newcommand{\dsum}{\ds\sum}

\newcommand{\eqskip}{ \vspace*{2mm}\\ }
\newcommand{\R}{\mathbb{R}}
\newcommand{\N}{\mathbb{N}}

\newcommand{\so}{{\rm o}}

\newcommand{\lc}{\lambda}
\newcommand{\ld}{\lambda}

\newcommand{\fr}[2]{\frac{\ds #1}{\ds #2}}
\newcommand{\hs}[1]{\mathbb{S}^{#1}_{+}}
\newcommand{\sn}[1]{\mathbb{S}^{#1}}
\newcommand{\wdg}[2]{\mathbb{W}^{#1}_{#2}}

\newcommand{\txtb}{\textcolor{black}}

\newcommand{\up}{\mathcal{U}}
\newcommand{\lo}{\mathcal{L}}

\newcommand{\La}[1]{\mbox{\Large $#1$}}
\newcommand{\LA}[1]{\mbox{\LARGE $#1$}}

\newcommand{\soutg}{\bgroup\markoverwith{\textcolor{green}{\rule[.5ex]{2pt}{1pt}}}\ULon}
\newcommand{\soutb}{\bgroup\markoverwith{\textcolor{blue}{\rule[.5ex]{2pt}{1pt}}}\ULon}
\newcommand{\soutr}{\bgroup\markoverwith{\textcolor{red}{\rule[.5ex]{2pt}{1pt}}}\ULon}

\begin{document}

\title[P\'{o}lya-type inequalities]
{P\'{o}lya-type inequalities on spheres and hemispheres}

\author[P. Freitas, J. Mao, I. Salavessa]{Pedro Freitas, Jing Mao, Isabel Salavessa}
\address{Departamento de Matem\'{a}tica, Instituto Superior T\'{e}cnico, Universidade de Lisboa, Av. Rovisco Pais, 1049-001 Lisboa, Portugal
\& Grupo de F\'{\i}sica Matem\'{a}tica, Faculdade de Ci\^{e}ncias, Universidade de Lisboa, Campo Grande, Edif\'{\i}cio C6,
1749-016 Lisboa, Portugal}
\email{psfreitas@fc.ul.pt}
\address{Faculty of Mathematics and Statistics, Key Laboratory of Applied Mathematics of Hubei Province,
Hubei University, Wuhan, 430062, China}
\email{jiner120@163.com}
\address{Grupo de F\'{\i}sica Matem\'{a}tica, Faculdade de Ci\^{e}ncias, Universidade de Lisboa, Campo Grande, Edif\'{\i}cio C6,
1749-016 Lisboa, Portugal
\& Departamento de F\'{\i}sica, Instituto Superior T\'{e}cnico, Universidade de Lisboa, Av. Rovisco Pais, 1049-001 Lisboa, Portugal} 
\email{isabel.salavessa@tecnico.ulisboa.pt}

\date{\today}

\begin{abstract}
Given an eigenvalue $\lambda$ of the Laplace-Beltrami operator on $n-$spheres or
$-$hemi\-spheres, with multiplicity $m$ such that $\lambda=\lambda_{k}=\dots = \lambda_{k+m-1}$, we characterise the lowest
and highest orders in the set $\left\{k,\dots,k+m-1\right\}$ for which P\'{o}lya's conjecture holds and fails.
In particular, we show that P\'{o}lya's conjecture holds for hemispheres in the Neumann case, but not in the Dirichlet case
when $n$ is greater than two.
We further derive P\'{o}lya-type inequalities by adding a correction term providing sharp lower and upper bounds for all
eigenvalues. This allows us to measure the deviation from the leading term in the Weyl asymptotics for eigenvalues
on spheres and hemispheres. As a direct consequence, we obtain similar results for domains which tile hemispheres.
We also obtain direct and reversed Li-Yau inequalities for $\mathbb{S}^2$ and $\mathbb{S}^4$, respectively.
\end{abstract}
\subjclass[2010]{}
\maketitle

\tableofcontents

\section{Introduction \label{sec:Introduction}}

Let $M$ be a smooth compact $n$-dimensional Riemannian manifold with metric $g$ and consider the Laplace-Beltrami operator in $M$. In
the case where $M$ has no boundary we consider the closed eigenvalue problem
\[
 \Delta u + \lc u = 0,
\]
and denote the corresponding sequence of eigenvalues by $0 = \lc_{0} \leq \lc_{1}\leq \lc_{2}\leq \dots$. When the boundary
$\partial M$ is non-empty, we impose \txtb{either Dirichlet or Neumann} boundary conditions, and consider the problems
\[
\begin{array}{lll}
\left\{
\begin{array}{lll}
 \Delta u + \ld u = 0, & \mbox{ in } M\eqskip
 u = 0, & \mbox{ on } \partial M
\end{array}
\right.
& \mbox{ and } &
\left\{
\begin{array}{lll}
 \Delta u + \mu u = 0, & \mbox{ in } M\eqskip
 \fr{\partial u}{\partial \nu} = 0, & \mbox{ on } \partial M
\end{array}
\right.,
\end{array}
\]
\txtb{where $\nu$ denotes the outer unit normal to $\partial M$. For Dirichlet and Neumann boundary conditions} the spectrum is now \txtb{given by
$0<\ld_{1}\leq \ld_{2} \leq \dots$ and $0=\mu_{0} \leq \mu_{1} \leq \mu_{2}\dots$, respectively}. In all three cases the sequence of eigenvalues
converges to infinity and satisfies the Weyl asymptotics \txtb{ for the manifold $M$}~\cite{hw,sava}
\begin{equation}\label{weyl}
 \lc_{k} = \cw(M) k^{2/n} + \so\left(k^{2/n}\right), \mbox{ as } k\to\infty,
\end{equation}
where the Weyl constant $\cw$ is given by
\[
 \cw = \txtb{\cw\left(M\right) :=} \fr{4\pi^2 }{\left(\omega_{n} |M|\right)^{2/n}}.
\]
Here $\omega_{n}$ and $|M|$ denote the volume of the unit ball in $\R^{n}$ and the volume of $M$, respectively -- \txtb{although defined
in this way $\cw$ depends explicitly on the volume of $M$, whenever it will be clear from the context which manifold we are referring to,
to simplify notation we will omit the dependence on $M$.} In the case of the Dirichlet problem, and under certain further conditions on $M$,
it can be shown that the remainder term in~\eqref{weyl} is of the 
form~\cite{sava}
\begin{equation}\label{remainder}
 \frac{\ds 2\pi^2\omega_{n-1}|\partial M|k^{1/n}}{\ds n \left(\omega_{n}|M|\right)^{1+1/n}} + \so\left(k^{1/n}\right),
\end{equation}
where now $|\partial M|$ denotes the $(n-1)-$measure of the boundary of $M$. One consequence
of this result is that for such a manifold and a sufficiently large order $k$ of the eigenvalue, $\lambda_{k}$
must be larger than the first term in the Weyl asymptotics~\eqref{weyl}, that is
\begin{equation}\label{Polya_ineq}
 \lambda_{k}\txtb{(M)} \geq \cw\txtb{(M)} k^{2/n}.
\end{equation}
 \txtb{In the Neumann case, the coefficient in $k^{1/n}$ in~\eqref{remainder} is negative, and the corresponding inequality is now
\[
 \mu_{k}(M) \leq \cw\txtb{(M)} k^{2/n}.
\]
}
In fact, in 1954 P\'{o}lya conjectured \txtb{these inequalities to hold} for {\it all} Dirichlet \txtb{and Neumann} eigenvalues of the Laplace
operator on bounded Euclidean domains~\cite{poly1}. A few years later, in 1961, he went on to prove this conjecture \txtb{for the Dirichlet problem} in the special case
of Euclidean domains which tile the plane~\cite{poly}, \txtb{and provide a partial result for the Neumann problem, completed in $1966$ by Kellner~\cite{kell}}.
Although some progress has been made since then, \txtb{and it is now known that~\eqref{Polya_ineq} is satisfied for some non-tiling domains such as certain Cartesian products~\cite{La}
or sufficiently thin sectors of some classes of spherically symmetric domains~\cite{frsa}, for instance}, the general case\txtb{, which has
become known as P\'{o}lya's conjecture or inequality,} remains open to this day.

\txtb{By exploring the examples of spheres and hemispheres, we aim a better understanding of the relations between P\'{o}lya-type inequalities
and Weyl asymptotics, in particular in situations where the latter do not possess a second term of the form~\eqref{remainder}. Our
results provide us with an idea of what we can possible expect to hold for more general situations, such as in
the case of manifolds for which all geodesics are closed~\cite{bess}, which would be a natural next step to consider in light of the present results.
}

\txtb{In another direction, inequalities for the average of the first $k$ eigenvalues which are the best possible compatible
with~\eqref{Polya_ineq} were obtained by Li and Yau for a general bounded domain $\Omega$ in $\R^{n}$ with Dirichlet
boundary conditions~\cite{ly}. These read as
\begin{equation}\label{euclid_liyau}
 \fr{1}{k}\, \dsum_{j=1}^{k} \lambda_{j}\left(\Omega\right) \geq \fr{n}{n+2} \cw\left(\Omega\right) k^{2/n},
\end{equation}
and we will also consider such inequalities below in the case of spheres -- for other recent estimates of this type and Riesz means on
spheres and hemispheres, see~\cite{blps}.
}

\txtb{A first motivation for} the present paper was thus to consider examples of manifolds where P\'{o}lya's conjecture does not hold, and to
see how~\eqref{Polya_ineq} can be modified to yield a set of valid sharp inequalities for the corresponding Laplacian eigenvalues.
More precisely, we shall consider the $n-$dimensional sphere
\[
\sn{n} = \left\{ x\in\R^{n+1}: \| x \| = 1\right\},
\]
with the canonical round metric, and the corresponding hemisphere
\[
 \hs{n} = \left\{ x\in\R^{n+1}: \| x \| = 1 \wedge x_{n+1}>0\right\},
\]
with \txtb{either Dirichlet or Neumann} boundary conditions on the equator. In the former case there is no boundary, and in the latter case the remainder term
in the Weyl asymptotics is not of the form given in~\eqref{remainder}. The behaviour of the remainder for such
manifolds  has been the object of much study in the literature -- see~\cite{cg} and the references therein, for instance, for recent
progress on this topic. This is another reason why we believe it is of interest to provide sharp inequalities of the type
given here as, indeed, P\'{o}lya's inequality~\txtb{\eqref{Polya_ineq}} will not hold for such manifolds in general, the \txtb{only} known exception
\txtb{until now} being $\hs{2}$~\cite{bb,grom}. Furthermore, some known Weyl remainder estimates are sharp for the round sphere -- see~\cite{cg,horm,sava} for
some examples and references. In fact, although these are very specific domains, because of their high degree of
symmetry \txtb{(and hence high eigenvalue multiplicities),} they are known to be useful benchmarks against which to compare general results.

\txtb{Furthermore, most of the results related to this have, until now and to the best of our knowledge, concentrated on asymptotics, such as those
given in~\cite{horm,sava}, for instance. Another purpose of the paper is thus} to go beyond these asymptotics and complement them with
\txtb{upper and lower} bounds which are both valid for all eigenvalues and sharp
in the first two terms. These results fall into two categories. On the one hand, we characterise instances
of the eigenvalues of hemispheres for which P\'{o}lya's inequality is and is not satisfied.  On the other hand, we determine and prove
modified (sharp) versions of P\'{o}lya's inequality for both spheres and hemispheres by adding correction terms, providing both
lower and upper bounds allowing us to measure the deviation from the first term in~\eqref{weyl} \txtb{in a precise way}.
 \txtb{These bounds give us what we believe to be relevant insights into the possible behaviour
of eigenvalues and the relation between the Weyl asymptotics and P\'{o}lya-type inequalities in general. We observe, for instance, that
although the geometric non-periodicity condition necessary to derive the second term in the asymptotics given by~\eqref{remainder} fails
in the case of $\hs{n}$, as all geodesics are periodic, it is still possible to show that the
remainder term, which is of order $\bo\left(k^{1/n}\right)$, oscillates between an upper bound of this order and a constant lower bound,
thus making the difference in behaviour at the upper and lower levels explicit -- see
Theorems~\ref{Theo04} and~\ref{Theo05} below for the details. Furthermore, and perhaps more surprisingly, we obtain that hemispheres satisfy P\'{o}lya's
conjecture in the case of Neumann boundary conditions not only for the two-dimensional hemisphere, but in fact for all $n$ (Theorem~\ref{neumannpolya}) --
as in the Dirichlet case, this had already been proven in~\cite{bb} in two dimensions.
}

\txtb{In the case of spheres our bounds show that the remainder term oscillates between positive and negative values, but now these are
both of order $\bo\left(k^{1/n}\right)$ (see Theorem~\ref{theorem2-11} below). However, and in spite of these upper and lower bounds
for the remainder term, we still have that the $2-$sphere satisifies a Li-Yau inequality for the average of eigenvalues, while this
inequality is reversed in the case of $\mathbb{S}^{4}$ -- see Theorem~\ref{thmLiYau} below.}

In order to state our main results, we need the concept of a chain of eigenvalues corresponding to a multiple eigenvalue.
Let $\bl_{K}$ be the $K^{\rm th}$ distinct eigenvalue of $\hs{n}$ or $\sn{n}$, with corresponding multiplicity $m=m(K)$,
such that $\bl_{K} = \ld_{q+j-1}$ for some $q\in\N$ and $j=1,\dots, m$. We shall call these eigenvalues the $K-$chain
associated with $\bl_{K}$ (or just a chain, if there is no danger of confusion), and say that $\ld_{q}$ and $\ld_{q+m-1}$ are
the eigenvalues with the lowest and highest order
of the chain, respectively -- a more detailed definition is given in Section~\ref{notback}.

\subsection{Hemispheres\label{subsec:hemi}}

We begin by giving a characterisation of which of the lowest and highest order eigenvalues in each chain satisfy P\'{o}lya's 
inequalities and which do not. In particular, we obtain that for every $n$ greater than $2$ there exist infinite sequences of
eigenvalues of $\hs{n}$ which do not satisfy~\eqref{Polya_ineq}. Although in this case the boundary is not empty and the known
conditions for~\eqref{remainder} to hold are not satisfied~\cite{sava}, as all geodesics are periodic,
\txtb{this by itself is not enough to imply the failure of~\eqref{Polya_ineq} for any eigenvalue, as the case of $\hs{2}$ shows.}

\begin{thmx} \label{Theo02}
For the Dirichlet eigenvalues of the Laplace-Beltrami operator on the  $n$-dimensional hemisphere $\hs{n}$ we have the following:
\begin{enumerate}[{\rm 1.}]
\item\label{pt1} if $n=2$ all eigenvalues satisfy  P\'{o}lya's inequality.
\item if $n\geq 3$ the eigenvalue with the highest order of any chain does not satisfy P\'{o}lya's inequality; in particular,
$\lc_{1}(\hs{n}) < \cw\txtb{\left(\hs{n}\right)}\txtb{=(n!)^{2/n}}$.
\item\label{pt3} for all $n$ there exists $K_n\geq 2$ such that for all $K\geq K_n$ the lowest order eigenvalue of the 
corresponding $K-$chain satisfies P\'{o}lya's inequality; in particular $K_{n} = 2$ for $n\leq 8$, and
 $K_9=3$.
\end{enumerate}
\end{thmx}

\begin{remark}
 As mentioned above, item 1. in Theorem~\ref{Theo02}  was first proved in~\cite{bb}. Our proof is not very dissimilar, but our general approach
 allows us to obtain the other results in this theorem and further combination with P\'{o}lya's method yields the results for wedges given
 in Section~\ref{sec:introwedges} below.
\end{remark}
\begin{remark}
As a consequence of Theorem~\ref{Theo02},  there exists a function $j^{\ddag}(K)$ such that any
eigenvalue $\lambda_{q+j-1}$ in a given $K-$chain satisfies
P\'{o}lya's inequality if and only if $j\leq j^{\ddag}(K)$. For large $K$ we can approximate $j^{\ddag}(K)$ by an element of $\mathbb{Q}_{n-1}[K]$
(see Proposition \ref{middlePolya}).
\end{remark}
\begin{remark}
With the type of analysis developed here it is also possible to obtain estimates for eigenvalue averages, and we present these in
Sections~\ref{averages} and~\ref{Converse averages}.
\end{remark}

We shall now provide estimates that measure how far from the first term in the Weyl asymptotics the eigenvalues actually are.
Our first result in this direction shows that by introducing a constant correction term to~\eqref{Polya_ineq} it is possible
to obtain a sharp lower bound satisfied by all eigenvalues of $\hs{n}$ for all $n$.

\begin{thmx}\label{Theo04} The Dirichlet eigenvalues on the hemisphere $\hs{n}$ satisfy
\[
 \lambda_{k}\Big( \hs{n} \Big) \geq \cw\txtb{\left(\hs{n}\right)}\, k^{2/n} - \fr{(n-1)(n-2)}{6}.
\]
This is asymptotically sharp for the eigenvalue with the highest order on each chain, in the sense that over this subsequence
of eigenvalues
\[
\lambda_{k}\Big( \mathbb{S}^n_+ \Big) - \cw\txtb{\left(\hs{n}\right)}\, k^{2/n} \txtb{\to} - \fr{(n-1)(n-2)}{6}
\]
as $k$ goes to infinity, with the sequence being identically zero for $n=2$.
\end{thmx}


The second result of this type consists in bounding the same difference from above. However, the situation is not symmetric and the divergence from
the main term is no longer bounded and is, in fact, of order $\bo(k^{1/n})$.
\begin{thmx}\label{Theo05}
For  all $n\geq 2$ the  eigenvalues $\ld_k$ on $\hs{n}$ satisfy the
inequality
\[
 \lambda_k\left( \hs{n} \right) \leq \cw\txtb{\left(\hs{n}\right)} (k-1)^{2/n}+2\sqrt{\cw\txtb{\left(\hs{n}\right)}}(k-1)^{1/n} +n,
\]
with equality at $k=1$. 
Furthermore, \txtb{we have
\[\fr{\ld_k\left( \hs{n}  \right) -\cw\left(\hs{n}\right) (k-1)^{2/n}}{2\sqrt{\cw\left(\hs{n}\right)}(k-1)^{1/n}}\to 1\]
along the subsequence of lowest order eigenvalues of each chain.}
\end{thmx}

\begin{remark}
 The constant term in the upper bound is not optimal, and in Section~\ref{sec: TWO THREE TERM} we carry out a more careful study
 of the low dimensional cases, providing a bound for $\hs{2}$ that is attained for all eigenvalues of lowest order in each chain,
 and showing that the constant $n$ may be replaced by a term that decreases from $n$ to $3/2$ and $0$, for $\hs{3}$ and $\hs{4}$, respectively.
\end{remark}

\txtb{The case of Neumann eigenvalues may be treated using similar techniques. Here we single out one such result, as it is slightly unexpected,
particularly in view of the above results. More precisely, we show that in the Neumann case $\hs{n}$ satisfies P\'{o}lya's conjecture, not only
when $n$ is two, which again had already been shown to be the case in~\cite{bb}, but that, in fact, now it holds in any dimension.
\begin{thmx}\label{neumannpolya}
 The Neumann eigenvalues on the hemisphere $\hs{n}$ satisfy the inequalities
\[
\cw\left(\hs{n}\right)(k+1)^{2/n}-2\sqrt{\cw\left(\hs{n}\right)}(k+1)^{1/n} -\fr{(n-2)(n+5)}{6} \leq \mu_k\left( \hs{n} \right) \leq \cw\left(\hs{n}\right) k^{2/n},
\]
with equality in the right-hand side when $k=0$ and for all $k$ of the form $j(j+1)/2$ when $n=2$. Furthermore,
along the lowest order subsequence of each chain we have
\[
 \cw\left(\hs{n}\right) k^{2/n} - \mu_k\left( \hs{n} \right) \to \fr{(n-1)(n-2)}{6}
\]
as $k\to\infty$.
The left-hand side inequality is asymptotically
sharp, in the sense that
\[
\frac{\cw(\mathbb{S}^n_+)(k+1)^{2/n}-\mu_k(\mathbb{S}^n_+) }{2\sqrt{\cw\left(\hs{n}\right)}(k+1)^{1/n}}\to 1
\]
along the subsequence of highest order of each chain.
\end{thmx}
}

\subsection{Wedges\label{sec:introwedges}}

In~\cite{bb}, B\'{e}rard and Besson also considered P\'{o}lya's conjecture on wedges defined by
\begin{equation}\label{wedgedef}
\wdg{n}{\pi/p} = \sn{n}\cap \left(\R^{n-1}\times\left\{(x_{n},x_{n+1})\in \R\times \R_+:
x_{n} = x_{n+1}\tan\varphi, \varphi \in \left(-\fr{\pi}{2p},\fr{\pi}{2p}\right)\right\}\right), 
\end{equation}
for $p$ a positive integer. They then proved that eigenvalues of both $\wdg{2}{\pi/2}$ and $\wdg{2}{\pi/4}$
satisfy P\'{o}lya's conjecture. By noting that $p$ copies of $\wdg{n}{\pi/p}$ tile $\hs{n} = \wdg{n}{\pi}$, and
using an argument similar to that of P\'{o}lya's for planar tiling domains, it is possible to obtain that domains which
tile $\hs{n}$ also satisfy similar inequalities. As a direct consequence of Theorem~\ref{Theo02}~\ref{pt1} we
extend, for instance, B\'{e}rard and Besson's result to $\wdg{2}{\pi/p}$ in the Dirichlet case.
\begin{thmx}\label{cor:wedge}
 The Dirichlet eigenvalues of $\wdg{2}{\pi/p}$ satisfy P\'{o}lya's conjecture for all $p$ in $\N$.
\end{thmx}
\txtb{
\begin{remark}Although more precise, this result is in the same spirit of a recent result by the first and third authors,
showing that sufficiently thin domains whose isometric copies tile a larger domain, do satisfy P\'{o}lya's conjecture -- see~\cite{frsa}
for the details.
\end{remark}
}

For other applications of the results in Section~\ref{subsec:hemi} to wedges, see Section~\ref{sec:wedges}.

\subsection{Spheres\label{sec:resultsspheres}}

Since the spectrum of $\sn{n}$ consists of the union of the Dirichlet and Neumann spectra on $\hs{n}$, and taking the results for hemispheres
given in Theorem~\ref{Theo05} into consideration, it is to be expected that \txtb{now both upper and lower} bounds measuring the deviation of the spectrum of $\sn{n}$
from the first term in its Weyl asymptotics should include a second term of order $k^{1/n}$. This is indeed the case, as is
shown in the next result, where we obtain sharp estimates for this deviation.

\begin{thmx}  \label{theorem2-11}
  For all $n\geq 2$ and $k=0,1, \ldots$, the  eigenvalues of $\,\mathbb{S}^n$ satisfy the following inequalities
\begin{eqnarray*}
 \lc_{k}(\mathbb{S}^n) &\leq&
\cw\txtb{\left(\mathbb{S}^n\right)}~ k^{{2}/{n}}+\sqrt{\cw\txtb{\left(\mathbb{S}^n\right)}}k^{{1}/{n}},
\end{eqnarray*}
holding for all $k\geq 0$
where $\cw(\mathbb{S}^n)=\left(\fr{n!}{2}\right)^{{2}/{n}}$, and 
\[
\begin{array}{lll}
\lc_{k}(\mathbb{S}^n) & \geq & \cw\txtb{\left(\mathbb{S}^n\right)}~ (k+1)^{2/n} -\sqrt{\cw\txtb{\left(\mathbb{S}^n\right)}}~ (k+1)^{1/n}\eqskip
& & \hspace*{5mm}-\left(\fr{n+1}{2}\right)^2  -\fr{n^2}{\sqrt{\cw\txtb{\left(\mathbb{S}^n\right)}}~ (k+1)^{1/n}-n},
\end{array}
\]
holding for $k\geq 2n^n/n!-1$.
 Both inequalities are asymptotically sharp in the sense that, regarding the former, we have
 \[
 \fr{\lambda_k-\cw k^{2/n}}{\sqrt{\cw}k^{1/n}}\to 1
\]
along the subsequence consisting of the lowest order eigenvalues of each chain, while for the latter
\[
 \fr{\cw (k+1)^{2/n}-\lambda_k}{\sqrt{\cw}(k+1)^{1/n}}\to 1
\]
along the subsequence of the highest order eigenvalues of each chain.
\end{thmx}
\begin{remark}
  In the case of $\mathbb{S}^2$ we have obtained sharper lower bounds,
  which are attained by eigenvalues with the highest order on each chain -- see Proposition~\ref{theorem2-6}.
  Other lower bounds for $\mathbb{S}^n$ may be found in Proposition~\ref{alternativeestimates}.
\end{remark}
\txtb{
Finally, we show that a Li-Yau inequality holds for the average of the first $k+1$ eigenvalues in the case of $\mathbb{S}^{2}$, but that 
this inequality is reversed in the case of $\mathbb{S}^{4}$. In fact, in both cases we prove stronger results, particularly for $\mathbb{S}^2$
where a second oscillatory term gives equality whenever $k+1$ is a perfect square and is positive otherwise. 
\begin{thmx}\label{thmLiYau} For all $k\in\N$, the averages of the first $k+1$ eigenvalues of $\mathbb{S}^2$ and $\mathbb{S}^4$ satisfy the
inequalities
\[
 \fr{1}{2} C_{W,2}\left(\mathbb{S}^{2}\right) k + \fr{k}{2(k+1)}\left| \sin\left(\pi\sqrt{k+1}\right)\right| \leq \fr{1}{k+1}\, \dsum_{j=0}^{k} \lambda_{j}\left(\mathbb{S}^{2}\right)
 \leq \fr{1}{2} C_{W,2}\left(\mathbb{S}^{2}\right) k + \fr{1}{2},
\]
and
\[
 \fr{1}{k+1}\, \dsum_{j=0}^{k} \lambda_{j}\left(\mathbb{S}^{4}\right) \leq \fr{2}{3} C_{W,4}\left(\mathbb{S}^{4}\right) k^{1/2}-\fr{2}{\sqrt{3k}(k+1)^2},
\]
respectively.
\end{thmx}
}
\txtb{
\begin{remark}
 In the case of $\mathbb{S}^2$, a lower bound without the (non-negative) oscillating term may already be found in~\cite{illa}.
\end{remark}
\begin{remark}
It is not difficult to see that $\mathbb{S}^3$ cannot satisfy an inequality with respect to a term of the form
\[ \fr{3}{5} C_{W,3}\left(\mathbb{S}^{3}\right) k^{2/3}
\]
similar to either the lower bound for $\mathbb{S}^2$ or the upper bound for $\mathbb{S}^4$, as for $k=1,2$ we obtain that the left-hand side is larger
than this term, while the situation is reversed for $k=3$. We conjecture that $\mathbb{S}^{n}$ satisfies an inequality of the same type as that for the $4-$sphere
for all $n\geq5$, but haven't been able to prove this as our approach to the proof of these inequalities is based on the analysis of certain polynomials in
two variables that already gets to be quite involved in the case of $\mathbb{S}^{4}$.
\end{remark}
}

\subsection{A note about the proofs and the structure of the paper}\txtb{The main difficulty in dealing with the eigenvalues of $\mathbb{S}^{n}$ and $\mathbb{S}^{n}_{+}$
is related to translating the actual eigenvalues and their corresponding multiplicities (two numbers) into the order of each eigenvalue (one number), singling out individual
eigenvalues among each $K-$chain. Thus, and although both eigenvalues and multiplicities are known explicitly, the proofs require a careful and precise handling of the
expressions involved, which quickly runs into both combinatorial and algebraic difficulties. The setting we use to deal with these indexes is, in a way, analogous to a
two-coordinate system, where the first coordinate is $K$, and the second varies within the multiplicity of each chain.
}

\txtb{
Apart from induction, which appears naturally in several of the proofs, one other key ingredient is the usage of elementary symmetric functions which we use to rewrite
some of the polynomial inequalities appearing when dealing with P\'{o}lya-type inequalities. One intermediate result which plays a key role in many of the proofs throughout
the paper is Lemma~\ref{Mother}, which provides sharp lower and upper bounds for the rising factorial $K(K+1)\dots (K+n-1)$ in terms of quadratic polynomials related to
the relevant eigenvalue expressions. Finally, note that the intermediate inequalities used in the proofs are normally
quite sharp, in that the difference between both sides converges to zero as one of the parameters involved becomes large. Because of this, in some cases
we found that the shortest way of deriving a proof was to prove it analytically from a certain value onwards, and then handle the remaining (finite) number of cases
individually, sometimes using Mathematica -- some of these, and other computations are collected at the end in Appendix~\ref{sec:MATH}.
}

\txtb{
For ease of reference of the reader, we gather all the necessary notation and background in the next section. Sections~\ref{sec:THEOREM B} to~\ref{sec: TWO THREE TERM}
contain the proofs of the Theorems~\ref{Theo02},~\ref{Theo04} and~\ref{Theo05} for hemispheres with Dirichlet boundary conditions, together with some other related results. The 
case of Neumann boundary conditions and the proof of Theorem~\ref{neumannpolya} is addressed in Section~\ref{neumannproof}.
Section~\ref{sec:wedges} then uses the Dirichlet hemisphere results, together with an adapted tiling argument, to derive similar inequalities for wedges tiling hemispheres,
including the proof of Theorem~\ref{cor:wedge}.
Finally, Section~\ref{sec:spheres} is dedicated to spheres and the proof of Theorems~\ref{theorem2-11} and~\ref{thmLiYau}. The three appendices at the end of the paper
collect some auxiliary results that are used throughout.
}

\section{Notation and background\label{notback}}

The manifolds under study in the present paper, namely $\sn{n}$, $\hs{n}$ and $\wdg{n}{\pi/p}$, all have high, unbounded, eigenvalue multiplicities.
For our purposes it will thus be convenient to consider not only the corresponding eigenvalues $\lambda_{k}$ in increasing order and repeated according
to multiplicity, as defined in the Introduction, but also the corresponding sequence of distinct eigenvalues, also considered in increasing order,
and which we will denote by $\bl_{K}$ (with an upper-case $K$). Whenever necessary, an explicit notation for the manifold under consideration will be used, as in
$\lambda_k(\sn{n})$ or $\bl_K(\sn{n})$, for instance, and similarly to the corresponding Weyl constant $\cw(\sn{n})$. However, if this is clear from
the context we will omit such an explicit reference.

For the sphere $\sn{n}$, and given $K\in \N_{0}$, we define the quantity $\sigma(K)$ to be the sum of mul\-ti\-plicities from the zero-th
eigenvalue up to the $K$-th distinct eigenvalue $$ \sigma(K)=\dsum_{i=0}^{K}m(i) =m(0)+m(1)+\ldots + m(K).$$
Note that,  since the sphere is connected, $m(0)=\sigma(0)=1$, and for $K\geq 1$, $\sigma(K)= \sigma(K-1)+ m(K)$. We make the convention
that $\sigma(-1)=0$.

Each of these distinct eigenvalues $\bl_K$ defines the $K-$chain of length $m(K)$ of non--distinct eigenvalues $\lambda_k$
\[
 \lambda_{\sigma(K-1)} = \lambda_{\sigma(K-1)+1} = \dots = \lambda_{\sigma(K-1)+m(K)-1} \left( = \bl_{K}\right),
\]
and we denote by $k_-=k_-(K)$ and $k_+=k_+(K)$ the lowest and the highest orders of the eigenvalues of the 
$K-$chain, respectively, that is, $ k_-=\sigma(K-1)$, and  $ k_+=\sigma(K)-1=k_-+m(K)-1$.
If $K=0$, $k_-=k_+=0$.

In the case of manifolds with boundary such as $\hs{n}$ or $\wdg{n}{\pi/p}$, and with Dirichlet boundary conditions, we proceed in a similar way as above,
except that now $K\in \N$.  We thus have
\begin{gather*}
 \sigma(K)=\dsum_{i=1}^{K}m(i) = m(1)+\ldots + m(K),\\
\lambda_{\sigma(K-1)+1} = \lambda_{\sigma(K-1)+2} = \cdots = \lambda_{\sigma(K-1)+m(K)} \left( = \bl_{K}\right),
\end{gather*}
with the convention that $\sigma(0)=0$. The corresponding lowest and highest orders in each chain are now given by $k_{-} = \sigma(K-1)+1$ and $k_{+}=\sigma(K)=k_{-}+m(K)-1$.

Two straightforward observations are as follows. In both eigenvalue problems, if one eigenvalue in a $K-$chain satisfies P\'{o}lya's inequality, then the
corresponding eigenvalue of the lowest order in that $K-$chain must also satisfy P\'{o}lya's inequality, that is
\begin{equation}\label{alphaPolya}
 \bl_{K}\geq \cw\, k_-^{2/n}.
\end{equation}
  All eigenvalues of the $K-$chain satisfy P\' {o}lya's inequality if and only if it is satisfied for the highest order $k_+$, that is
\begin{equation}\label{StrongPolya}
 \bl_{K}\geq \cw\,k_+^{2/n}.
\end{equation}

\subsection{  Eigenvalues of $\mathbb{S}^n$\label{sec:eigenvaluesofspheres}}
The Weyl  constant of  $\mathbb{S}^n$  is given by
$$ \cw\txtb{\left(\mathbb{S}^n\right)}=\left( \frac{n!}{2}\right)^{{2}/{n}}.$$
It is well known (see e.g  \cite{BGM}) that the distinct closed eigenvalues  are given by 
$$\bl_K =K(K+n-1), \quad K=0,1,2,\dots, $$
where the multiplicity of $\bl_K$ equals the dimension of
the space of homogeneous  harmonic polynomials of degree $K$,
that is,
\begin{gather*}
m(K)=\begin{pmatrix}
n+K\\
n
\end{pmatrix}
- \begin{pmatrix}
 n+K-2\\
 n
\end{pmatrix},
\end{gather*}
where ${\ds \binom{m}{k}}$ is considered to be zero if $m<k$.
We shall now compute the sum of the multiplicities of the first $K$ eigenvalues.
We recall the notion of  $K$ to the $n$ rising factorial 
\begin{equation} \label{RisibgFactorial}
K^{\overline{n}}=K(K+1)\ldots (K+n-1).
\end{equation}
Note that $K^{\overline{1}}=K$ and, by convention $K^{\overline{0}}=1$. We use the following equivalent notations
$\Gamma(n+1)=n!=1^{\overline{n}}$, where $\Gamma(x)$ is the Gamma function.

\begin{lemma} \label{a-4}
Let $n\in\N$, $n\geq 2$, and $K\in\N_{0}$. Then
\begin{eqnarray}
m(K) &=& \frac{(K+1)^{\overline{n-1}}}{(n-1)!}\left(\frac{2K+n-1}{K+n-1}\right)=
 \frac{K^{\overline{n-1}}}{(n-1)!}\left(\frac{2K+n-1}{K}\right)\nonumber\\
\label{identity}
\sigma(K)&=&\frac{\Gamma(n+K)}{\Gamma(K+1)n!}(n+2K)
~=~ \frac{(K+1)^{\overline{n-1}}}{n!}\left(n+2K\right).
\end{eqnarray}
\end{lemma}

\begin{proof}
It is straightforward to check that the identities with the binomial terms and those with the rising factorial are equivalent.
It thus remains only to prove the first identity for $\sigma(K)$, which we do by induction. 
From
\begin{eqnarray}\label{a-4-1}
\sigma(K) &=& \dsum\limits_{\ell=0}^{K}\left[\begin{pmatrix}
n+\ell\\
n
\end{pmatrix}
- \begin{pmatrix}
 n+\ell-2\\
 n
\end{pmatrix}\right],
\end{eqnarray}
we see that when $K=0$~\eqref{identity} holds.
Assume now that \eqref{identity} holds for some $K$. We have
\begin{eqnarray*}
\sigma(K+1)&=&\frac{\Gamma(n+K)}{\Gamma(K+1)\Gamma(n+1)}(n+2K)+
\begin{pmatrix}
n+K+1\\
n
\end{pmatrix}
- \begin{pmatrix}
 n+K-1\\
 n
\end{pmatrix}\eqskip
& = & \frac{\Gamma(n+K+1)}{\Gamma(K+2)n!}\left[n+2(K+1)\right].
\end{eqnarray*}
\end{proof}

We extend $\sigma$ as a smooth function defined over all reals, given by the same formula, namely,
$$ \sigma(x)= \fr{1}{n!}[(x+1)(x+2)\ldots (x+n-1)](n+2x).$$
The eigenvalues on each $K-$chain, $\lambda_k$,  are ordered by the integers $ k_-\leq k\leq k_+$ where we now write
\begin{eqnarray*}
k_- &=& \sigma(K-1)=   \frac{K^{\overline{n-1}}}{n!}(n+2(K-1)),\\
k_+ &=& \sigma(K)-1=  \frac{(K+1)^{\overline{n-1}}}{n!}(n+2K) -1,
\end{eqnarray*}
and so,
\begin{eqnarray*}
\begin{array}{ll}
\cw k_-^{2/n}=\left(K^{\overline{n-1}}(K+\fr{n}{2}-1)\right)^{2/n},&\quad~
\cw k_+^{2/n}=\left((K+1)^{\overline{n-1}}(K+\fr{n}{2})-\fr{n!}{2}\right)^{2/n}.
\end{array}
\end{eqnarray*}

\subsection{  Eigenvalues of $\mathbb{S}^n_+$\label{sec:eigenvaluesofhemispheres}  }
We now consider the Dirichlet eigenvalues of the unit hemisphere  
$\mathbb{S}^n_+$ of $\mathbb{R}^{n+1}$, $n\geq 2$. 
Note that $|\mathbb{S}^n_+|=\frac{n+1}{2}\omega_{n+1}$.
Moreover, from the well--known recursion formula $\omega_n=\frac{2\pi}{n}\omega_{n-2}$, and
$\omega_2=\pi$, $\omega_1= 2$, we can prove by induction on $n$ that
$ ({n+1})\omega_n\omega_{n+1}/2={2^{n}}\pi^n/n!$, 
and so $\omega_n|\mathbb{S}^n_+|={2^n}\pi^n/n!$, yielding
\begin{equation}\label{PPUU}
  \cw\txtb{\left(\mathbb{S}^n_+\right)} =(n!)^{2/n}.
\end{equation}
The distinct eigenvalues of $\mathbb{S}^n_{+}$ are given by (cf. \cite{bb}, \cite{bsj} )
\begin{equation}\label{distinct}
 \bl_K=K(K+n-1),\quad K=1,2\ldots,
\end{equation}
with multiplicity
\begin{equation}\label{multip}
m(K)=\binom{n+K-2}{n-1}=\binom{n+K-2}{K-1}=\frac{K^{\overline{n-1}}}{(n-1)!}.
\end{equation}
Recall the parallel summation $\dsum_{ s=0}^{k-1}\binom{r+s}{s}=\binom{r+k}{k-1}$, valid for all 
$r\in \mathbb{N}_0$ (cf.\ \cite{A}, \cite{bsj}).
Letting $r=n-1$ we get
$$\sigma(K)=\dsum_{s=1}^{K}\binom{n+s-2}{s-1}=\dsum_{s=0}^{K-1}\binom{n-1+s}{s}
= \binom{n+K-1}{K-1}=\fr{\Gamma(n+K)}{\Gamma(K)\Gamma(n+1)}= \fr{K^{\overline{n}}}{n!}.$$

The $K-$chain of eigenvalues $\lambda_k$ defined by $K(K+n-1)$ corresponds to the integers $k$ such that $k_-\leq k\leq k_+$, where
$$ 
k_-=\sigma(K-1)+1=\fr{(K-1)^{\overline{n}}+n!}{n!}, \quad\quad k_+=\sigma(K)=\fr{K^{\overline{n}}}{n!}.
$$
Therefore,
\begin{eqnarray*}
\cw\,k_-^{2/n} 
=\left((K-1)^{\overline{n}}+n!\right)^{2/n},\quad\quad
 \cw k_+^{2/n}= (K^{\overline{n}})^{2/n}.
\end{eqnarray*}

\txtb{We proceed in a similar way in the Neumann case, denoting the eigenvalues of $\hs{n}$, repeated according to multiplicities, by
\[
0=\mu_0<\mu_1\leq \mu_2\leq \ldots
\]  
while the distinct eigenvalue are, as in the Dirichlet case, given by $\bar{\lambda}_K=K(K+n-1)$, but now for $K\in\N_{0}$. The corresponding multiplicities
are now
\[
m'(K)=\binom{n+K-1}{n-1}=\frac{(K+1)^{\overline{n-1}}}{(n-1)!}
\]
Then 
\[\sigma'(K)=m'(0)+m'(1)+\ldots m'(K)=\sum_{s=0}^K
\binom{n-1+s}{s}=\binom{n+K}{K}= \frac{(K+1)^{\bar{n}}}{n!},\]
the corresponding $K$-chain is given by the eigenvalues 
$\mu_k$ where 
\[
k=\sigma'(K-1)+r=\frac{K^{\overline{n}}}{n!} +r, \quad r=0,\ldots m'(K)-1,
\]
and so the lowest and highest orders are given respectively by
\[\begin{array}{l} 
k_{-}' =k_{-}'(K)=\sigma'(K-1)= \fr{K^{\bar{n}}}{n!}
 \eqskip
k_{+}'= k_{+}'(K)= \sigma'(K)-1=\fr{(K+1)^{\bar{n}}}{n!}-1,\end{array}\]
where we use the convention $\sigma'(-1)=0$.
As a consequence,
\[
\begin{array}{rll}
\cw\,k_{-}'(K)^{2/n} & = & \left[K^{\bar{n}}\right]^{2/n}\eqskip
\cw\,k_{+}'(K)^{2/n} & = & \left[(K+1)^{\bar{n}}-n!\right]^{2/n}\eqskip
\cw\,\left[k_{+}'(K)+1\right]^{2/n} & = & \left[(K+1)^{\bar{n}}\right]^{2/n}.
\end{array}
\]
}

We now introduce the following two functionals on distinct eigenvalues $\bl_{K}$ of $\mathbb{S}^n_+$, 
\begin{gather*}
\mathcal{R}(K)=\mathcal{R}_n(K):= \frac{\cw\,\sigma(K)^{{2}/{n}}}{\bl_{K}}=
\frac{\left(\frac{\Gamma(n+K)}{\Gamma(K)}\right)^{2/n}}{K(K+n-1)}=\frac{(K^{\overline{n}})^{2/n}}{K(K+n-1)},\\
\Phi(K)=\Phi_n(K) :=\cw (\sigma(K))^{{2}/{n}}-\bl_{K}=\left(\frac{\Gamma(n+K)}{\Gamma(K)}\right)^{2/n}-K(K+n-1),
\end{gather*}
where we shall omit the index $n$ whenever this is clear.
With this notation we may reformulate inequality~(\ref{StrongPolya})
as $\mathcal{R}(K)\leq 1$ (or $\Phi(K)\leq 0$), respectively. 
By Weyl's asymptotic formula taking the subsequence defined by $k=\sigma(K)$, $K=1,2, \ldots,$ we have $\mathcal{R}(K)\to 1$ when $K\to +\infty$.

\subsection{Elementary symmetric functions}\label{sec:symmetricfunctions}
P\'{o}lya's inequalities for spheres and hemispheres are equivalent to certain polynomial inequalities which may be stated in terms of
elementary symmetric functions. Following~\cite{macd}, for instance, for any natural number $n$ we define the elementary symmetric functions
$\sigma_j:\R^{n}\to\R$ by
\begin{eqnarray}\label{elementarysym}
 &&\sigma_0(x_1, \ldots, x_n)=1, \quad
 \sigma_j(x_1, \ldots, x_n)=\dsum_{i_1<i_2<\ldots<i_j}x_{i_1}\ldots x_{i_j},
\end{eqnarray}
and the related constants $s_j(n)$, $j=0,1, \ldots, n$,
\begin{equation}
s_j(n)=\sigma_j(1,2, \ldots,n). \label{constantsym} \end{equation}
These functions appear in the factorization of monic polynomials such as
\[
(K+x_1)(K+x_2)\ldots(K+x_n)= \dsum_{0\leq j\leq n}\sigma_j(x_1, \ldots, x_n)K^{n-j},
\]
which, in the particular case of $x_{i}=i-1$, yields
\begin{equation}\label{risingpolynomial}
K^{\overline{n}}=\dsum_{0\leq j\leq n-1}s_j(n-1)K^{n-j}.
\end{equation}

\section{Proof of Theorem~\ref{Theo02}\label{sec:THEOREM B}}

\txtb{
Our first results consist in looking for which $K\geq 1$ and $n\geq 2$, P\'{o}lya's inequality
holds for the lowest order eigenvalue of a $K-$chain,
that is the weakest inequality holds
\begin{equation}\label{lowerorder}
K(K+n-1)=\bl_K ~\geq~ \cw\, {k_-}^{{2}/{n}}= ((K-1)^{\overline{n}}+ n!)^{{2}/{n}},
\end{equation}
and  when it holds for the highest order eigenvalue, that is the strongest inequality holds
\begin{equation}\label{largerorder}
K(K+n-1)=\bl_K ~\geq~ \cw\, k_+^{{2}/{n}}=\left(K^{\overline{n}}\right)^{{2}/{n}}.
\end{equation}
}

The first eigenvalue ($k=K=1$) satisfies P\'{o}lya's inequality if and only if $n^{n/2}\geq n!$, which holds for $n=1,2$,
while from Stirling's lower bound~\eqref{StirlingBounds} we see it cannot hold for $n\geq 3$. This proves the last claim in item 2..

If $K\geq 2$,  all eigenvalues of the $K-$chain satisfy P\'{o}lya's inequality  if and only if~\eqref{largerorder} is satisfied.
This inequality holds for $n=2$,  and statement 1.\ is proved. From the left hand--side inequality of Lemma~\ref{Mother} we conclude
that (\ref{largerorder}) never holds for  any $n\geq 3$. This completes the proof of item 2.
Furthermore, some eigenvalue of the $K-$chain satisfies P\'{o}lya's inequality if  (\ref{lowerorder}) holds, or equivalently  
\begin{equation}\label{others}
Q_{n}(K):= (K(K+n-1))^{n}-\La{(} (K-1)^{\overline{n}}+n!\La{)}^2\geq 0.
\end{equation}
For $n=2,\dots, 8$ inequality~(\ref{others}) holds for any  $K\geq 2$.
This may be verified by determining the polynomials $Q_n(K)$ and computing their derivatives with respect to $K$, since for all $K\geq 2$ these 
are positive polynomials and $Q_n(2)>0$ for all $n\leq 8$. However, this will no longer be the case for $n=9$.

Next we complete the proof of item 3. for any $n\geq 3$ by proving the following lemma.
\begin{lemma} \label{QnK} 
If $n\geq 3$, 
the polynomial  $Q_n(K)$ is of degree  $2n-1$ with principal coefficient given by
$2n$. In particular, we can find $K_n\geq 2$ such that for all $K\geq K_n$ 
the lowest order eigenvalue $\lambda_{k_-}$ of the $K-$chain satisfies P\'{o}lya's inequality.
\end{lemma}

\begin{proof}  We have

\[
 \begin{array}{lll}
\left[K(K+n-1)\right]^{n} & = & \left[(n-1)K + K^2\right]^{n}\eqskip
 & = & K^{2n}+ \dsum_{l=1}^{n}\binom{n}{l}(n-1)^{l} K^{2n-l}.
 \end{array}
\]
Set $\sigma_l=\sigma_l(-1,1,2,\ldots, n-2)$, for $l=0, 1\ldots, n-1$, (see (\ref{elementarysym})).
We have the following identity,
\[(K-1)K(K+1)\ldots(K+n-2)=K[(K-1)(K+1)\ldots(K+n-2)]=K\left(\sum_{l=0}^{n-1}\sigma_lK^{n-1-l}\right),\]
and so
\begin{gather*}
\begin{array}{lll}
\La{[}(K-1)K(K+1)\ldots(K+n-2)+n! \La{]}^2 & = & \La{[}(K-1)K(K+1)\ldots(K+n-2)\La{]}^2\eqskip
& & \hspace*{5mm} +(n!)^2+ 2 n! (K-1)K(K+1)\ldots(K+n-2)\eqskip
& = & K^2\left( \dsum_{l=0}^{n-1}\sigma_{l}K^{n-1-l}\right)^2+(n!)^2+
 2n! K\dsum_{l=0}^{n-1} \sigma_{l}K^{n-1-l}\eqskip
& = & K^2\left(\sigma_{0}K^{n-1}+\sigma_{1}K^{n-2}+ \ldots +\sigma_{n-2}K+\sigma_{n-1}\right)^2\eqskip
& & \hspace*{5mm} + 2n!K \dsum_{l=0}^{n-1}\sigma_{l}K^{n-1-l}+ (n!)^2.
\end{array}
\end{gather*}
Note that if $0\leq i<j\leq n-1$, then $1\leq i+j\leq 2n-3$. Hence,
\begin{gather*}
\begin{array}{lll}
\La{[}(K-1)K(K+1)\ldots(K+n-2)+n! \La{]}^2 &=&  K^2\Big(\sigma_0^2K^{2(n-1)}+\sigma_1^2K^{2(n-2)}+\ldots+ \sigma_{n-2}^2K^{2}+\sigma_{n-1}^2\eqskip
& & \hspace*{5mm}+2\dsum_{j=0}^{n-1}\dsum_{i=0}^{j-1} \sigma_iK^{n-i-1}\sigma_jK^{n-j-1}\Big)\eqskip
& & \hspace*{10mm} + 2n!K \dsum_{l=0}^{n-1}\sigma_{l}K^{n-1-l}+ (n!)^2\eqskip
& = & K^{2n}+ \sigma_1^2K^{2n-2}+ \sigma_2^2K^{2n-4}+\ldots +\sigma_{n-2}^2K^4\eqskip
& & \hspace*{5mm} +\sigma_{n-1}^2K^2+ 2\dsum_{j=0}^{n-1}\dsum_{i=0}^{j-1}\sigma_i\sigma_j K^{2n-(i+j)}\eqskip
& & \hspace*{10mm} + 2n! \La{(}K^{n} +\sigma_{1}K^{n-1}+ \ldots +\sigma_{n-2}K^2+\sigma_{n-1}K\La{)} +(n!)^2\eqskip
& = & K^{2n}+ \dsum_{s=1}^{n-1} \sigma_s^2 K^{2n-2s}+
 \dsum_{l=1}^{2n-3}\LA{(}\sum_{\tiny \begin{array}{c}
0\leq i<j\leq n-1\\ i+j=l \end{array}}2\sigma_i\sigma_j \LA{)} K^{2n-l}\eqskip
& & \vspace*{5mm}\hspace*{5mm} +  2 n! \La{(}K^{n} +\sigma_{1}K^{n-1}+ \ldots +\sigma_{n-1}K\La{)}
+(n!)^2.
\end{array}
\end{gather*}
Therefore,~\eqref{others} holds if and only if 
\begin{eqnarray*}
Q_n(K) &=& \sum_{l=1}^n \binom{n}{l}(n-1)^{l} K^{2n-l}
- \sum_{s=1}^{n-1} \sigma_s^2 K^{2n-2s}
 - \sum_{l=1}^{2n-3}\LA{(}\!\!\!\!\sum_{\tiny \begin{array}{c}
0\leq i<j\leq n-1\\ i+j=l \end{array}}\!\!\!\!\!\!\!\!2\sigma_i\sigma_j \LA{)} K^{2n-l}\eqskip
&& \hspace*{5mm} -2 n! \La{(}K^{n} +\sigma_{1}K^{n-1}+ \ldots +\sigma_{n-2}K^2+\sigma_{n-1}K\La{)}
-( n!)^2\quad\geq \quad 0.
\end{eqnarray*}
Using that $\sigma_1(1, 2, \ldots, n-2)=s_1(n-2)$ (see (\ref{elementarysym}), (\ref{constantsym}))
and 
Appendix \ref{sec:Combinatorics} we have
$$\sigma_1(-1,1,2,\ldots, n-2)=-1+\sigma_1(1,2,\ldots, n-2)=-1+\frac{(n-1)(n-2)}{2}=
\frac{n(n-3)}{2}.$$
Hence, the polynomial  $Q_n(K)$ is of degree  $2n-1$ with principal coefficient given by
$$ \binom{n}{1}(n-1) -2\sigma_0\sigma_1
=n(n-1)-n(n-3)=2n,$$
and the lemma follows.
\end{proof}

The next result is a consequence of Theorem~\ref{Theo02}~\ref{pt3}, but we shall now give a direct proof.
\begin{corollary}\label{2} 
The lowest order eigenvalue of the $2-$chain satisfies P\' {o}lya's inequality if and only if $n\leq 8$.
\end{corollary}

\begin{proof}
For any $n\geq 2$ if we make $K=2$ inequality~(\ref{others}) becomes equivalent to
\[
 \phi(n):=(2n!)^2- (n+1)^n 2^{n}\leq 0.
\]
For $n\leq 8$ we have $\phi(n)<0$,  but $\phi(9)>0$. We will prove by induction that $\phi(n)>0$ holds for all $n$ greater than $9$.
By the induction hypothesis we have
$$ ((n+1)!)^2 =(n!)^2 (n+1)^2> (n+1)^n 2^{n-2} (n+1)^2 =(n+1)^{n+2}2^{n-2}.$$
If we show that  $\xi(n)\geq 2$, where
$\xi(n)={(n+1)^{n+2}}/{(n+2)^{n+1}},$
we finish the proof, and we see this holds since $\xi'(n)>0$ and  $\xi(9)>3$.
\end{proof}

\section{P\'{o}lya's inequality for eigenvalue averages on $\mathbb{S}^n_+$\label{averages}} 

In the Euclidean case the point of departure for the Li and Yau estimates for single eigenvalues are their results for the sum of the first $n$ eigenvalues
given by~\eqref{euclid_liyau}. Here we may
also derive estimates of a similar type, and we shall do so for both the first $n$ eigenvalues and also within each chain of eigenvalues.
If $n=2$, by Theorem~\ref{Theo02}~\ref{pt1} all  eigenvalues satisfy P\'{o}lya. Hence we will assume $n\geq 3$, except where specific values are indicated.

\begin{thm}\label{Theo02sums}
For the Dirichlet eigenvalues of the Laplace-Beltrami operator on the  $n$-dimensional hemisphere $\hs{n}$ we have the following:
\begin{enumerate}[{\rm 1.}]
\item for all $n$ there exists $K'_n\geq 2$ such that for all $K\geq K'_n$ the eigenvalue of the corresponding $K-$chain satisfies
\[\bl_{K}> \fr{\cw}{m}
\sum_{k=q}^{q+m-1} k^{2/n}.
\]
In particular $K'_n=2$, for $n=3,4, 5$, $K'_6=3$ and  $K'_{10}=10$.
\item  for all $n$ the total average sequence satisfies
\[
{\ds \lim_{k\to+\infty}}\fr{1}{k}\sum_{j=1}^{k}(\lambda_j -\cw j^{2/n}) = +\infty.
\]
\end{enumerate}
\end{thm}
\begin{remark}
For $n=2$ item 1. reduces to $K(K+1)>m(K)+1=K+1$, which holds for all $K\geq 2$,
while for $K=1$ we have equality.  Moerover,  choosing the largest order of the $K-$chain, $k=k_+(K)=\sigma(K)=K(K+1)/2$, $K\geq 1$,  we can compute the full sum explicitly to obtain 
\[\fr{1}{k}\dsum_{j=1}^{k}(\lambda_j -\cw j^{2/n})=\fr{2}{3}(K-1), \quad
\dsum_{j=1}^{k}\lambda_j = k^2 + k \fr{\sqrt{8k+1}}{3}. \]
For a generic order $k$ in the $(K+1)-$chain, \txtb{this sum equals a more involved expression of a similar form, namely, $\dsum_{j=1}^{k}\lambda_{j} = k^2 + kQ(k)$, where $Q(k)$
is greater than or equal to one.}
\end{remark}

\begin{remark}\label{sumLiYau}
From  item 2. it follows that \txtb{there exists a positive value} $L$ such that
 \[\dsum_{j=1}^{k}\lambda_j\geq \frac{n}{n+2}\cw k^{\frac{2}{n}+1}+ Lk\] holds for $k$  sufficiently large -- compare with the
Li and Yau inequality  $\dsum_{j=1}^{k}\lambda_j\geq \frac{n}{n+2}\cw k^{\frac{2}{n}+1}$ (cf.\ \cite{ly}) for Dirichlet eigenvalues on bounded domains in
 $\mathbb{R}^n$.
\end{remark}

Let $k\in [k_-, k_+]$ be an integer of a $K-$chain, that is, 
$k=k_j=\sigma(K-1)+j$ where $j=1,2,\ldots, m(K)$. For each positive integer (or positive real, when appropriate)  $j$  we define
\begin{equation}\label{Theta(K)}
 \Upsilon_j(K) := (K-1)^{\overline{n}}+jn!,
\end{equation}
and consider the following quantities defined by the elementary symmetric functions  (\ref{constantsym})
\begin{equation}\label{allsigmas}
\hat{s}_l := s_l(n-1), \quad\mbox{if~} 0\leq l\leq n-1, \quad\quad \hat{s}_{n}:= n!,
\end{equation}
where the value for $\hat{s}_{n}$ is defined to be $n!$ for the sake of simplicity. Let
\begin{equation}\label{hatx0}
 {y}_j(K)~=~ \frac{\hat{s}_1}{K-1}+\frac{\hat{s}_2}{(K-1)^2}+\ldots
\frac{\hat{s}_{n-1}}{(K-1)^{n-1}} +\frac{j\hat{s}_{n}}{(K-1)^n},\quad \forall K\geq 2.
\end{equation}
The following identities hold (see (\ref{risingpolynomial}))
\begin{eqnarray}
\Upsilon_j (K)& = & (K-1)^n + \hat{s}_1(K-1)^{n-1}+  \hat{s}_2(K-1)^{n-2}+
\ldots + \hat{s}_{n-1}(K-1) + j\hat{s}_n. \label{UpsilonALT2}\\
&=&(K-1)^n(1+{y}_j(K)) \label{UpsilonALT1} 
\end{eqnarray}

Then, for $K\geq 2$
\begin{gather}\label{cwkj}
\cw k_j^{2/n} = (\Upsilon_j(K))^{2/n}
= (K-1)^2 \left(1+ {y}_j(K)\right)^{2/n }.
\end{gather}

We will consider a mean--value function defined on each $K-$chain by
$$P_{m}(K):=\frac{1}{m(K)}\sum_{k=k_-(K)}^{k_+(K)}\left(\lambda_k -\cw k^{2/n}\right)
= K(K+n-1)-\frac{1}{m(K)}\sum_{j=1}^{m(K)}\cw k_j^{2/n},$$
and  the following polynomial function
 $$ Q(x):= \sum_{l=1}^{n-3} \left({\frac{2}{n}\hat{s}_{l+2}+\hat{s}_{l+1}}\right){x^l}+{2(n-1)!}{x^{n-2}},$$
where in case $n=3$ the summation term is assumed to be zero. 

\begin{proposition}\label{averagePolya} There exists $K'_n\geq 2$ such that $\dsum_{k=k_-(K)}^{k_+(K)} (\lambda_k-\cw k^{2/n})\geq 0$
for all $K\geq K' _n$.  Moreover, for all $K\geq 2$, $P_{m}(K)\geq (K-1)+ T'_n$, 
where $T'_n=(n-\fr{2}{n}\hat{s}_2-\hat{s}_1)-Q(1)$. 
We may take $K'_n\geq 2$ satisfying  $ K'_n\geq -T'_n+1$.
\end{proposition}
\begin{proof}
We must have $K'_n\geq 2$, since P\'{o}lya's inequality is not satisfied for $K=1$. Let
$$Pol_j(K):=  K(K+n-1)-\cw\, k_{j}^{2/n} $$
Using (\ref{cwkj}) and  the upper bound in  Lemma~\ref{Taylor}, 
$(1+ y)^{2/n}\leq 1+ \frac{2}{n}y$,  we get for $K\geq 2$
\begin{eqnarray*}
Pol_j(K) 
&\geq&  (K-1)^2+(n+1)(K-1)  +n
-(K-1)^2\left(1+\fr{2}{n}\left[\dsum_{l=1}^{n-1}\fr{\hat{s}_l}{(K-1)^l}
+\fr{jn!}{(K-1)^{n}}\right]\right)\\
&=& [n+1-\fr{2}{n}\hat{s}_1](K-1)  +[n-\fr{2}{n}\hat{s}_2]
-\fr{2}{n}\left(\dsum_{l=3}^{n-1}\fr{\hat{s}_l}{(K-1)^{l-2}}\right)-\fr{2}{n}\fr{jn!}{(K-1)^{n-2}}\\
&=& 2(K-1)+ [n-\fr{2}{n}\hat{s}_2]-\fr{2}{n}\left(\dsum_{l=3}^{n-1}\fr{\hat{s}_l}{(K-1)^{l-2}}\right)-\fr{2}{n}\fr{jn!}{(K-1)^{n-2}},
\end{eqnarray*}
where if $n=3$ the summation term is zero.
Since $K^{\overline{n-1}}=(K-1)^{\overline{n}}/(K-1)$, from (\ref{risingpolynomial})  we have 
$m(K)=\fr{1}{(n-1)!}\dsum_{i=0}^{n-1}\hat{s}_i(K-1)^{n-(i+1)}$.
Therefore,
\begin{eqnarray*}
\fr{1}{m(K)}\dsum_{j=1}^{m(K)}j=\fr{m(K)+1}{2}=\fr{1}{2}+\fr{1}{2(n-1)!}\dsum_{i=0}^{n-1}
\hat{s}_i(K-1)^{n-{(i+1)}},
\end{eqnarray*}
and using that   $\hat{s}_{n-1}=(n-1)!$ we thus obtain the following estimate for the average
\begin{eqnarray*}
\fr{1}{m(K)}\dsum_{j=1}^{m(K)}Pol_j(K) &\geq& 
(K-1) + [n-\fr{2}{n}\hat{s}_2-\hat{s}_1]
 -\dsum_{l=1}^{n-3} \fr{\frac{2}{n}\hat{s}_{l+2}+\hat{s}_{l+1}}{(K-1)^l}-\fr{2(n-1)!}{(K-1)^{n-2}}.
\end{eqnarray*}

Then, for $K\geq 2$,
$P_{m}(K)\geq K-1 + [n-\frac{2}{n}\hat{s}_2-\hat{s}_1]-Q((K-1)^{-1})$.
Let $x_0\in (0,1]$ such that for all $0<x\leq x_0$, $Q(x)\leq x^{-1} +n-\frac{2}{n}\hat{s}_2-\hat{s}_1 $. Then $T'_n= n-\frac{2}{n}\hat{s}_2-\hat{s}_1-Q(1)$, and 
for $(K-1)\geq x_0^{-1}$, $P_{m}(K)\geq 0$. The  lower bound for $K'_n$ is obtained by
taking $K'_n\geq  1-T'_n$.
\end{proof}

For small $n$ we can improve the choice of $K'_n$ by direct inspection, instead of using the rough upper bound for $(1+y)^{2/n}$
as was done above. This allows us to obtain $K'_n=2$ for $n=3,4,5$, $K'_6=3$, and $K'_{10}=10$. These
considerations together with the above proposition prove item 1. of Theorem~\ref{Theo02sums}.

If $K=1$, $P_{m}(K)=n-(n!)^{2/n}<0$. Given $K\geq 2$, let $j^{\ddag}(K)\in [1,  m(K)]$ 
such that $Pol_{j^{\ddag}(K)}(K)=0$. Then  $Pol_j(K)<0$ for all $j^{\ddag}(K)<j\leq m(K)$
and  $Pol_j(K)\geq 0$, $\forall j\leq j^{\ddag}(K)$. This $j^{\ddag}(K)$ exists  by  Lemmas \ref{Mother} and \ref{QnK}, and it is smooth for $K\in (1, +\infty)$
by the implicit function theorem. In Proposition \ref{middlePolya} we will approximate $j^{\ddag}(K)$ by a polynomial function $j^{\dag}(K)$ with rational coefficients.
For each $n\geq 2$ we define 
$$j^*(K)=j_{n-1}^*(K):=H_{n-1}(K-1)=\dsum_{l=0}^{n-1}b^*_l(n-1)(K-1)^{l}= m(K),$$ 
 where $b^*_{l}(n-1):=\fr{\hat{s}_{n-l-1}}{(n-1)!}$, $l=0, ...n-1$, and take  $j(K):=\dsum_{l=0}^{n-1}b_l(K-1)^l$ to be any  polynomial of degree at most $n-1$ (not
 necessarily integer-valued), such that
$1\leq j(K)\leq j_{n-1}^*(K)$ for $K$ large.
Therefore, either $b_{n-1}< b^*_{n-1}(n-1)=\fr{1}{(n-1)!}$, or $b_{n-1}=b^*_{n-1}(n-1)$, and $b_{n-2}<b^*_{n-2}(n-1)=
\fr{\hat{s}_1}{(n-1)!}=\fr{n}{2(n-2)!}$, and so on. 

\begin{lemma}\label{F}  We have
\begin{eqnarray*}\begin{array}{ll}
Pol_{j(K)}(K) &= K(K+n-1) -\left( (K-1)^{\overline{n}} +n!j(K)\right)^{2/n}\eqskip
&=F_1(K-1)+F_0+ \fr{F_{-1}}{(K-1)} +\so\left(\fr{1}{K-1}\right),
\end{array}
\end{eqnarray*}
where
\begin{gather*} \left\{\begin{array}{l}
F_1:=(n+1)-\fr{2}{n}(\hat{s}_1+n! \, b_{n-1})\eqskip
F_0:= n-\fr{2}{n}(\hat{s}_2+n! \, b_{n-2})+\fr{(n-2)}{n^2}(\hat{s}_1+n! \, b_{n-1})^2\eqskip
F_{-1}:=-\fr{2}{n}(\hat{s}_3+n! \, b_{n-3}) +\fr{n-2}{n^2}2(\hat{s}_1+n! \, b_{n-1})(\hat{s}_2+n! \, b_{n-2})-
\fr{2(n-1)(n-2)}{3n^3}(\hat{s}_1+n! \, b_{n-1})^3,
\end{array}\right.
\end{gather*}
and $\hat{s}_3=0$ if $n=3$.
 The following holds.\\[2mm]
{\rm 1)} $F_1=0$ if and only if  $b_{n-1}=\frac{1}{(n-1)!}$, that is,
$j^*(K)-j(K)$ is of degree at most $n-2$. Moreover, either $F_1=0$ or $F_1>0$ and so $Pol_{j(K)}{(K)}>0$ for $K$ sufficiently large.\\[1mm]
{\rm 2)} If $F_1=0$, then $F_0=0$ if and only if $b_{n-2}=\frac{5n+2}{12(n-2)!}$.
Moreover, if $F_0>0$, we have $b_{n-2}< \frac{5n+2}{12(n-2)!}< b^*_{n-2}{(n-1)}$, and $Pol_{j(K)}{(K)}\geq 0$ for $K$ sufficiently large; if $F_0<0$, we have
$ \frac{5n+2}{12(n-2)!}<b_{n-2}\leq b^*_{n-2}{(n-1)}$ and $Pol_{j(K)}{(K)}<0$ for $K$ sufficiently large;  if $F_0=0$,
we have $b_{n-2}<b^*_{n-2}{(n-1)}$ and the sign of $Pol_{j(K)}(K)$ will be given by the sign of $F_{-1}$ if not zero, proceeding
as above taking a higher order Taylor expansion.\\[1mm]
{\rm 3)} Given $j\in \mathbb{N}$,   $Pol_j(K)>0$ and $Pol_{m(K)-j}(K)<0$ 
hold for all $K$ sufficiently large.
\end{lemma}

\begin{proof}
Using the third order Taylor expansion of $(1+y)^{2/n}$ for small $y$ we obtain
\begin{gather*}
Pol_{j(K)}(K)
=(K-1)^2+(n+1)(K-1)+n -(K-1)^{2}\left(1+\dsum_{l=1}^{n-1}\fr{\hat{s}_l+n!b_{n-l}}{(K-1)^l}
+ \fr{n!b_0}{(K-1)^n}\right)^{2/n} 
\end{gather*}\\
\begin{gather*}
\begin{array}{lll}
&=& (K-1)^2+(n+1)(K-1)+n-(K-1)^{2}\left\{1+\fr{2}{n}\left(\dsum_{l=1}^{n-1}\fr{\hat{s}_l+n!b_{n-l}}{(K-1)^l}
+\fr{n!b_0}{(K-1)^n}\right)\right.\eqskip
& & \hspace*{5mm}\lefteqn{\left.
-\fr{n-2}{n^2}\left(\dsum_{l=1}^{n-1}\fr{\hat{s}_l+n!b_{n-l}}{(K-1)^l}+\fr{n!b_0}{(K-1)^n}\right)^2\right.}\eqskip
& &\hspace*{1cm}\left. +\fr{2(n-1)(n-2)}{3n^3}\left(\dsum_{l=1}^{n-1}\fr{\hat{s}_l+n!b_{n-l}}{(K-1)^l}+\fr{n!b_0}{(K-1)^n}
\right)^3+\so\left(\fr{1}{(K-1)^3}\right)\right\}^{2/n}\eqskip
& = & F_1(K-1) +F_0 + \fr{F_{-1}}{(K-1)} +\so\left(\fr{1}{K-1}\right).
\end{array}
\end{gather*}
The statements now follow trivially. Note that in~1) $F_1<0$
cannot hold since $j(K)\leq j^*(K)$. In~2) the inequality
$\frac{5n+2}{12(n-2)!}< b^*_{n-2}{(n-1)}$ holds for all $n\geq 3$. 
 Given $j$, the constant polynomial $j(K)=j$ satisfies
$j(K)< m(K)$ for $K$ sufficiently large, and $b_{n-1}=0$ implying 
$F_1=2$ and so $Pol_j(K)>0$ proving the first inequality of~3). The second inequality is a trivial application of~1).
\end{proof}

\begin{proposition}\label{middlePolya} 
 let $G=\frac{n-2}{12}\in \mathbb{Q}$.
Consider the element of $\mathbb{Q}[K]$ of degree $(n-1)$ given by
$$j^{\dag}(K)=j^*_{n-1}(K)-Gj^*_{n-2}(K).$$
For any $j(K)\in \mathbb{R}[K]$ satisfying $1\leq j(K)\leq j^*(K)$ for $K$ large, the following holds.\\[1mm]
(a) If $j(K)> j^{\dag}(K)$ for $K$ large, then  $j(K)> j^{\ddag}(K)$ for all $K$ sufficiently large. \\[1mm]
(b) If  $j(K)< j^{\dag}(K)$ for $K$ large, then $j(K)< j^{\ddag}(K)$  for all $K$ sufficiently large.
\\[1mm]
In particular,  given $\epsilon>0$, there exist $K_{\epsilon}$ such that 
$\sup_{K\geq K_{\epsilon}}| j^{\ddag}(K) - j^{\dag}(K)|\leq \epsilon$.
\end{proposition}

\begin{remark}
 If  $j(K)= j^{\dag}(K)$, in order to determine the sign of $Pol_{j(K)}(K)$ we need to use a Taylor expansion of higher order in Lemma~\ref{F}.
\end{remark}

\begin{proof}
Denote the coefficients of $j^*_{n-1}(K)$, $j^*_{n-2}(K)$, and $j(K)$ by $b^*_l$, $\tilde{b}^*_l$ and $b_l$, respectively. \\[1mm]
(a) By the conditions on $j(K)$,
we must have $b_{n-1}=b^*_{n-1}$ and so, $F_1=0$ from 1. of the previous Lemma. Moreover, since $\hat{s}_1/(n-1) - G= (5n+2)/12$, then 
$$b_{n-2}> b^*_{n-2}-G\tilde{b}^*_{n-2}=\frac{\hat{s}_1}{(n-1)!}-
\frac{G}{(n-2)!}=\frac{5n+2}{12(n-2)!}.$$
This means $F_0<0$ and so (a) follows from Lemma~\ref{F}-2.
(b) From the  conditions on $j(K)$, $b_{n-1}< b^*_{n-1}$ and so $F_1>0$, 
and the result follows  from from Lemma~\ref{F}-1.
Taking $j(K)=j^{\dag}(K)\pm \epsilon$ (note that $ 1\leq j(K)\leq j^*(K)$ for $K$ large) we obtain 
$j^{\dag}(K)-\epsilon \leq j^{\ddag}(K) \leq j^{\dag}(K)+\epsilon$, for all  $K\geq K_{\epsilon}$
proving the last statement.
\end{proof}
Using Proposition~\ref{averagePolya}, given $L>0$ we can take $K_{n,L}\geq K'_n+1$ such that
$P_{m}(K)\geq L$ for all $K\geq K_{n,L}$. Then for $K\geq K_{n,L}$  we have
\begin{eqnarray}
PM(K) &:=& \frac{1}{\sigma(K)}\sum_{k=1}^{k_+(K)}  \left(\lambda_k -\cw k^{2/n}\right)\nonumber\\
&=& \sum_{K'=1}^{K'_{n}-1} \frac{m(K') P_{m}(K')}{\sigma(K)}+\sum_{K'=K'_n}^{K_{n,L}-1} \frac{m(K') P_{m}(K')}{\sigma(K)}
+ \sum_{K'=K_{n,L}}^{K} \frac{m(K') P_{m}(K')}{\sigma(K)}  \nonumber\\
&\geq&T_n
+   \frac{(\sigma(K)-\sigma(K_{n,L}-1))}{\sigma(K)} L \label{(AA)}
\end{eqnarray}
where $T_n=\dsum_{K'=1}^{K'_{n}-1}{m(K') P_{m}(K')}/{\sigma(K)}$.
The expression in (\ref{(AA)}) converges to $L$ when $K\to +\infty$. Therefore, since $L$ is arbitrary, $PM(K)$ may be made as large as we want,
proving the following lemma.
\begin{lemma} 
${\ds \lim_{K\to +\infty}} PM(K)=+\infty$. 
\end{lemma} 

\begin{proof}[Proof of item 2. of Theorem~\ref{Theo02sums}]
Let $k\geq 2$ belong to a $K-$chain, $k=\sigma(K-1)+j(K)$, where $1 \leq j(K)\leq m(K)$.
Then 
\[
\fr{1}{k}\dsum_{s=1}^k(\lambda_s-\cw s^{2/n})
=\frac{\sigma(K-1)}{\sigma(K-1)+j(K)}PM(K-1)+\frac{1}{\sigma(K-1)+j(K)}\sum_{j=1}^{j(K)}
Pol_j(K).
\]
 Moreover, since $j(K)\leq m(K)$,
 $$\frac{j(K)}{\sigma(K-1)+j(K)}\leq \frac{m(K)}{\sigma(K-1)+m(K)}=\frac{m(K)}{\sigma(K)}
=\frac{n}{K+n-1}\to 0,$$ when $K\to +\infty$. 
Therefore,  
$ \frac{\sigma(K-1)}{\sigma(K-1)+j(K)}\to 1$, and the first term of the above equality converges to $+\infty$ as a consequence of
previous lemma. If $j(K)\leq j^{\ddag}(K)$, the second term is positive, and we are done. If $j(K)> j^{\ddag}(K)$, we note that, 
$$ \frac{1}{\sigma(K-1)+j(K)}\sum_{j=1}^{j(K)}Pol_j(K)
\geq  \frac{1}{\sigma(K-1)+j(K)}\sum_{j=1}^{m(K)} Pol_j(K)\geq 0, $$
where the last inequality follows from item 1. This completes the proof.
\end{proof}

\section{Proof of Theorem~\ref{Theo04}\label{phiconvergence}}

\txtb{The proof of Theorem~\ref{Theo04} is based on the fact that the function $\Phi$ defined in Section~\ref{sec:eigenvaluesofhemispheres} converges to $(n-1)(n-2)/6$
as $K$ goes to infinity, and never goes above this value. While the first part is essentially based on an expansion of this function at infinity, the second part is
more involved, relying on the sharp upper bound provided by Lemma~\ref{Mother}.}

\begin{lemma}\label{B}
The expansion
$$\left[\frac{\Gamma(n+K)}{\Gamma(K)}\right]^{2/n}=K^2+(n-1)K+ \frac{(n-1)(n-2)}{6}+ \so(\txtb{1})$$
holds as $K \to \infty$.\\
\end{lemma}

\begin{proof} We have
\[
\begin{array}{lll}
\left[\fr{\Gamma(n+K)}{\Gamma(K)}\right]^{2/n} & = &\left[K(K+1)\ldots(K+n-1)\right]^{2/n} \eqskip
& = & \left[K^n\left(1+\fr{1}{K}\right)\left(1+\fr{2}{K}\right)\ldots\left(1+\fr{n-1}{K}\right)\right]^{2/n}\eqskip
& = & K^2\left[1 +\hat{s}_1\fr{1}{K} +\hat{s}_2\fr{1}{K^2} + \so\left(\fr{1}{K^2}\right)\right]^{2/n}\eqskip
&=& K^2\left[1+\fr{n(n-1)}{2K}+\fr{(n-1)(n-2)n(3n-1)}{24 K^2}+ \so\left(\fr{1}{K^2}\right)\right]^{2/n},
\end{array}
\]
where $\hat{s}_1$ and $\hat{s}_2$ are as in~(\ref{allsigmas}). The lemma now follows from the binomial expansion. 
\end{proof}

\begin{lemma}\label{C}
$$\lim_{K\to \infty}\Phi(K)=\frac{(n-1)(n-2)}{6}=:c(n).$$
\end{lemma}
\begin{proof}
From Lemma~{\ref{B}},
\begin{eqnarray*}
\Phi(K) &=& \left(K^2+(n-1)K+ \frac{(n-1)(n-2)}{6}+ \so(K)\right)- K(K+n-1)\\
&=&\frac{(n-1)(n-2)}{6}+ \so(K),
\end{eqnarray*}
as $K\to+\infty$, and the limit follows.
\end{proof}

\begin{proof}[Proof of Theorem \ref{Theo04}.] 
If $n=2$ then $\Phi(K)=0$. We may thus assume $n\geq 3$. \txtb{By Lemma~\ref{Mother}, $\Phi(K)$ remains below $c(n)$ for all $K$ and so}
\begin{equation}\label{AQUI}
 \bl_K> \left(\frac{\Gamma(n+K)}{\Gamma(K)}\right)^{2/n}-c(n).
\end{equation}
Hence, if $\lambda_k$ is in the $K-$chain, then $k\leq k_+$ and  using~\eqref{AQUI} we obtain
$$\lambda_k=\bl_K> \cw\, \sigma(K)^{{2}/{n}}-c(n)= \cw\, k_+^{2/n}
-c(n)\geq \cw k^{2/n}-c(n).$$
Equality is only possible if  $n=2$, when $C_{W,2}=1$ and  $c(2)=\Phi(K)=0$. By Lemma~\ref{C} above, $\Phi(K)$ converges to $c(n)$ as $K$ goes to infinity,
proving the sharpness of the statement with respect to the sequence of highest order on each chain.
\end{proof}

\section{Proof of Theorem~\ref{Theo05}  \label{sec: THEOREM CD}}
\txtb{
We now describe the $K-$chains of an hemisphere in a precise way. By writing explicitly the lower and higher orders $k_{\pm}(K)$
in terms of $K^{\bar{n}}$, and using a simple but fundamental inequality between $K(K+n-1)$ and a fractional power of the 
rising factorial, it is possible to obtain a proof of Theorem~\ref{Theo05}.
}

\begin{proof}[Proof  of Theorem~\ref{Theo05}.]
The inequality \txtb{in the theorem} holds for $k=1$. Now we take $k\geq 2$ in a  $(K+1)$-chain, where $K\geq1$. The corresponding lowest order eigenvalue
is given by
$k_-=\sigma(K)+1=K^{\overline{n}}/n!+1$. By Lemma~\ref{Mother} the following inequality holds
\[K(K+n-1)\leq \left(K^{\overline{n}}\right)^{2/n}=(k_--1)^{2/n}(n!)^{2/n}.\]
In particular we obtain
\[K\leq (k_--1)^{1/n}(n!)^{1/n}.\]
Therefore, 
\begin{eqnarray*}
\lambda_k = \bl_{K+1} &=&(K+1)(K+n)\\
&=& K(K+n-1)+2K+n\\
 &\leq& (k_--1)^{2/n}(n!)^{2/n} + 2(k_--1)^{1/n}(n!)^{1/n} +n\\
 &\leq&  (k-1)^{2/n}(n!)^{2/n} + 2(k-1)^{1/n}(n!)^{1/n} +n\\
&=& \cw (k-1)^{2/n} + 2\sqrt{\cw} (k-1)^{1/n} +n,
\end{eqnarray*}
and we are done.
Next we prove the claimed limit at infinity.
Consider the function defined for all $K\geq 2$,  
$ \Theta(K) := {(\bl_K-\Upsilon(K)^{2/n}-n})/{(2\Upsilon(K)^{1/n})}$, 
where $\Upsilon(K)=\sigma(K-1)=k_--1$ is given in (\ref{Theta(K)}) for $j=0$, and $k_-$ is the lower order index the $K$-chain. 
From the inequality this function is bounded from above by one.
Consider the quantities $\hat{s}_l$ defined by~(\ref{allsigmas})  and 
$\hat{x}(K)=y_0(K)$ given in (\ref{hatx0}) for $j=0$.
 Using Lemma~\ref{Taylor} when $\hat{x}(K)\to 0$,  we have $((K-1)^{\overline{n}})^{1/n}=(K-1)+\bo\left(1\right)$ when $K\to +\infty$ and
\begin{gather*}
\begin{array}{lll}
((K-1)^{\overline{n}})^{2/n} & = & (K-1)^2\left[ 1+\fr{2}{n}\hat{x}(K)
-\fr{(n-2)}{n^2}{\hat{x}(K)^2}+ \so\left(\hat{x}(K)\right)^2\right]\eqskip
& = & (K-1)^2+\fr{2}{n}\hat{s}_1(K-1)+\left[\fr{2}{n}\hat{s}_2-\fr{(n-2)}{n^2}\hat{s}_1^2\right]
+\bo\left(\fr{1}{K}\right).
\end{array}
\end{gather*}

Moreover, $K(K+n-1)=(K-1)^2+ (n+1)(K-1)+n$, and we get
\begin{eqnarray*}
\lim_{K\to +\infty}\Theta(K) 
&=& \lim_{K\to +\infty}\fr{ (n+1-n+1)K -\fr{2}{n}\hat{s}_2+\fr{(n-2)\hat{s}_1^2}{n^2}}{2K}~=~1. 
\end{eqnarray*}
\end{proof}

\section{Sharp two- and three-term upper bounds for $\mathbb{S}^2_+$, $\mathbb{S}^3_+$ and 
$\mathbb{S}^4_+$ \label{sec: TWO THREE TERM}}

Recall from Section~\ref{sec:eigenvaluesofhemispheres} that  $\sigma(K)=\frac{K^{\overline{n}}}{n!}$. Hence,
 $\sigma:[1, +\infty)\to [1, +\infty)$ is the smooth increasing function
$$ \sigma(x)=\frac{1}{n!}x(x+1)\ldots (x+n-1),$$
that can be extended to $x\in [0, 1]$ by the same formula, giving $\sigma(0)=0$. We can define its inverse,
$\sigma^{-1}:[0, +\infty)\to [0, +\infty)$, with $\sigma^{-1}(0)=0$,  that is increasing for $k\geq 1$.
Since $k_- -1= \sigma(K-1)$ and  $ k_+=\sigma(K)$, then 
$  \sigma^{-1}(k_+)= K =\sigma^{-1}(k_--1)+1$. In particular, for any $\lambda_k$ of the $K-$chain,
we have $k_-\leq k\leq k_+$, and so
\begin{eqnarray}
\lambda_k=\bl_K =K(K+n-1) &=&
  \left(\sigma^{-1}(k_--1)+1\right) \left(\sigma^{-1}(k_--1)+n\right)\nonumber\\
&\leq&\left(\sigma^{-1}(k-1)+1\right) \left(\sigma^{-1}(k-1)+n\right)=:\up(k)\label{IN}
\end{eqnarray}
\txtb{with equality if and only if $k=k_-$, while
\[
\begin{array}{lll}
  \lambda_k=K(K+n-1) &=& \sigma^{-1}(k_+)(\sigma^{-1}(k_+)+n-1\eqskip
   & \geq & \sigma^{-1}(k)(\sigma^{-1}(k)+n-1).
\end{array}
\]}
If $k=K=1$, $\lambda_1=n= \up (1)$.

Given $z\in [0, +\infty)$,  $\sigma^{-1}(z)$ is a non-negative 
root of a polynomial function of degree $n$, that provides formulas to describe each $K$ in terms of the lowest order index $k_-$ or the largest order
index $k_+$ of the chain. Rewriting the expression of $\up (k)$ in the form $\cw k^{2/n}+ c'k^{1/n}+ \ldots$ a remaining bounded term will appear,
$R_{-}(k)$, allowing us to derive  sharp estimations for the eigenvalues, as we will do in next
propositions. For simplicity, and due to the nature of the polynomial equations involved, we will restrict ourselves to the cases of $n=2,3$ and $4$.

\subsection{$\mathbb{S}^2_+$}
If $n=2$ we have $C_{W,2}=2$ and, from the above formulas,
\begin{gather*}
\sigma^{-1}(x)=-\frac{1}{2}+\frac{1}{2}\sqrt{8x+1}\\
\txtb{\mbox{and }2k_+} = \bl_K=K(K+1)= 1+ 2(k_- -1) +\sqrt{8(k_--1)+1}.
\end{gather*}
Thus, the two- and three- term upper bounds in the next proposition follow immediately. \txtb{The lower bound is P\'{o}lya's inequality from Theorem~\ref{Theo02},}
while the last statement was proved in Theorem~\ref{Theo05}.

\begin{proposition}\label{Hemi2dim} If $n=2$, $C_{W,2}=2$ and  the  Dirichlet eigenvalues of $\mathbb{S}^2_+$  satisfy the following inequalities
for all $k=1,2,\ldots$,
\[
\begin{array}{llll}
\txtb{ C_{W,2}k} & \leq \lambda_k & \leq & C_{W,2}(k-1) +2\sqrt{C_{W,2}}\sqrt{(k-1)+\frac{1}{8} } +1\eqskip
& & = &C_{W,2}\,k +2\sqrt{C_{W,2}}\,\sqrt{k-\fr{7}{8}}~-1\eqskip
& & < & C_{W,2}\, k +2\sqrt{C_{W,2}}\,\sqrt{k},
\end{array}
\]
with equality in the first right-hand side inequality if and only if $k=k_-$. The last inequality is strict for any $k\geq 1$ but it
is asymptotically sharp for the lowest order eigenvalues of each chain in the sense that 
$(\lambda_k-C_{W,2} k)/(2\sqrt{C_{W,2}}\sqrt{k})\to 1$ when $k=k_-\to \infty$.\\
\end{proposition}

\subsubsection{$\mathbb{S}^3_+$}
Now we consider the case $n=3$. We have $C_{W,3} = 6^{{2}/{3}}$, $ \bl_{K} = K(K+2)$,
and from 
$\sigma(K)={K(K+1)(K+2)}/{6}$ we obtain
$$ \begin{array}{ll} 
&\sigma^{-1}(x)=3^{-{2}/{3}}L(x)^{{1}/{3}}+3^{-{1}/{3}}L(x)^{-{1}/{3}}-1,\\[1mm]
\mbox{with}& 
 L(x)=3^3x + \sqrt{3^6x^2 -3},
\quad \forall x\geq 1.
\end{array}$$
Note that $\sigma(0)=0$ by the above expression, and this is the only non-negative value of $K$ for which $\sigma$ vanishes, so we set $\sigma^{-1}(0)=0$.

\begin{proposition}\label{Hemi3dim}
If $n=3$, $C_{W,3}=6^{{2}/{3}}$ and  the Dirichlet eigenvalues of  $\mathbb{S}^3_+$,  $\lambda_k$, $k=1, 2, \ldots$ satisfy the following inequality
\begin{eqnarray}\label{INEQ3.1}
\lambda_k &\leq& C_{W,3}\,(k-1)^{{2}/{3}}+2\sqrt{C_{W,3}}\,(k-1)^{{1}/{3}} +\tilde{R}_-(3,k),
\end{eqnarray}
 where  $\tilde{R}_-(3,k)$   is a function that at $k=1$ takes on the value $3=n=\lambda_1$, and for $k\geq 2$ it is a decreasing function,
 with $\tilde{R}_-(3,2)\approx 1.06383$ and converging to $2/3$ as $k\to +\infty$. Equality in (\ref{INEQ3.1}) holds if and only if $k=k_-$.
An equivalent upper bound for $\lambda_k$ is given by
\begin{equation}\label{INEQ3.2}
\lambda_k \leq C_{W,3}\,k^{{2}/{3}}+2\sqrt{C_{W,3}}\,k^{{1}/{3}} + \hat{R}_-(3,k),
\end{equation}
with equality at $k=k_-$ of each chain, 
where  $\hat{R}_{-}(3,k)$ at $k=1$ evalutes to $-C_{W,3}- 2 \sqrt{C_{W,3}}+ 3 \approx -3.936168 $,  and it is bounded  for $k\geq 2$,
with $\hat{R}_-(3,2)\approx -2.4870064$,  increasing up to $2/3$ as $k\to \infty$.  It is given by 
\begin{eqnarray*}
 \hat{R}_-(3,k)&=&   \tilde{R}_-(3,k)-C_{W,3}\left( k^{{2}/{3}}-(k-1)^{{2}/{3}}\right)
-2\sqrt{C_{W,3}}\left( k^{{1}/{3}}- (k-1)^{{1}/{3}}\right).
\end{eqnarray*}
\end{proposition}

\begin{proof}

Applying the above formula of $\sigma^{-1}$ to $x=k_--1$
and insert into (\ref{IN})  we obtain the following expression for $\up (k)$, 
\begin{gather*}
\begin{array}{lll}
\up (k) & = & \left(3^{-{2}/{3}}L(k-1)^{{1}/{3}}+3^{-{1}/{3}}L(k-1)^{-{1}/{3}}\right)^2
+2\left(3^{-{1}/{3}}L(k-1)^{{1}/{3}}+3^{-{1}/{3}}L(k-1)^{-{1}/{3}}+2\right)\eqskip
& = & C_{W,3}(k-1)^{{2}/{3}}\left(\fr{1}{2}+\fr{1}{2}\sqrt{1-\fr{3}{(27(k-1))^2}}\right)^{{2}/{3}}
+2 \sqrt{C_{W,3}}(k-1)^{{1}/{3}}\left(\fr{1}{2}+\fr{1}{2}\sqrt{1-\fr{3}{(27(k-1))^2}}\right)^{{1}/{3}}\eqskip
 & &  \hspace*{5mm}+ \fr{2}{3} 
 +\fr{2}{3^{{1}/{3}}\left(27(k-1)+\sqrt{(27(k-1))^2-3}\right)^{{1}/{3}}}
+\fr{1}{3^{{2}/{3}}\left(27(k-1)+\sqrt{(27(k-1))^2-3}\right)^{{2}/{3}}}.
\end{array}
\end{gather*}
Writing $\up (k)$ involving remainder $\tilde{R}_-(3,k)$, it satisfies $\tilde{R}_-(3,1)=3$ and for $k\geq 2$  it is given by
\[\tilde{R}_-(3, k) = A(k)+B(k) +\frac{2}{3}+ C(k)+D(k),\]
where
\begin{eqnarray*}
\left\{\begin{array}{lcl}
A(k) &=& -C_{W,3}(k-1)^{{2}/{3}}\left\{ 1-\left( \frac{1}{2}+\frac{1}{2}\sqrt{1-\frac{3}{(27(k-1))^2}}\right)^{{2}/{3}}\right\}\eqskip
B(k)&=&-2\sqrt{C_{W,3}}(k-1)^{{1}/{3}}\left\{ 1-\left( \frac{1}{2}+\frac{1}{2}\sqrt{1-\frac{3}{(27(k-1))^2}}\right)^{{1}/{3}}\right\}\eqskip
D(k)&=&\frac{1}{3^{{2}/{3}}\left(27(k-1)+\sqrt{(27(k-1))^2-3}\right)^{{2}/{3}}}\eqskip
C(k)&=& 2\sqrt{D(k)}.
\end{array}\right.
\end{eqnarray*}
It follows that
$\tilde{R}_-(3,k)$ is a positive function decreasing to $2/3$ when $k\to \infty$ --
see proof in Section \ref{tilde{R}minus3} of Appendix~\ref{sec:MATH}.
We note first that $\hat{R}_-(3,k)$  converges at infinity to the same limit as $\tilde{R_-}(3,k)$ , since 
$0<k^{a/3}-(k-1)^{a/3}\leq \frac{a}{3}
(k-1)^{-(3-a)/3}$  for $a=1,2$.  
On the other hand after some straightforward simplifications we may rewrite $\hat{R}_{-}(3,k)$ as
\begin{gather}\label{hatR3}
\hat{R}_-(3,k)=-\cw k^{2/3}-2\sqrt{\cw}k^{1/3}+\fr{2}{3} + 
\left( D(k)+\fr{D(k)^{-1}}{3^2}\right) + 2\left( \sqrt{D(k)}+\fr{(\sqrt{D(k)})^{-1}}{3}\right)
\end{gather} 
From this expression  we obtain a negative value at $k=2$, implying $\hat{R}_-(3,k)$ changes sign. In
Section \ref{tilde{R}minus3} of Appendix~\ref{sec:MATH} we prove that $\hat{R}_-(3,k)$
is an increasing function for $k$ sufficiently large. 
\end{proof}

\subsubsection{}
Now we consider the case $n=4$. We have $C_{W,4}= (24)^{{1}/{2}}$, $\bl_K=K(K+3)$. From
$ \sigma(K)= {K(K+1)(K+2)(K+3)}/{24}$ we obtain
$$\sigma^{-1}(x)=\frac{1}{2}\left(-3 + \left(5 + 4 \sqrt{24 x+1}\right)^{{1}/{2}}\right).$$

\begin{proposition}\label{Hemi4dim}
For $n=4$, $C_{W,4}=\sqrt{24}$ and the Dirichlet eigenvalues of $\mathbb{S}^n_+$, $k=1,2,\ldots$,  satisfy the following inequality
\begin{eqnarray*}
\lambda_k &\leq& C_{W,4}( k-1)^{{1}/{2}}+ 2\sqrt{C_{W,4}}(  k -1)^{{1}/{2}} +\tilde{R}_-(4,k),
\end{eqnarray*}
where $\tilde{R}_-(4,k)$ is a decreasing function with $\tilde{R}_-(4,1)=4$
 and converging to zero as $k\to \infty$.
Equality holds in the above inequality if and only if $k=k_-$, that is at eigenvalues of lowest order of each chain.  Equivalently, we have 
\begin{eqnarray*}
\lambda_k &\leq& C_{W,4}\,k^{{1}/{2}}+ 2\sqrt{C_{W,4}}\, k ^{{1}/{4}} + \hat{R}_-(4,k),\\
\mbox{where}&& \hat{R}_-(4,k)=   \tilde{R}_-(4,k)-C_{W,4}\left( k^{{1}/{2}}-(k-1)^{{1}/{2}}\right)-2\sqrt{C_{W,4}}\left( k^{{1}/{4}}- (k-1)^{{1}/{4}}\right),
\end{eqnarray*}
and equality holds  if and only if $k=k_-$ as well. The remainder 
$\hat{R}_{-}(4,k)$  is not monotonic and changes sign. It evaluates to $-C_{W,4}- 2 \sqrt{C_{W,4}}+ 4 \approx -5.326$ at $k=1$, then vanishes somewhere
on the interval $[400,500,]$, 
it has a local  positive maximum of approximately $0.0322267$ around $k=6452$,  and  converges to zero at infinity.
\end{proposition}

\begin{proof}
From (\ref{IN}) and the expression of $\sigma^{-1}$ we obtain $\up (k)$ and develop it as follows.
\begin{align*}
\up (k) &= \sqrt{1 +24 (k-1)}+\sqrt{5+ 4 \sqrt{1 +24 (k-1)}}\\
&= C_{W,4}\left( k -1+ \frac{1}{24}\right)^{{1}/{2}}+ 2\sqrt{C_{W,4}}\left( \left( k -1+ \frac{1}{24}\right)^{{1}/{2}} +\frac{5}{4C_{W,4}}\right)^{{1}/{2}}\\
&= C_{W,4}(k-1)^{{1}/{2}}\sqrt{ 1 +\frac{1}{24 (k-1)}}+ 2\sqrt{C_{W,4}}(k-1)^{{1}/{4}}
 \sqrt{\sqrt{ 1 +\frac{1}{24 (k-1)}}+
\frac{5}{4\sqrt{ 24 (k-1)}}} \\
&= C_{W,4}(k-1)^{{1}/{2}}+ 2\sqrt{C_{W,4}}(k-1)^{{1}/{4}}+ \tilde{R}_{-}(4,k),
\end{align*}
where $\tilde{R}_{-}(4,1)=4=n=\lambda_1$ and for $k\geq 2$
\begin{gather*} 
\tilde{R}_{-}(4,k) = C_{W,4}(k-1)^{{1}/{2}}\left(\sqrt{ 1 +\frac{1}{24 (k-1)}}-1\right)+ 2\sqrt{C_{W,4}}(k-1)^{{1}/{4}}
\left( \sqrt{\sqrt{ 1 +\frac{1}{24 (k-1)}}+
\frac{5}{4\sqrt{ 24 (k-1)}}}-1\right).
\end{gather*}
Clearly $\tilde{R}_{-}(4, k)>0$, and a straightforward calculation yields ${\ds \lim_{k\to \infty}} \tilde{R}_{-}(4,k)=0$.
The proof that $\tilde{R}_{-}(4,k)$ is a decreasing function of $k$ is given in Section~\ref{tilde{R}minus4} of Appendix~\ref{sec:MATH}.
The expression using $\hat{R}_{-}(4,k)$ follows by a direct computation, or
by rewriting $\up (k)$ obtaining
\[ \up (k)= C_{W,4}k^{1/2}+2\sqrt{C_{W,4}}k^{1/4}+\hat{R}_-(4,k),\]
with
\[\hat{R}_-(4,k)= \sqrt{5+ 4\sqrt{24k -23}}+\sqrt{24k -23}-\sqrt{24k}-2(24k)^{1/4}.\]
From this expression we see that the limit at infinity is zero, and derive the remaining properties of $\hat{R}_-(4,k)$.
\end{proof}

\subsection{Single term lower bound}
The main purpose of this Section is to prove the following one-term lower bound without an additive constant and valid for all $k\geq 1$.  
\begin{proposition} \label{Theo03}
On $\mathbb{S}^n_+$, $n\geq 2$, the following inequality  is valid for all  eigenvalues $\lambda_k$,
\begin{equation}\label{Li-Yau}
 \lambda_k\geq 
\left(\frac{n}{(n!)^{{2}/{n}}}\right)\cdot \cw k^{{2}/{n}}= nk^{{2}/{n}} =\lambda_1 k^{{2}/{n}},
\end{equation} 
with equality holding  for the first eigenvalue. For $n=2$ this is P\'{o}lya's inequality~\eqref{Polya_ineq}.
\end{proposition}

This one-term inequality should  be compared with the main inequality in \cite{ly} for Euclidean domains, where a
lighter correction constant $n/(n+2)$ is multiplied. Furthermore, our result on the hemisphere implies
$\lambda_{k+1}\geq \lambda_1 k^{{2}/{n}}$, an opposite inequality compared with  Euclidean domains  \cite{cy2007}. 
In~\cite[Corollary 1.1]{cy2007} it is shown that spherical domains satisfy an inequality with two correction constants. The multiplicative
constant, $n/\sqrt{(n+2)(n+4)}$,  converges to one when $n \to \infty$,
while our correction constant ${n}/(n!)^{2/n}$ converges to zero when $n\to \infty$. Their additive correction constant,
$n^2/4$, is comparable to our $c(n)=(n-1)(n-2)/6$  in Theorem \ref{Theo04}, but with no need of an extra multiplicative
correction constant. These multiplicative constants turn out to be a compromise in order to make the inequality hold
for lower eigenvalues also. The proof of Proposition~\ref{Theo03} relies on properties of $\mathcal{R}(K)$.
This function may be extended to all real $K>0$ and is smaller than one, by Lemma~\ref{Mother}. 

\begin{lemma}\label{R'} Let $\psi(t)$ be the digamma function~\cite[p.\ 253]{A}. For all real $K>0$ 
\begin{gather*}
\mathcal{R}'(K) = \mathcal{R}(K)\left(\frac{2}{n}\left(\psi(n+K)-\psi(K)\right)-\frac{2K+n-1}
{K(n+K-1)}\right).
\end{gather*}
\end{lemma}
\begin{proof}
 Denoting the derivative with respect to $K$ by $'$, for any  real $K\geq 1$ we have
 \begin{gather*}
\begin{array}{lll}
\mathcal{R}'(K) & = & \left({\left(\fr{\Gamma(n+K)}{\Gamma(K)}\right)^{{2}/{n}}}\times{(K(n+K-1))^{-1}}
\right)'\eqskip
& = & \left(\fr{\Gamma(n+K)}{\Gamma(K)}\right)^{{2}/{n}}{
\left(
\frac{\fr{2}{n}\left(\fr{\Gamma(n+K)}{\Gamma(K)}\right)'}{\left(\fr{\Gamma(n+K)}{\Gamma(K)}\right)}K(n+K-1)-(2K+n-1)\right)}\times {(K(n+K-1))^{-2}}\eqskip
& = & \left(\fr{\Gamma(n+K)}{\Gamma(K)}\right)^{{2}/{n}}\left(
\fr{2\left(\ln(\Gamma(n+K))-\ln(\Gamma(K))\right)'}{n K(n+K-1)}-\fr{(2K+n-1)}{(K(n+K-1))^2}\right)\eqskip
& = & \left(\fr{\Gamma(n+K)}{\Gamma(K)}\right)^{{2}/{n}}\left(
\fr{2\left(\psi(n+K)-\psi(K)\right)}{n K(n+K-1)}-\fr{(2K+n-1)}{(K(n+K-1))^2}\right)\eqskip
& = & \fr{2}{n K(n+K-1)}\left(\fr{\Gamma(n+K)}{\Gamma(K)}\right)^{{2}/{n}}\left(\psi(n+K)-\psi(K) - \fr{n(2K+n-1)}{2K(n+K-1)}\right),
\end{array}
\end{gather*}
and the expression for the derivative follows. 
\end{proof}

\begin{lemma} \label{R'K} If $n\geq 3$ we have $\mathcal{R}'(K)<0$ for any
real $K>0$.
\end{lemma}
\begin{proof} From the above expression of $\mathcal{R}'(K)$,  we need to show that for any real $K>0$, the following inequality
\[\psi(n+K)-\psi(K) = \fr{1}{K}+\ldots+ \fr{1}{K+n-1} <  \fr{n(2K+n-1)}{2K(n+K-1)}\]
holds for all integer  $n\geq 3$. We fix $K$ and prove by induction on $n$. It holds for $n=3$
since the above inequality is equivalent to the polinomial inequality $K(K+2)< (K+1)^2$.
Assume now that the inequality  holds for $n$. We have
\begin{eqnarray*}
 \left(\fr{1}{K}+\fr{1}{K+1}+\ldots+ \fr{1}{K+n-1}+ \fr{1}{K+n}\right) &\leq& 
 \fr{n(2K+n-1)}{2K(K+n-1) } + \fr{1}{K+n}\\
&=&  \fr{n(2K+n-1)(K+n)+2K(n+K-1)}{2K(n+K-1)(K+n)}.
\end{eqnarray*}
The induction is proved if we show that
\[\fr{n(2K+n-1)(K+n)+2K(n+K-1)}{2K(n+K-1)(K+n)}<\fr{(n+1)(2K+n)}{2K(K+n)}\]
or, equivalently, $(2K+n-1)(K+n)< (K+n-1)(2K+n+1)$, 
that is, $-1+n>0$.
\end{proof}

\begin{proof}[Proof of Proposition \ref{Theo03}]
 If $n\geq 3$, By Lemma \ref{R'K}  we  have 
$\mathcal{R}'(K)<0$. In particular $\mathcal{R}(L)<\mathcal{R}(K)$ for $L>K\geq 1$.
Therefore,
$$
\begin{array}{lll}
\fr{\bl_{L}}{ \bl_{K}}~ > ~ \left[\fr{\Gamma(K)\Gamma(n+L)}{\Gamma(L)\Gamma(n+K)}\right]^{2/n} & \mbox{ and } &
\fr{\bl_{K}}{ \bl_1}~\geq~\cw\left[\fr{\sigma(K)}{(n!)}\right]^{2/n}.\end{array}$$
Finally, $\bl_1=n$, and  if 
$\lambda_k$ is in the $K-$chain, from the last inequality we have 
\begin{eqnarray*}
\lambda_k &=& \bl_K\geq 
n(n!)^{-2/n}\cw(\sigma(K))^{2/n} \geq
n (n!)^{-2/n}\cw\, k^{2/n},
\end{eqnarray*}
and Proposition \ref{Theo03} is proved. 
\end{proof}

\section{The Neumann case: proof of Theorem~\ref{neumannpolya}\label{neumannproof}}
\txtb{ The proof proceeds in a way similar to those of Theorems~\ref{Theo04} and~\ref{Theo05}.}

\txtb{If $k$ is in the $K$-chain, i.e. $\mu_k=K(K+n-1)$ and
$k_{-}'(K)\leq k\leq k_{+}'(K)$,  we have 
\[
\begin{array}{lll}
\cw k^{2/n}-\mu_k & \geq &\cw k_{-}'(K)^{2/n}-\mu_k\eqskip
& = & \left(K^{\bar{n}}\right)^{2/n}-K(K+n-1)\eqskip
& = &\Phi(K)
\end{array}
\]
 where the function $\Phi(K)$ (defined in Section~\ref{sec:eigenvaluesofhemispheres}) was shown in Section~\ref{phiconvergence} to converge to $c(n)$  and,
 by Lemma~\ref{Mother}, is always below this value. Also by that lemma, $\Phi(K)\geq 0$, and we conclude that $\cw k^{2/n}\geq \mu_k$.
 When $n=2$, $\Phi(K)=0$, and thus $\cw k^{2/n}-\mu_k$ vanishes for all $k=k_{-}'(K)$, that is $k=K^{\bar{2}}/2=K(K+1)/2$.
 }

\txtb{
To prove the lower bound we start from 
\[
K+1 \leq \sqrt{(K+1)(K+n)}\leq \left((K+1)^{\overline{n}}\right)^{1/n}=\sqrt{\cw} \left[k_{+}'(K)+1\right]^{1/n}
\]
to obtain
\[
\begin{array}{lll}
\cw k^{2/n}-\mu_k & \leq & \cw k_{+}'(K)^{2/n}-K(K+n-1) \eqskip
 &\leq &\cw (k_{+}'(K)+1)^{2/n}-K(K+n-1)\eqskip 
 &=& ((K+1)^{\overline{n}})^{2/n}-K(K+n-1)\eqskip 
&=& \Phi(K+1)+(K+1)(K+n)-K(K+n-1)\eqskip
&=& \Phi(K+1)+2K +n\eqskip
&=& \Phi(K+1)+ 2(K+1)+(n-2)\eqskip
& \leq & c(n)+(n-2) +2\sqrt{\cw}\left[k_{+}'(K)+1\right]^{1/n}.
\end{array}
\]
Hence $\cw \left[k_{+}'(K)+1\right]^{2/n}-\mu_k \leq [c(n)+n-2] +2\sqrt{\cw}(k_{+}'(K)+1)^{1/n}$, and
 \begin{eqnarray*}
 \mu_k +(c(n)+n-2)&\geq& \sqrt{\cw} (k_{+}'(K)+1)^{1/n}\left(\sqrt{\cw}(k_{+}'(K)+1)^{1/n}-2\right)\\
 &\geq& \sqrt{\cw} (k+1)^{1/n}\left(\sqrt{\cw}(k+1)^{1/n}-2\right)
 \end{eqnarray*}
 Now, proceeding as in the proof of Theorem~\ref{Theo05}, using the identity $K(K+n-1)=(K+1)^2+(n-3)K-1$, recalling that $\hat{s}_1$ and $\hat{s}_2$ are as defined
 in~\eqref{allsigmas} and using the explicit expressions in Appendix~\ref{sec:Combinatorics} with $m=n-1$, we have
 \[
 \begin{array}{lll}
\fr{\cw (k_+(K)+1)^{2/n}-\mu_k}{2\sqrt{\cw}(k_+(K)+1)^{1/n}} & = & \fr{((K+1)^{\bar{n}})^{2/n}-K(K+n-1)}{2\left[(K-1)^{\bar{n}}\right]^{1/n}}\eqskip
 &=&\fr{ (K+1)^2+\frac{2}{n}\hat{s}_1(K+1)+\left[\frac{2}{n}\hat{s}_2-\frac{(n-2)}{n^2}\hat{s}_1^2\right]+\bo\left(\frac{1}{K}\right)}{2\left[(K-1)^{\bar{n}}\right]^{1/n}}\eqskip
 & & \hspace*{5mm} -\fr{ (K+1)^2+(n-3)K+1}{2\left[(K-1)^{\bar{n}}\right]^{1/n}}\eqskip
 &=& \fr{2K + n +\left[\frac{2}{n}\hat{s}_2-\frac{(n-2)}{n^2}\hat{s}_1^2\right]+\bo\left(\frac{1}{K}\right)}
 {2(K+1) + \bo\left(\frac{1}{K}\right)}\to 1
 \end{array}
 \]
 when $K\to +\infty$, concluding the proof of Theorem~\ref{neumannpolya}.
 }

\section{The case of wedges $\wdg{n}{\pi/p}$\label{sec:wedges}}

We recall the concept of a tiling domain of a manifold $M$ in Euclidean space.
If a domain  $M'\subset M\subset \mathbb{R}^N$ contains $p$ non-overlapping subdomains congruent with a model domain $M''\subset M$,
we write $M'\supset pM''$; if these $p$ subdomains cover $M'$ without gaps, we write $M'=pM''$. 
P\'{o}lya \cite[Lemma 1]{poly} proved the following result in the case of domains in the plane, but a similar proof holds in this more
general situation.
\begin{lemma} \label{lemma1-1} 
If $M'\supset pM''$, then $\lambda'_{kp}\leq\lambda''_{k}$ for
$k\in\mathbb{N}$, where $\lambda'_{kp}$, $\lambda''_{k}$ are the
$(kp)$-th and $k$-th  Dirichlet eigenvalues of the Laplacians on $M'$ and
$M''$, respectively.
\end{lemma}

An example of tiling domains on the $n$-dimensional sphere $\mathbb{S}^n$ is given by the wedges defined by~\eqref{wedgedef},
for which we have that $p$ copies of $\mathcal{W}^n_{{\pi}/{p}}$ cover a hemisphere $\mathbb{S}^n_+=\mathcal{W}^n_{{\pi}}=
p\mathcal{W}^n_{{\pi}/{p}}$.

Theorem~\ref{cor:wedge} on $\mathcal{W}^2_{{\pi}/{p}}$ follows immediately 
by applying Lemma~\ref{lemma1-1} and part 1. of Theorem~\ref{Theo02} for the Dirichlet eigenvalues of the Laplacian on the $2$-dimensional
hemisphere $\mathbb{S}^{2}_{+}$. Applying Theorem~\ref{Theo04} and Lemma~\ref{lemma1-1}, we obtain the following.

\begin{thm} \label{theorem-ti} Let $M'$ and $M''$ be two domains in $\mathbb{S}^n_+$.
If $M''$ tiles $M'$ with $M'=pM''$, and $n\geq 2$, then for any $k\geq 1$ we have 
\begin{eqnarray*}
\lambda''_{k}+\frac{(n-1)(n-2)}{6}\geq \left( pk\, n! \right)^{2/n},
\end{eqnarray*}
where 
$\lambda''_{k}$ is the $k$-th Dirichlet eigenvalue of
the Laplacian on $M''$. Furthermore, in case $M'=\mathbb{S}^n_+$, then
$$ \lambda''_{k}+\frac{(n-1)(n-2)}{6}\geq \cw(M'')k^{{2}/{n}}.$$
\end{thm}

An immediate consequence of the above result is the following.
\begin{corollary}\label{theorem1-2} The  eigenvalues  $\lambda''_k$ of  the wedge $\mathcal{W}^{n}_{{\pi}/{p}}$ with
$p\in\mathbb{N}$, satisfy the inequality
\begin{eqnarray} \label{coro1-1}
\lambda''_{k}+\frac{(n-1)(n-2)}{6}\geq \cw(\mathcal{W}^{n}_{{\pi}/{p}}) k^{2/n}= \left( pk\,n!\right)^{2/n}.
\end{eqnarray}
 Moreover, when $n=2$, 
equality holds in (\ref{coro1-1}) if and only if
$k=\frac{m(m+1)}{2}$, $m\in\mathbb{N}$.
\end{corollary}

\begin{proof}[Proof of Theorem \ref{theorem-ti}.]
Let $\lambda''_k$, $\lambda'_k$ and $\lambda_k$ be the Dirichlet eigenvalues of $M''$, $M'$ and $\mathbb{S}^n_+$, respectively.
Let $c(n)=(n-1)(n-2)/6$. 
 By (\ref{PPUU}),
 $(\omega_n|\mathbb{S}^n_+|)^{2/n}=4\pi^2/(n!)^{2/n}$.
In fact, if $M''$ tiles $M'$  with
$M'=pM''$ for some $p\in\mathbb{N}$, then $|M'|=p|M''|$ and
by Theorem \ref{Theo04} and Lemma \ref{lemma1-1}, we have
\begin{eqnarray*}
{\lambda''}_{k}\geq{\lambda'}_{pk}\geq 
\lambda_{pk}\geq
\frac{4\pi^{2}(pk)^{{2}/{n}}}
{\left(\omega_{n}|\mathbb{S}^n_+|\right)^{{2}/{n}}}
-c(n)= (pk\Gamma(n+1))^{2/n}-c(n)
\qquad k\in\mathbb{N}.
\end{eqnarray*}
If $M'=\mathbb{S}^n_+$, then $|\mathbb{S}^n_+|=p|M|$ and so

\begin{eqnarray*}
\lambda''_{k}\geq\frac{4\pi^{2}(pk)^{{2}/{n}}}{\left(\omega_{n}p|M|)\right)^{{2}/{n}}}-c(n),\qquad
\forall k\in\mathbb{N}.
\end{eqnarray*}
The conclusion of Theorem \ref{theorem-ti} follows and  Corollary \ref{theorem1-2} by taking 
$M''=\mathcal{W}^{n}_{{\pi}/{p}}$.
\end{proof}

Combining Lemma~\ref{lemma1-1} with Proposition~\ref{Theo03} yields the following result
for wedges tiling $M' =\mathbb{S}^n_+$ with $p$ tiles.

\begin{corollary} \label{CorTheo03}
 The  Dirichlet eigenvalues $\lambda''_k$ of  $\mathcal{W}^n_{{\pi}/{p}}$
satisfy
$$\lambda''_k\geq \left(\frac{n}{(n!)^{2/n}}\right) \cw(\mathcal{W}^n_{{\pi}/{p}})
\, k^{2/n}= np^{2/n}\,k^{2/n}, \quad\forall k\geq 1.$$
\end{corollary}

\section{Spheres\label{sec:spheres}}

We recall the notation given in Section 2.1. for the closed eigenvalues of $\mathbb{S}^n$, namely
$\cw=(n!/2)^{2/n}$, $\bl_K=K(K+n-1)$ for $K=0,1, \ldots$, 
$\sigma(K)=(K+1)^{\overline{n-1}}(n+2K)/n!$, and the lowest and highest orders
$k_-=\sigma(K-1)$, $k_+=\sigma(K)-1$ of a $K-$chain, and we define $\sigma(0)= 1$, corresponding to letting $k_-=k_+=0$ if $K=0$.
 We may further extend $\sigma^{-1}$ continuously down to $0$ by $\sigma^{-1}(0)=-1$, to obtain $\sigma^{-1}:[0, \infty)\to [-1, \infty)$
 (possibly complex valued).

 The next lemma follows immediately from the expressions of $k_{\pm}$ and $\bl_K$.

\begin{lemma} \label{idea} For each $K\in \mathbb{N}\cup \{0\}$, the eigenvalues of $\mathbb{S}^n$
 in the
$K-$chain, and its orders given by the integers $k_-\leq k\leq k_+$, satisfy the inequalities
 \[\begin{array}{lllllllll}
\sigma^{-1}(k+1) & \leq &  \sigma^{-1}(k_{+}+1) & = & K & = & \sigma^{-1}(k_-)+1 & \leq &  \sigma^{-1}(k)+1,\eqskip
  \lo(k)~ & \leq & ~ \lo(k_{+}) & = &\bl_K & = & \up (k_-)~\leq~\up (k),
\end{array}\]
where
\[\left\{
\begin{array}{lcl}
\up (k):= \left[\sigma^{-1}(k)+1\right]\left[\sigma^{-1}(k)+n \right],\\[2mm]
\lo(k):=\sigma^{-1}(k+1)\left[\sigma^{-1}(k+1)+n-1\right].
\end{array}\right.
\]
Equality holds in both right hand-sides if and only if $k=k_-$, and 
  in both left hand-sides  if and only if $k=k_+$.
\end{lemma}

\noindent

\subsection{$\mathbb{S}^2$}\label{sharpS2}
When  $n=2$, the eigenvalues are given by
$\bl_K=K(K+1)$, $K=0,1,2,\dots$, with sum of multiplicities $\sigma(K)=(K+1)^2$. The eigenvalues of the $K-$chain are given by
\begin{eqnarray*} 
\lambda_{K^{2}}=\lambda_{K^{2}+1}=\cdots=\lambda_{K^{2}+2K-1}
=\lambda_{(K+1)^{2}-1}=K(K+1).
\end{eqnarray*}
The next proposition now follows directly from Lemma~\ref{idea} and the fact that for $\mathbb{S}^2$
we have $ \sigma^{-1}(x)=\sqrt{x}-1$.

\begin{proposition} \label{theorem2-6}
We have $C_{W,2}(\mathbb{S}^2)=1$ and the eigenvalues $\lambda_k$ of the $2-$sphere, satisfy
\begin{eqnarray*}
k+1-\sqrt{k+1} ~\leq~ \lambda_{k}~\leq ~ k+\sqrt{k}
\end{eqnarray*}
for any $k\geq 0$. Moreover, for each $K=0,1, \ldots$, upper bounds are attained at $k=K^2$, that is, at the lowest order eigenvalue of a distinct eigenvalue $K(K+1)$,
while lower bounds are attained at the largest order $k=(K+1)^2-1$ of the same chain. 
\end{proposition}

\subsection{$\mathbb{S}^3$}
When $n=3$  the sum of multiplicities of the  eigenvalues
$\bl_K=K(K+2)$ , $K=0,1,2,\dots$, equals
$\sigma(K) =(K+1)(K+2)(2K+3)/6$. If  ${K}=\sigma^{-1}(x)$, then ${K}$ is the only real root of
$$ \sigma(K)= \frac{1}{3}(K+1)(K+2)(K+\fr{3}{2})= x.$$
Solving this third order polynomial equation allows us to determine
$ \sigma^{-1}(x)$  as described in the next lemma.
 
\begin{lemma} \label{G(x)}
If $n=3$ then  
$$ \sigma^{-1}(x)= \fr{3^{-{2}/{3}}}{2}G(x)^{{1}/{3}}+\fr{3^{-{1}/{3}}}{2}G(x)^{-{1}/{3}}-\frac{3}{2},$$
where $G$ is defined by
$$ G(x)= 108\, x +\sqrt{(108\, x)^2-3},
\quad \forall x\geq 0.$$
This function is univocally defined, smooth, positive and increasing for  $x\geq \fr{\sqrt{3}}{108}$, and  complex valued  for
$ x\in ({0},  \fr{\sqrt{3}}{108})$, where we  choose the branch with $G(0)=i\sqrt{3}$.
From $\sigma^{-1}(0)=-1$ we must have $\up (0)=0$ and
$  \fr{3^{-{2}/{3}}}{2}G(0)^{{1}/{3}}+\fr{3^{-{1}/{3}}}{2}G(0)^{-{1}/{3}}
=\fr{1}{2}$.
\end{lemma}

From this and Lemma~\ref{idea}, we may now derive the formulas of $\up (k)$ and $\lo(k)$.

\begin{lemma} \label{UPDOWN} We have $\up (0)=\lo(0)=0$ and for $k\geq 1$,
\begin{gather*}\begin{array}{lll}
 \up (k) & := &  -\fr{7}{12} + C_{W,3}\,
k^{{2}/{3}}\left( \fr{1}{2}\left[1+\sqrt{1-\frac{3}{(108\, k)^2}}\,\right]\right)^{{2}/{3}}
+ \sqrt{C_{W,3}}\,
k^{{1}/{3}}\left( \fr{1}{2}\left[1+\sqrt{1-\frac{3}{(108\, k)^2}}\,\right]
\right)^{{1}/{3}}\eqskip
& & \hspace*{5mm} +{ 2^{-{5}/{3}}3^{-{4}/{3}}k^{-{1}/{3}}\left[ 1+\sqrt{1-\fr{3}{(108\, k)^2}}\,\right]^{-{1}/{3}}}
+ { 2^{-{10}/{3}}3^{-{8}/{3}}k^{-{2}/{3}}\left[ 1+\sqrt{1-\fr{3}{(108\, k)^2}}\,\right]^{-{2}/{3}}} \eqskip
\lo(k) &:=& -\fr{7}{12} +C_{W,3}\,(k+1)^{{2}/{3}}
\left(\frac{1}{2}\left[1+\sqrt{1-\fr{3}{(108(1+k))^2}}\right]\right)^{{2}/{3}}
-\sqrt{C_{W,3}}\,(k +1)^{{1}/{3}}
\left(\frac{1}{2}\left[1+\sqrt{1-\fr{3}{(108(1+k))^2}}\right]\right)^{{1}/{3}}\eqskip
& & \hspace*{5mm} -{2^{-{5}/{3}}\cdot 3^{-{4}/{3}}
 (k+1)^{-{1}/{3}}\left[1+\sqrt{1-\fr{3}{(108(1+k))^2}}\right]^{-{1}/{3} }}
+ {2^{-{10}/{3}}\cdot 3^{-{8}/{3}}
 (k +1)^{-{2}/{3}}\left[1+\sqrt{1-\fr{3}{(108(1+k))^2}}\right]^{-{2}/{3} }}.
\end{array}
\end{gather*}
\end{lemma}

The inequalities obtained in the next proposition  are an immediate consequence of Lemma~\ref{idea}, using the expression for $G(x)$.

\begin{proposition}\label{closed 3-sphere}If $n=3$, we have $C_{W,3}=3^{2/3}$ and
the  eigenvalues for the $3$-sphere, $\lambda_k$, $k=0,1,\ldots$, satisfy the following inequalities
\begin{gather*}
 C_{W,3} \,(k+1)^{{2}/{3}}- \sqrt{C_{W,3} }\,(k+1)^{{1}/{3}} + R_+(3,k) ~ \leq~ 
\lambda_k ~\leq ~ C_{W,3}\,  k^{{2}/{3}}+ \sqrt{C_{W,3} }\, k^{{1}/{3}} + R_-(3,k),
\end{gather*}
where  both functions $R_{\pm}(3,k)$ are negative and bounded with values on a small interval. They are given by
$R_-(3,0)=0$, $R_+(3,0)=-C_{W,3} + \sqrt{C_{W,3} }$, and for $k\geq 1$,
\begin{gather*}
R_-(3,k) =
-C_{W,3} \left\{ k^{2/3}-2^{-2/3}\left( k+\sqrt{ k^2-2^{-4} 3^{-5} }\right)^{{2}/{3}}\right\}
 -\sqrt{C_{W,3}}\left\{ k^{1/3}-2^{-1/3}\left( k+\sqrt{ k^2-2^{-4} 3^{-5} }\right)^{{1}/{3}}\right\}\\
- \fr{7}{12} + 2^{-{10}/{3}}3^{-{8}/{3}}(k+\sqrt{k^2-2^{-4}3^{-5}})^{-{2}/{3}}
+  2^{-{5}/{3}}3^{-{4}/{3}}(k+\sqrt{k^2-2^{-4}3^{-5}})^{-{1}/{3}},
\end{gather*}
\begin{eqnarray*}
R_+(3, k)&=& 
 -C_{W,3} \left\{ (k+1)^{{2}/{3}}-2^{-2/3}\left( (k+1)+\sqrt{ (k+1)^2-{2^{-4} 3^{-5}} }\right)^{{2}/{3}}\right\}\\
 &&+\sqrt{C_{W,3} }\left\{ (k+1)^{{1}/{3}}-2^{-1/3}\left( (k+1)+\sqrt{ (k+1)^2-{2^{-4} 3^{-5}} }\right)^{{1}/{3}}\right\}\\
&&- \fr{7}{12} + 2^{-{10}/{3}}3^{-{8}/{3}}\left((k+1)+
\sqrt{(k+1)^2-2^{-4}3^{-5}}\right)^{-{2}/{3}}\\
&&-  2^{-{5}/{3}}3^{-{4}/{3}}\left((k+1)+\sqrt{(k+1)^2-2^{-4}3^{-5}}\right)^{-{1}/{3}}.
\end{eqnarray*}
 Moreover, $R_-(3,k)$ is a decreasing negative function for $k\geq 1$,
 and converges to $-7/12$ when $k\to  \infty$. 
Function $R_+(3,k)$ is a negative increasing function for all $k\geq 0$,  values
 $-C_{W,3}  +\sqrt{C_{W,3} }$ at $k=0$, and converges to $-7/12$ when $k\to \infty$.
Furthermore, in the inequality, lower bounds can be achieved  at the largest order eigenvalue of each $K-$chain, that is
$k=k_+=\sigma^{-1}(K)-1$, and upper bounds can be achieved at the lowest order eigenvalue of the same chain, $k=k_-=\sigma^{-1}(K-1)$.
\end{proposition} 

More details on the remaining functions $R_{\pm}(3,k)$ are given in next Lemma \ref{THIRD TERM}.
\begin{lemma}\label{THIRD TERM}
 The following functions $R_{\pm}(3,k)$ are negative and bounded with values on a small interval. They are given by
$R_-(3,0)=0$, $R_+(3,0)=-C_{W,3}+ \sqrt{C_{W,3}}$, and for $k\geq 1$ 
$$\begin{array}{lcl}
 R_-(3,k) &:=& A_2(k)+ A_1(k)-\frac{7}{12} + B_1(k) + B_2(k),\\[1mm]
R_+(3,k) &:=& D_2(k) +D_1(k) - \frac{7}{12} + E_1(k) + E_2(k),
\end{array}$$
where 
\begin{eqnarray*}
A_2(k)&=&-C_{W,3} \left\{ k^{2/3}-2^{-2/3}\left( k+\sqrt{ k^2-2^{-4} 3^{-5} }\right)^{{2}/{3}}\right\},\\
A_1(k)&=& -\sqrt{C_{W,3}}\left\{ k^{1/3}-2^{-1/3}\left( k+\sqrt{ k^2-2^{-4} 3^{-5} }\right)^{{1}/{3}}\right\},\\
B_1(k)&=& 2^{-{5}/{3}}3^{-{4}/{3}}\left(k+\sqrt{k^2-2^{-4}3^{-5}}\right)^{-{1}/{3}},\\
B_2(k) &=& 2^{-{10}/{3}}3^{-{8}/{3}}\left(k+\sqrt{k^2-2^{-4}3^{-5}}\right)^{-{2}/{3}},
\end{eqnarray*}
and 
\begin{eqnarray*}
D_2(k) = A_2(k+1), \quad
D_1(k) = -A_1(k+1),\\
E_1(k) =-B_1(k+1), \quad
E_2(k)= B_2(k+1).
\end{eqnarray*}
All functions converge to zero when $k\to +\infty$, $A_i+B_i$ are positive decreasing functions, $E_1+E_2$  and  $D_1+D_2$ are both negative and increasing. 
\end{lemma}

\begin{proof}[ Proof of Proposition \ref{closed 3-sphere} and Lemma~\ref{THIRD TERM}]
The construction of the remaining function $R_{\pm}(3,k)$  follows from Lemma~\ref{UPDOWN}, and similarly for the initial values at $k=0$.  We  see that 
each of these four terms converges to zero.
In Section~\ref{appb21} of Appendix~\ref{sec:MATH},  we prove $A_i(k)+B_i(K)$ are positive functions decreasing
to zero.
Consequently, $R_-(3,k)$ is a decreasing negative function for $k\geq \sqrt{3}/108$, decreasing  to  $-7/12$ when  $k \to \infty$. 
It follows immediately that $R_+(3,k)$  also converges  to  $-7/12$ when  $k \to\infty$.
It is clear that $E_1+E_2<0$. In the Appendix we compute the derivative of $D_2+D_1$
and using some Taylor estimations we conclude that it must be positive for all $k\geq 1$. Since $D_2+ D_1$ is negative at $k=1$ and increasing to zero, we
obtain that it must be negative for all $k\geq 1$. Then $R_+(3,k)$ is also negative. In a similar way, it is possible to prove that $E_1+E_2$ is an increasing
function and conclude that $R_+(3,k)$ is also increasing.
\end{proof}

\subsection{$\mathbb{S}^4$}
Now we consider the  eigenvalues of $\mathbb{S}^4$,  $\bl_K = K(K+3)$, $K=0,1, \ldots$, with $\sigma(K) = (K+1)(K+2)^2(K+3)/12$, $k_-=\sigma(K-1)$ and $k_+=\sigma(K)-1$.

\begin{lemma} When $n$ equals $4$ the inverse map of $\sigma$ is given by
$$\sigma^{-1}(x)= \frac{1}{2}\left( -4+\sqrt{2}\sqrt{1+\sqrt{1+48 x}}\right).$$
\end{lemma}

\begin{proposition}\label{closed 4-sphere} If $n=4$, we have $C_{W,4}=\sqrt{12}$ and 
the  eigenvalues $\lambda_k$, $k=0,1, \ldots$, of the $4$-dimensional sphere 
$\mathbb{S}^4$ satisfy the following inequalities
$$ C_{W,4} (k+1)^{{1}/{2}}-\sqrt{C_{W,4}}(k+1)^{{1}/{4}} +R_+(4,k)\leq \lambda_k\leq C_{W,4} \,k^{{1}/{2}}+\sqrt{C_{W,4}}\,k^{{1}/{4}} +R_-(4,k),$$
where  both $R_{\pm}(4,k)$ are negative bounded functions
 given by $R_-(4,0)=0$, $R_+(4,0)=-C_{W,4}+\sqrt{C_{W,4}}$ and for $k\geq 1$ 
\begin{gather*}
R_-(4,k) = -\frac{3}{2} + C_{W,4}\left\{\sqrt{k+\frac{1}{48}}-k^{{1}/{2}}\right\}
+ \sqrt{C_{W,4}}\left\{\sqrt{
\sqrt{k+\frac{1}{48 }}+ \frac{1}{2C_{W,4}}}-k^{{1}/{4}}\right\},\\[2mm]
R_+(4,k) = -\frac{3}{2} + C_{W,4}\left\{\sqrt{(k+1)+\frac{1}{48}}-(k+1)^{{1}/{2}}\right\}
- \sqrt{C_{W,4}}\left\{\sqrt{\sqrt{(k+1)+\frac{1}{48}}+ \frac{1}{2C_{W,4}}}- (k+1)^{{1}/{4}}
\right\}.
\end{gather*}
They satisfy $R_{-}(4,0)=0$, $R_+(4,0)=-C_{W,4}+\sqrt{C_{W,4}}$,
 and $\lim_{k\to +\infty}R_{\pm}(4,k)=
-\frac{3}{2}.$
Moreover, we have
 $R_+(4,k)\leq R_-(4, k+1)$, being $R_-(4,\cdot)$ decreasing while $R_+(4,\cdot)$ increasing.
Furthermore, lower bounds in the inequality can be achieved  at the largest order eigenvalue of each $K-$chain and upper bounds can be achieved at the lowest order eigenvalue of the same chain.
\end{proposition}

\begin{proof}
From Lemma \ref{idea} and  previous lemma,
\begin{align}\label{UPDOWN4one}
\up (k) &:= -\fr{3}{2} +\fr{1}{2}\sqrt{1+ 48 k}+{\sqrt{2}}\sqrt{1+\sqrt{1+ 48 k}},\\
\lo(k) &:= -\fr{3}{2} +\frac{1}{2}\sqrt{1+ 48 (k+1)}-\frac{1}{\sqrt{2}}\sqrt{1+\sqrt{1+ 48(k+1)}}.\label{UPDOWN4two}
\end{align}
Then $\up (0)=\lo(0)=0$ and for $k\geq 1$ we may write in the following way
\begin{gather*}
\up (k)=  -\frac{3}{2}+C_{W,4}\,k^{{1}/{2}}\sqrt{1+\frac{1}{48k}}+ \sqrt{C_{W,4}}\,k^{{1}/{4}}
\sqrt{\frac{1}{2C_{W,4}k^{{1}/{2}}}+\sqrt{ 1 +\frac{1}{48k}}}\\
=  C_{W,4}k^{{1}/{2}}+  \sqrt{C_{W,4}}\,k^{{1}/{4}} + R_-(4,k),\\
\lo(k) =-\frac{3}{2}+C_{W,4}(k+1)^{{1}/{2}}\sqrt{1+\frac{1}{48(k+1)}}- 
\sqrt{C_{W,4}}(k+1)^{{1}/{4}}
\sqrt{\frac{1}{2C_{W,4}(k+1)^{{1}/{2}}}+\sqrt{ 1 +\frac{1}{48(k+1)}}}\\
=  C_{W,4}(k+1)^{{1}/{2}}-  \sqrt{C_{W,4}}(k+1)^{{1}/{4}} + R_+(4,k),
\end{gather*}
where $R_{\pm}(4,k)$ are  given by
\begin{gather*}
R_-(4,k)= -\frac{3}{2}+C_{W,4}k^{{1}/{2}}\left\{ \sqrt{1+\frac{1}{48k}}-1\right\}+
\sqrt{C_{W,4}}k^{{1}/{4}}\left\{ \sqrt{\sqrt{1+\frac{1}{48k}}+
 \frac{1}{2C_{W,4} k^{{1}/{2}} }}-1\right\},\\
R_+(4,k) = -\frac{3}{2}+C_{W,4}(k+1)^{{1}/{2}}\left\{ \sqrt{1+\frac{1}{48(k+1)}}-1\right\}
-\sqrt{C_{W,4}}(k+1)^{{1}/{4}}\left\{ \sqrt{\sqrt{1+\frac{1}{48(k+1)}}+
 \frac{1}{2C_{W,4} (k+1)^{{1}/{2}} }}-1\right\}.
\end{gather*}
From (\ref{UPDOWN4one}) $\up (0)=0$, and we must have  $R_-(4,0)=0$. Similarly,
from (\ref{UPDOWN4two}) $\lo(0)=0$, and so  $R_+(4,0)=-C_{W,4}+\sqrt{C_{W,4}}$. Computing the corresponding limits we obtain
${\ds \lim_{k\to +\infty}} R_{\pm}(4,k)=-\frac{3}{2}$. The sign and monotonic properties of $R_{\pm}(4,k)$ are described in Appendix \ref{sec:MATH}, Section \ref{tilde{R}minus4}.
\end{proof}

\subsection{$\mathbb{S}^n$: proof of Theorem~\ref{theorem2-11}}

In this Section we will prove Theorem~\ref{theorem2-11} valid for $n\geq 2$. We start with the upper bound. 

\begin{lemma}\label{UPPER2termSn} For all $n\geq 2$ and $K\geq 0$ we have
$$ K(K+n-1)\leq \cw\,\sigma(K-1)^{{2}/{n}}+\sqrt{\cw}\,\sigma(K-1)^{{1}/{n}}$$
with equality only for $K=0$. Furthermore if we set for $K\geq 1$, 
\begin{gather*}
 \Omega(K):= \fr{\bl_K-\cw\,\sigma(K-1)^{2/n}}{
\sqrt{\cw}\,\sigma(K-1)^{1/n}}
= \fr{K(K+n-1)-\left(K^{\overline{n-1}}\left(\fr{n}{2}+K-1\right)\right)^{{2}/{n}}}{\left(K^{\overline{n-1}}\left(\fr{n}{2}+K-1\right)\right)^{{1}/{n}}},
\end{gather*}
then ${\ds \lim_{K\to +\infty}} \Omega(K)=1$.
\end{lemma}

\begin{proof}
Recall that $\sigma(-1)=0$, and so equality holds for $K=0$. Now assume $K\geq 1$.
We want to show that
$$ K(K+n-1) \leq \left(K^{\overline{n-1}}\left(K+\frac{(n-2)}{2}\right)\right)^{{2}/{n}}
+\left(K^{\overline{n-1}}\left(K+\frac{(n-2)}{2}\right)\right)^{{1}/{n}}. $$

Using Lemma \ref{Mother}
\begin{eqnarray*}
\lefteqn{
\left(K^{\overline{n-1}}\left(K+\frac{(n-2)}{2}\right)\right)^{{1}/{n}} =
 \left(K^{\overline{n-1}}\right)^{{1}/{n}}
\left(K+\frac{(n-2)}{2}\right)^{{1}/{n}}}\\
&\geq& \left(\sqrt{K(K+n-2)}\right)^{{(n-1)}/{n}}\left(K+\frac{(n-2)}{2}\right)^{{1}/{n}}
= \sqrt{K(K+n-2)} (\varphi(K))^{{1}/{2n}},
\end{eqnarray*}
where
$$\varphi(K):=\frac{\left(K+\frac{(n-2)}{2}\right)^{2}}{K(K+n-2)}
>1  \quad\quad(=1\mbox{~if~}n=2).
$$
Hence
\begin{eqnarray*}
\lefteqn{\left(K^{\overline{n-1}}\left(K+\frac{(n-2)}{2}\right)\right)^{{2}/{n}}
+\left(K^{\overline{n-1}}\left(K+\frac{(n-2)}{2}\right)\right)^{{1}/{n}}}\\
&\geq& K(K+n-2)(\varphi(K))^{{1}/{n}}+ \sqrt{K(K+n-2)}(\varphi(K))^{{1}/{2n}}\\
&\geq & K(K+n-2) + K = K(K+n-1).
\end{eqnarray*}
We have just proved that $\Omega(K)\leq 1$, that is, the inequality in the Lemma holds. Now we prove it converges to one at infinity.
Note that 
$$K^{\overline{n-1}}\left(\fr{n}{2}+K-1\right) =K^{\overline{n}}
\left(\fr{\frac{n}{2}+K-1}{K+n-1}\right).$$
From (\ref{risingpolynomial}), $K^{\overline{n}}=K^n\Big(1+\dsum_{i=0}^{n-1}{\hat{s}_i}/{K^i}\Big)$, where $\hat{s}_i$ is as in~\eqref{allsigmas},
and by the following Taylor expansion,
 $(1+y)^{a/n}=1+\fr{a}{n}y+\fr{1}{2}\fr{a}{n}\left(\fr{a}{n}-1\right)y^2+\so(y^2)$, 
we get for $n\geq 3$,
\begin{gather*}
\begin{array}{lll}
\left(K^{\overline{n}}\left(\fr{\fr{n}{2}+K-1}{K+n-1}\right)\right)^{\fr{a}{n}}
& = & \left(K^n\left[1+\fr{\hat{s}_1}{K}+\fr{\hat{s}_2}{K^2}+\so\left(\fr{1}{K^2}\right)
\right]\left(1-\fr{\fr{n}{2}}{K+n-1}\right)\right)^{\fr{a}{n}}\eqskip
&=& K^a\left( 1 + \left[\fr{\hat{s}_1}{K}-\fr{\fr{n}{2}}{K+n-1}\right]
+\left[\fr{\hat{s}_2}{K^2}-\fr{\fr{n}{2}\hat{s}_1}{K(K+n-1)}\right]+
 \so\left(\fr{1}{K^2}\right)\right)^{\fr{a}{n}}\eqskip
&=&  K^a\left( 1 + \fr{a}{n}\left[\fr{\hat{s}_1}{K}-\fr{\fr{n}{2}}{K+n-1}\right]+
 \bo\left(\fr{1}{K^2}\right)\right).
 \end{array}
\end{gather*}
Since $\hat{s}_1= {n(n-1)}/{2}$ (see Appendix \ref{sec:Combinatorics})
\begin{gather*}
\lim_{K\to +\infty}\Omega(K)=\lim_{K\to +\infty}\fr{ K^2+(n-1)K-K^2\left( 1 + \fr{2}{n}
\left[\fr{\hat{s}_1}{K}-\fr{\fr{n}{2}}{K+n-1}\right]+
 \bo\left(\fr{1}{K^2}\right)\right)
}{K\left( 1 + \fr{1}{n}\left[\fr{\hat{s}_1}{K}-\fr{\fr{n}{2}}{K+n-1}\right]+
 \bo\left(\fr{1}{K^2}\right)\right)}=1.
\end{gather*}
\end{proof}

\begin{proof}[Proof of the upper bound of Theorem \ref{theorem2-11}]
If $\lambda_k$ is in the $K-$chain, then
$k\geq \sigma(K-1)=k_-$ and from previous lemma
$$ \lambda_k=K(K+n-1)\leq \cw\, k_-^{{2}/{n}}+\sqrt{\cw}\, k_-^{{1}/{n}}
\leq \cw \,k^{{2}/{n}}+\sqrt{\cw}\, k^{{1}/{n}}.$$
Equality holds for $k=0$ but not for any other $k\geq 1$ in case $n\geq 3$, 
because $\varphi(K)>1$ (see proof of previous lemma). The inequality is asymptotically sharp as a consequence of ${\ds \lim_{K\to +\infty}} \Omega(K)=1$.
The case $n=2$ is stated in Proposition~\ref{theorem2-6}.
\end{proof}

In next proposition we derive three lower bound estimates for the eigenvalues of the $n$-sphere, where the last one is  the lower bound  stated in Theorem \ref{theorem2-11}.

\begin{proposition}\label{alternativeestimates} The  eigenvalues $\lambda_k$ of $\mathbb{S}^{n}$ satisfy the following estimates:\\[1mm]
(1) For any $k\geq 0$, we have
 $$\lambda_k ~\geq ~\cw\, (k+1)^{{2}/{n}}-2\sqrt{\cw}\,(k+1)^{{1}/{n}}-\fr{(n^2+2n-7)}{4}.$$
(2) 
For any $K\geq 1$, if $k$ is in the $K$-chain, then
\begin{eqnarray*}
\lambda_k &\geq& \cw (k+1)^{2/n}\left(1-\fr{1}{K}\right) -\left(\fr{n-1}{2}\right)^2.
\end{eqnarray*}
(3) For $k \geq 2n^n/n!-1$  we have
\begin{eqnarray*}
\lambda_k&\geq &  \cw (k+1)^{2/n} -\cw^{1/2} (k+1)^{1/n}-\left(\fr{n+1}{2}\right)^2  -\fr{n^2}{\cw^{1/2} (k+1)^{1/n}-n}.
\end{eqnarray*}

\end{proposition}
\begin{proof}
(1) The inequality   for $k=0$  translates into
$$(n!/2)^{2/n}-2(n!/2)^{1/n}-(n^2 +2n-7)/4\leq 0,$$ that
  decreases on $n$  and means $-1,25\leq 0$ for $n=2$, that is true.  Now we take $k\geq 1$ in a  $K$-chain, where $K\geq 1$. The corresponding higher-order eigenvalue is given by
$k_+=\sigma(K)-1$, where $\sigma(K)=(n+2K)(K+1)^{\overline{n-1}}/n!$. Applying  Lemma~\ref{Mother}  we obtain the following inequality
\begin{eqnarray}
(k_++1)\fr{n!}{2} &=& \left(K+\fr{n}{2}\right)(K+1)\ldots (K+n-1)\label{Kbound}\\
&\leq&  (K+n)(K+1)\ldots (K+n-1) = (K+1)^{\overline{n}}\nonumber\\
&\leq& \left[\left((K+1)+\fr{n-1}{2}\right)^2\right]^n
\nonumber\\
&=& \left( (K+1)(K+n)+\left(\fr{n-1}{2}\right)^2\right)^{n/2}. \nonumber
\end{eqnarray}
Therefore, 
\[(K+1)(K+n)\geq (k_+
+1)^{2/n}\left(\fr{n!}{2}\right)^{2/n} -\left(\fr{n-1}{2}\right)^2.\]
From the first identity in~\eqref{Kbound} we have  $(K+1)^n \leq (k_++1){n!}/{2}$, and, consequently,
\[K+1\leq (k_++1)^{1/n}\left(\fr{n!}{2}\right)^{1/n}.\]
Since $K(K+n-1) + 2K+ n = (K+1)(K+n)$ and applying the above inequalities 
 we obtain 
\begin{eqnarray*}
\lambda_k = \bl_{K} &=&(K+1)(K+n) -2K-n\\
&\geq& \cw(k_+ +1)^{2/n}-\left(\fr{n-1}{2}\right)^2-2(K+1) +2-n\\
&\geq & \cw(k_+ +1)^{2/n} -2\sqrt{\cw}(k_++1)^{1/n} -\fr{n^2+2n -7}{4}\\
&\geq & \cw(k +1)^{2/n} -2\sqrt{\cw}(k+1)^{1/n} -\fr{n^2+2n -7}{4},
\end{eqnarray*}
where we used in the last inequality the fact that
$k\leq k_+$ and $\zeta(x)=\cw x^2-2\sqrt{\cw}x$ is an increasing function for $x$ such that $\sqrt{\cw}x \geq 2$,
that is true for $x=k+1\geq 2$ and $n\geq 2$.\\[2mm]
(2) From (\ref{Kbound}) we derive an alternative estimate
\begin{eqnarray}
(k_++1)\fr{n!}{2} &=& \left(K+\fr{n}{2}\right)(K+1)\ldots (K+n-1)\nonumber\eqskip
&=& \left(1+\fr{n}{2K}\right)K(K+1)\ldots (K+n-1)\eqskip\label{Kbound2}
& = & \left(1+\fr{n}{2K}\right) K^{\overline{n}}\nonumber\eqskip
&\leq & \left(1+\fr{n}{2K}\right)\left[ K(K+n-1)+\left(\fr{n-1}{2}\right)^2\right]^{n/2},\label{Kbound3}
\end{eqnarray}
where we used in the last inequality Lemma~\ref{Mother}.
Therefore, we have
\[K(K+n-1) \geq \left(1+\fr{n}{2K}\right)^{-2/n}(k_++1)^{2/n}\left(\fr{n!}{2}\right)^{2/n}-\left(\fr{n-1}{2}\right)^2.\]
Since $(1+t/K)^{1/t}$ is decreasing in $t$ (for positive $t$ and $K$),  $(1-1/K)(1+n/2K)^{2/n}<(1-1/K^2)<1$, we get
 for $k$ in the $K$-chain,
\begin{eqnarray*}
\lambda_k &\geq& \cw (k+1)^{2/n}\left(1+\fr{n}{2K}\right)^{-2/n} -\left(\fr{n-1}{2}\right)^2\\
 &\geq &\cw (k+1)^{2/n}\left(1-\fr{1}{K}\right) -\left(\fr{n-1}{2}\right)^2.
\end{eqnarray*}
\noindent
(3)
Now inequality~\eqref{Kbound2} yields 
\begin{eqnarray*}
(k_++1)\fr{n!}{2} &\leq & \left(1+\fr{n}{2K}\right)\left(K+\fr{n-1}{2}\right)^n, 
\end{eqnarray*}
leading to
\begin{eqnarray*}
K+ \fr{n-1}{2} &\geq& (k_++1)^{1/n}\left(\fr{n!}{2}\right)^{1/n}\left( \fr{2K}{2K+n}\right)^{1/n},
\end{eqnarray*}
and so
\begin{eqnarray*}
K \geq \sqrt{\cw}(k+1)^{1/n}\left( \fr{2}{2+n}
\right)^{1/n}-\fr{n-1}{2}.
\end{eqnarray*}
For $k \geq ([(n-1)/2]^n(n+2)/n!)-1$  the right hand--side of previous inequality is positive and we obtain
\begin{equation}\label{Kbound4}
-{K}^{-1}\geq -\fr{1}{\cw^{1/n}(k+1)^{1/n}\left( \fr{2}{2+n}\right)^{1/n}-\fr{n-1}{2}}.
\end{equation}

Moreover, from (\ref{Kbound}) we have
\[ (k_-+1) \fr{n!}{2} = \left(K+\fr{n}{2}\right) (K+1)\ldots (K+n-1)\leq (K+n)^{n}.\]
Then $K+n \geq (k_-+1)^{1/n} \left(\fr{n!}{2}\right)^{1/n}$, and
\[K\geq \sqrt{\cw}(k_-+1)^{1/n} -n.\]
The right hand--side is positive for $k>2n^n/n!-1$ (which is larger than $([(n-1)/2]^n(n+2)/n!)-1$), and thus 
  \begin{equation}\label{Kbound5}
-\fr{1}{K}\geq -\fr{1}{\cw^{1/n}(k+1)^{1/n}-n}.
\end{equation}
In this case, from the second alternative estimate (2) we derive
\begin{eqnarray*}
\lambda_k &\geq & \cw (k+1)^{2/n} -\fr{1}{K}\cw (k+1)^{2/n}
-\left(\fr{n-1}{2}\right)^2\\
&\geq & \cw (k+1)^{2/n} -\fr{\cw (k+1)^{2/n}}{\cw^{1/2} (k+1)^{1/n}-n}-\left(\fr{n-1}{2}\right)^2\\
&\geq & \cw (k+1)^{2/n} -\fr{\cw (k+1)^{2/n}-n^2}{\cw^{1/2} (k+1)^{1/n}-n} -\fr{n^2}{\cw^{1/2} (k+1)^{1/n}-n} 
-\left(\fr{n-1}{2}\right)^2\\
&=& \cw (k+1)^{2/n} -\left( \cw^{1/2} (k+1)^{1/n}+n\right) 
-\left( \fr{n-1}{2}\right)^2 -\fr{n^2}{\cw^{1/2} (k+1)^{1/n}-n}\\
&=&  \cw (k+1)^{2/n} -\cw^{1/2} (k+1)^{1/n}-\left(\fr{n+1}{2}\right)^2  -\fr{n^2}{\cw^{1/2} (k+1)^{1/n}-n}.
\end{eqnarray*}
\end{proof}
  Next we prove the lower bound (3)  
  is sharp completing the proof of Theorem \ref{theorem2-11}.

\begin{proof}[Proof of the limit with respect to the lower bound in Theorem~\ref{theorem2-11}]
Consider the function 
\begin{gather*}
\Psi(K):=\fr{ \cw\, \sigma(K)^{{2}/{n}}-\bl_K}{\sqrt{\cw}\,
\sigma(K)^{{1}/{n}}}= 
\fr{\left((K+1)^{\overline{n-1}}\left(\fr{n}{2}+K\right)\right)^{2/n}-
K(K+n-1)}{\left((K+1)^{\overline{n-1}}\left(\frac{n}{2}+K\right)\right)^{1/n}}.
\end{gather*}
The equality of the two quotients comes from 
the value $\cw$ and (\ref{identity})  in Section~\ref{sec:eigenvaluesofspheres}. From $k_++1=\sigma(K)$, by showing that  $ {\ds \lim_{K\to +\infty}}\Psi(K)= 1$, we prove the limit in 
Theorem~\ref{theorem2-11}.

 In the following estimates we use the equality $(K+1)^{\overline{n-1}}=K^{\overline{n}}/K$,
Lemma~\ref{Mother}, and~\eqref{AQUI}.
\begin{eqnarray*}
\left((K+1)^{\overline{n-1}}\left(\fr{n}{2}+K\right)\right)^{{2}/{n}}-K(K+n-1)
&=& \left(K^{\overline{n}}\left(\fr{n}{2K}+1\right)\right)^{{2}/{n}}-K(K+n-1)\eqskip
&\geq & K(K+n-1)\left( \left(\fr{n}{2K}+1\right)^{{2}/{n}}-1\right),\\[2mm]
\left(K^{\overline{n}}\left(\fr{n}{2K}+1\right)\right)^{{2}/{n}}-K(K+n-1)
&\leq &
 \left(\left(\fr{n}{2K}+1\right)^{{2}/{n}}-1\right)K(K+n-1)\eqskip
 & & \hspace*{5mm} + \left(\fr{n}{2K}+1\right)^{2/n}c(n),
\end{eqnarray*}
where $c(n)=(n-1)(n-2)/6$. Setting
$$ \tau(K):= \fr{\left[(K+1)^{\overline{n-1}}\left(\fr{n}{2}+K\right)\right]^{2/n}-K(K+n+1)}{(K+n-1)},$$
we get 
$$  \lim_{K\to +\infty} \tau(K)\geq  \lim_{K\to +\infty} K \left[ \left(\fr{n}{2K}+1\right)^{{2}/{n}}-1\right]=1$$
and
$$ \lim_{K\to +\infty} \tau(K)\leq\lim_{K\to +\infty} K \left[ \left(\fr{n}{2K}+1\right)^{{2}/{n}}-1\right]
+\fr{c(n)}{K+n+1}\left(\fr{n}{2K}+1\right)^{2/n}= 1.$$
Hence ${\ds \lim_{K\to +\infty}}\tau(K)=1$.
Now,
 $$\Psi(K)=\tau(K)\cdot \fr{(K+n-1)}{\left(K^{\overline{n}}\left(\fr{n}{2K}+1\right)\right)^{{1}/{n}}},$$
and using the expression of $K^{\overline{n}}$ as in the proof of~Lemma~\ref{UPPER2termSn},
we have
\begin{eqnarray*}
\frac{(K+n-1)}{\left(K^{\overline{n}}\left(\fr{n}{2K}+1\right)\right)^{{1}/{n}}}=
\frac{(K+n-1)}{K\left(\left(1+\dsum_{i=0}^{n-1}\fr{\hat{s}_i}{K^i}\right)\left(\fr{n}{2K}+1\right)\right)^{{1}/{n}}}.
\end{eqnarray*}
Therefore,  ${\ds \lim_{K\to +\infty}}\Psi(K)=1$.
\end{proof}

\subsection{Converse P\'{o}lya's inequality for averages over chains on $\mathbb{S}^n$: proof of Theorem~\ref{thmLiYau}} \label{Converse averages}
We consider any element $k\in [k_-, k_+]$ of the $K-$chain, 
$k=k_j=\sigma(K-1)+j$ where $j=0,1,\ldots, m(K)-1$.
Let $\hat{s}_i$ be given as  in (\ref{allsigmas}). From~\eqref{risingpolynomial} we  have 
$ K^{\overline{n}}=K^n\left(1+z(K)\right)$ with $z(K)=\dsum_{i=1}^{n-1}
{\hat{s}_i}/{K^i}$. 
Furthermore, 
$$\begin{array}{lll} \fr{1}{m(K)}\dsum_{j=0}^{m(K)-1}j=\fr{m(K)-1}{2} & \mbox{ and } & \fr{1}{m(K)}\dsum_{j=0}^{m(K)-1}j^2=\fr{(2m(K)-1)(m(K)-1)}{6}.\end{array}$$
We define
$$ Pol_j(K):=\overline{\lambda}_K-\cw k_j^{2/n}.$$
P\'{o}lya's inequality  is given by the inequality $Pol_j(K)\geq 0$.
We have $Pol_0(0)=0$, but for $K=1$ $Pol_j(1)\geq 0$ is equivalent to $n\geq (n!(1+j)/2)^{2/n}$, $j=0,1,\ldots n$, that only holds  if $n=2$ and
$j=0,1$, or $n=3,4$ and $j=0$.
 We have the following conclusion on the P\'{o}lya average on $n$-spheres.

\begin{proposition} 
$1.$ If $n=2$, then for all $K\geq 0$, $\dsum_{k=k_-(K)}^{k_+(K)} (\lambda_k-\cw k^{2/n})= 0$.\\[1mm]
$2.$ If $n\geq 3$, 
there exists $K_n\geq 1$ s.t. for all $K\geq K_n$, $\dsum_{k=k_-(K)}^{k_+(K)} (\lambda_k-\cw k^{2/n})< 0.$
\end{proposition}
\begin{proof}
If $n=2$ then, $Pol_0(K)=K(K+1)$ and  for $K\geq 1$ and $j=0,\ldots, m(K)-1=2K$ we have 
$$Pol_j(K)=K(K+1)-(\sigma(K-1)+j)=K(K+1) -K^2 -j =K-j,$$
and  the average is given by
\[\fr{1}{m(K)}\dsum_{j=0}^{m(K)-1}Pol_j(K)= K-\fr{1}{2K+1}\left(\dsum_{j=0}^{2K}j\right)
=K-\fr{2K(2K+1)}{(2K+1)2}=0.\]

Now assume $n\geq 3$.
 Let
$$ \varphi(K):= \left(	\fr{\fr{n}{2}+K-1}{K+n-1}\right)^{2/n}=\left( 1-\fr{\fr{n}{2}}{K+n-1}\right)^{2/n}.$$
We  have
\begin{eqnarray*}
\cw k_j^{2/n} 
&=&\left(\left(	\fr{\fr{n}{2}+K-1}{K+n-1}\right) K^{\overline{n}}+ \fr{n!}{2}j\right)^{2/n}
=\varphi(K) K^2 (1+ w_j(K))^{2/n },
\end{eqnarray*}
where 
$$w_j(K)=  {z}(K)+\left(	\fr{K+n-1}{\fr{n}{2}+K-1}\right)\fr{n!j}{2K^n}= z(K)+
j G(K), \quad 
\mbox{with} \quad G(K)= \left(	\fr{K+n-1}{\fr{n}{2}+K-1}\right)\fr{n!}{2K^n}.$$
Now we assume $n\geq 3$.
Using Lemma~\ref{Taylor}, we have
$ (1+w_j)^{2/n}\geq  1+\fr{2}{n}w_j -\fr{(n-2)}{n^2}w_j^2$. Hence,
\begin{gather*}
\begin{array}{lll}
Pol_j(K) & := & K(K+n-1)-\cw\, k_j^{2/n}\eqskip
& = & K^2+(n-1)K  -\varphi(K)K^2\left(1+ w_j(K)\right)^{2/n}\eqskip
&\leq& (1-\varphi(K))K^2+(n-1)K  -\varphi(K)K^2\left( \frac{2}{n}w_j -\fr{(n-2)}{n^2}w_j^2\right)
\eqskip
&=& (1-\varphi(K))K^2+(n-1)K  -\varphi(K)K^2\left( \fr{2}{n}z-\fr{(n-2)}{n^2}z^2\right)
\eqskip
&& \hspace*{5mm} +2j\varphi(K)K^2\left( -\fr{1}{n}+\fr{(n-2)}{n^2}z \right)G+ j^2\varphi(K)K^2\fr{(n-2)}{n^2}G^2.
\end{array}
\end{gather*}

Now we estimate the average
\begin{gather*}
\begin{array}{lll}
\fr{1}{m(K)}\dsum_{j=0}^{m(K)-1}Pol_j(K) & \leq &
(1-\varphi(K))K^2+(n-1)K  -\varphi(K)K^2\left( \fr{2}{n}z-\fr{(n-2)}{n^2}z^2\right)\eqskip
&& \hspace*{5mm}+\varphi(K)K^2\left( -\fr{1}{n}+\fr{(n-2)}{n^2}z \right)G(m(K)-1)
+
\varphi(K)K^2\fr{(n-2)}{n^2}G^2\fr{(m(K)-1)(2m(K)-1)}{6}
\eqskip
&=& (1-\varphi(K))K^2+ (n-1)K-\varphi(K)K^2\left[ \fr{2}{n}\left(\dsum_{i=1}^{n-1}\fr{\hat{s}_i}{K^i}\right) -\fr{(n-2)}{n^2}\left( \dsum_{i=1}^{n-1}\fr{\hat{s}_i}
{K^i}\right)^2\right]
\eqskip
&& \hspace*{5mm}+\varphi(K)K^2\left(-\fr{1}{n}+\fr{(n-2)}{n^2}\left( \sum_{i=1}^{n-1}\fr{\hat{s}_i}
{K^i}\right)\right)
\times \left(\fr{K+n-1}{K-1+\fr{n}{2}}\right)\fr{n!}{2K^n}(m(K)-1)
\eqskip
&&\hspace*{10mm}+\varphi(K)K^2\fr{(n-2)}{n^2}\times \left[\left(\fr{K+n-1}{K-1+\fr{n}{2}}\right)
\fr{n!}{2K^n}\right]^2\left(\fr{(m(K)-1)(2m(K)-1)}{6(m(K)-1)}\right).
\end{array}
\end{gather*}
Note that $(1-\varphi(K))K\to 1$ when $K\to +\infty$. Hence $(1-\varphi(K))K^2\sim K$.
Moreover, $K^nG(K)\to n/2$  when $K\to +\infty$.
 Then we rearrange the w.h.s. to get an inequality as follows (we are assuming $n\geq 3$)
\[\fr{1}{m(K)}\dsum_{j=0}^{m(K)-1}Pol_j(K)\leq A(K) K + B(K) + \fr{C(K)}{K} +o\left(\fr{1}{K}\right),\]
where $A(K)$, $B(K)$, and $C(K)$ are bounded functions, namely,

\begin{gather*}
\begin{array}{lll}
A(K) & = & \left[(1-\varphi(K))(K+ (n-1)) -\varphi(K)\left(\fr{2}{n}\hat{s}_1
+\fr{1}{n}\left(\fr{K+{n-1}}{K+\fr{n-2}{2}}\right)\fr{(m(K)-1)n!}{2K^{n-1}}\right)\right],
 \eqskip
B(K)&=&\varphi(K)\left[-\fr{2}{n}\hat{s}_2 +\fr{(n-2)}{n^2}\hat{s}_1^2-\fr{(n-2)}{n^2}
\hat{s}_1\left(\fr{K+{n-1}}{K+\fr{n-2}{2}}\right)\left(\fr{(m(K)-1)n!}{2K^{n-1}}\right)
\right.\eqskip
&&\left.\quad\quad\quad\quad\quad
+\fr{(2m(K)-1)(m(K)-1)}{6}\fr{(n!)^2}{4K^{2(n-1)}}\left(\fr{(n-2)}{n^2}\right)\left(\fr{K+{n-1}}{K +\fr{n-2}{2}}\right)^2~\right],\eqskip
{C(K)}&=& {\varphi(K)}\left[-\fr{2}{n}\hat{s}_3 +\fr{2(n-2)}{n^2}\hat{s}_1\hat{s}_2 +\fr{(n-2)}{n^2}\hat{s}_2 \left(\fr{K+{n-1}}{K+\fr{n-2}{2}}\right)
\left(\fr{(m(K)-1)n!}{2K^{n-1}}\right) + \so\left(\fr{1}{K}\right)\right].
\end{array}
\end{gather*}
Now $B(K)$ and $C(K)$ are bounded, while
\[\lim_{K\to + \infty}A(K)= 2-n.\]
Hence, for $n\geq 3$ we have
$\fr{1}{m(K)}\dsum_{j=0}^{m(K)-1}Pol_j(K)<0$ for $K$ suficiently large.
\end{proof}
\txtb{The above results for averages over chains of eigenvalues allow us to prove Theorem~\ref{thmLiYau} by writing
the desired inequalities as the positivity of certain polynomials in two variables.
}
\begin{proof}[Proof of Theorem~\ref{thmLiYau}]
\txtb{
In the particular case of $\mathbb{S}^2$, and as shown in the previous proposition, the following sum up to a higher order index
$k=k_+(K)=\sigma(K-1)+m(K)-1=\sigma(K)-1$ satisfies
\[
\dsum_{j=0}^{k}(\lambda_j-j)=\dsum_{i=0}^{K}\dsum_{j=0}^{m(i)-1}
Pol_j(i)=0.
\]
Therefore, the total average up to an arbitrary index of a $K$-chain, $k=\sigma(K-1)+r$, $K\geq 1$,  where
$r=0, \ldots, m(K)-1=2K$, is given by
\[
\begin{array}{lll}
\fr{1}{k+1}\dsum_{j=0}^{k}(\lambda_j-j)&=&
\fr{1}{k+1}\left(\dsum_{j=0}^{k_+(K-1)}(\lambda_j-j)+ \dsum_{j=0}^r Pol_j(K)\right) \eqskip
& = & \fr{1}{k+1} \dsum_{j=0}^r Pol_j(K)\eqskip
&=&\fr{1}{k+1} \left( (r+1)K-\dsum_{j=0}^rj\right)\eqskip
& = & \fr{(r+1)(2K-r)}{2(K^2+r+1)}.
\end{array}
\]
Considering this as a function of the $r$ variable, it is straightforward to see that it has a global maximum equal to $1/2$
for positive $K$ and $r$ (attained when $r=K-1$), and so
\[\frac{1}{k+1}\dsum_{j=0}^k\lambda_j\leq
 \frac{1}{k+1}\dsum_{j=0}^k j +\fr{1}{2} = \fr{1}{2}\cw k+\fr{1}{2},
 \]
 thus proving the upper bound in the first part of Theorem~\ref{thmLiYau}.
}

\txtb{ For the lower bound, we need to prove that
\[
 \fr{(r+1)(2K-r)}{2(K^2+r+1)} \geq \fr{ k}{2(k+1)}\left| \sin\left( \pi \sqrt{k+1}\right)\right|,
\]
for $1\leq K$, $0\leq r \leq 2K$ and where $k=K^2+r$. We thus see that the two denominators cancel out and we need to show
that
\[
 \fr{K^2+r}{(r+1)(2K-r)}\left| \sin\left( \pi \sqrt{K^2+r+1}\right)\right| \leq 1.
\]
First note that
\[
 \left| \sin\left( \pi \sqrt{K^2+r+1}\right)\right| = \left| \sin\left( \pi \sqrt{K^2+r+1} - \pi K\right)\right|,
\]
and since $0\leq r \leq 2K$ we have $0 \leq \sqrt{K^2+r+1} - K \leq 1$. We may thus use the inequality \[ \sin(\pi x) \leq \pi x(1-x) + 4(4-\pi)x^2(1-x)^2\] valid
for $x$ in $[0,1]$ to obtain
\begin{gather*}
\begin{array}{lll}
 \fr {K^2+r}{(r+1)(2K-r)}\left| \sin\left( \pi \sqrt{K^2+r+1}\right)\right| & \leq & \fr{K^2+r}{(r+1)(2K-r)} (\sqrt{K^2+r+1}-K)(K+1-\sqrt{K^2+r+1})\eqskip
 & & \hspace*{5mm} \times \left[ \pi + 4(4-\pi)(\sqrt{K^2+r+1}-K)(K+1-\sqrt{K^2+r+1})\right]\eqskip
 & = & \fr{(K^2+r)\left[ \pi + 4(4-\pi)(\sqrt{K^2+r+1}-K)(K+1-\sqrt{K^2+r+1})\right]}{(\sqrt{K^2+r+1}+K)(K+1+\sqrt{K^2+r+1})}\eqskip
 &\leq & \fr{\pi K(K+2)}{(\sqrt{K^2+1}+K)(K+1+\sqrt{K^2+1})},
\end{array}
\end{gather*}
where in the last step we replaced $r$ by $2K$ and $0$ in the numerator and denominator, respectively. It is possible to see that the resulting function is smaller
than one for $K$ larger than or equal to five, while the remaining (finite) number of cases may be checked individually to see that, for integers $K$ and $r$ such
that $1\leq K \leq 4$ and $0\leq r \leq 2K$ the expression before the last inequality takes on its maximum value for $K=r=4$ which is given by
\[
 3 \pi  \left(987 \sqrt{21}-4523\right)-11808 \sqrt{21}+54112 \approx 0.967.
\]
}

\txtb{
In case of $\mathbb{S}^4$, recall that
\[
 \fr{n}{n+2}\cw\left(\mathbb{S}^4\right)=\fr{4}{\sqrt{3}} \mbox{, } m(K)=\fr{(2K+3)(K+1)(K+2)}{6} \mbox{ and }
 \sigma(K)=\fr{(K+1)(K+2)^2(K+3)}{12}.
\]
Given $k=\sigma(K-1)+r$ of the $K$-chain, where $K\geq 1$, $0\leq  r\leq m(K)-1$, we have
\[
\begin{array}{lll}
\dsum_{j=0}^k\lambda_j & = &  \dsum_{j=0}^{k_+(K-1)}\lambda_j +\dsum_{j= 0}^r\lambda_{\sigma(K-1)+j} \eqskip
& = &\dsum_{j=0}^{K-1}m(j)\bar{\lambda}_j+\dsum_{j=0}^r\bar{\lambda}_K\eqskip
&=&K (3 + K) \left[\fr{1}{18}(-1 + K) (1 + K)^2 (2 + K) +  (1 + r)\right].
\end{array}
\]
To show that a reversed Li-Yau inequality holds for all $k\geq 1$, is equivalent to showing
\[
\left(\sum_{j=0}^k\lambda_j\right)^2 \leq \left[\frac{4}{\sqrt{3}}(k+1)k^{1/2}\right]^2=\frac{16}{3}\left[\sigma(K-1)+r+1\right]^2\left[\sigma(K-1)+r\right].
\]
Replacing the above expression for the sum of the eigenvalues we see that this last inequality is, in turn, equivalent to the non-negativeness of the following polynomial
\[
\begin{array}{lll}
Q(K,r) & := &  288 K - 1488 K^2 + 392 K^3 + 1139 K^4 + 294 K^5 + 285 K^6 + 468 K^7 + 273 K^8\eqskip
 & & \hspace*{5mm} + 70 K^9 + 7 K^{10} + 1728 r + 1152 K r - 2160 K^2 r +  540 K^3 r + 1800 K^4 r\eqskip
 & & \hspace*{10mm} + 504 K^5 r - 72 K^6 r - 36 K^7 r  + 3456 r^2 + 864 K r^2 - 756 K^2 r^2\eqskip
 & & \hspace*{15mm} - 216 K^3 r^2 + 108 K^4 r^2 + 1728 r^3,
\end{array}
\]
for all $K\geq 1$ and  $0\leq r\leq m(K)-1$ -- for reference below, we note that $Q$ is the result of multiplication of the direct substitution by $324$.
We first note that the polynomial $1728r + 3456 r^2 + 1728 r^3 = 1728r(1+r)^2$ appearing in the expression for $Q$ above
is always positive for non-negative $r$, and it is thus enough to prove that $P(K,r) := Q(K,r) - 1728r(1+r)^2$ is always
positive in the region
\[
 A = \left\{ (K,r) \in \R^{2}: 1\leq K \wedge 0\leq r \leq m(K)-1\right\}.
\]
This has the advantage that now $P$ is a polynomial of the second degree in $r$. Differentiating it with respect to $r$ thus gives
a linear equation in this variable which may be solved to obtain
\[
 r = \fr{K^6+2 K^5-14 K^4-50 K^3-15 K^2+60 K-32}{6 \left(K^3-2 K^2-7 K+8\right)}.
\]
Differentiating $P$ with respect to $K$, and replacing in this derivative the expression for $r$ found above yields
\[
 \fr{4 (K+1)^3 (K+2)}{\left(K^3-2 K^2-7 K+8\right)^2}p(K),
\]
where
\[
\begin{array}{lll}
 p(K) & = & 10 K^{11}+27 K^{10}-217 K^9-750 K^8+1146 K^7+6075 K^6\eqskip
 & & \hspace*{5mm} +971 K^5-12924 K^4-2678 K^3+10644 K^2-2112 K-768.
\end{array}
\]
Since $p$ has no real zeros for $K$ greater than or equal to two, we only need to check the cases of $K$ equal to $1$ and $2$, and
what happens on the remaining part of the boundary of the set $A$ defined above. We have $P(1,r) = 1728 (1+r)$ which is minimal
when $r$ is zero. For $K=2$ we obtain $P(2,r)= -1296 r^2 + 33696 r + 216432$ which is greater than or equal to $P(2,0)=216432$ for
$r\leq m(2)-1=13$.
}

\txtb{
We have $P(K,0) = 7 K^{10}+70 K^9+273 K^8+468 K^7+285 K^6+294 K^5+1139 K^4+392 K^3-1488 K^2+288 K$, which
takes on its minimum on the interval $[1,+\infty)$ at $K=1$, where again it equals $1728$. Finally, over the line $r=m(K)-1$ we have
\[
 P(K,m(K)-1) = K (1 + K)^2 (2 + K)^2 (72 + 36 K + 32 K^2 + 59 K^3 + 34 K^4 + 7 K^5),
\]
which, on the interval $[1,+\infty)$, is greater than or equal to $P(1,m(1)-1) = P(1,4) = 8640$. We thus see that $P$ has $1728$
as its minimal value on the set $A$, attained when $K=1$ and $r=0$. Since $Q(K,r) \geq P(K,r)$ in $A$ and $Q(K,0) = P(K,0)$, we see
that the minimal value of $Q$ in $A$ is also $1728$.
}

\txtb{We have thus shown that
\[
 \fr{1}{324} Q(K,r) = \left[\frac{4}{\sqrt{3}}(k+1)k^{1/2}\right]^2 - \left(\sum_{j=0}^k\lambda_j\right)^2 \geq \fr{16}{3},
\]
from which the pretended result now follows from the inequality $\sqrt{1-x}\leq 1-x/2$.
}
\end{proof}

\appendix
\section{Auxiliary results }\label{sec:auxil}
In order to derive the P\'{o}lya--type inequalities we have made extensive usage of several results which we now collect here. Except for
Lemma~\ref{Mother} below, these are mostly known, but we include them here for ease of reference.

We will use the following Stirling bounds for $n!$. For any positive real $x>0$ there exists some $0<\theta <1$ such that \cite[6.1.38]{A},
$ \Gamma(x+1)= \sqrt{2\pi} x^{x+\frac{1}{2}}e^{-x+ \frac{\theta}{12x}}$. Hence, 
for any $n\geq 1$ we have
\begin{equation}\label{StirlingBounds}
  \sqrt{2\pi}e^{-n}n^{n+\frac{1}{2}}< n! \leq e^{1-n}n^{n+\frac{1}{2}}.
\end{equation}

The next two lemmas are also crucial tools for our eigenvalues estimates. The first of these gives sharp upper and lower bounds for the
quotients of two $\Gamma$ functions, which we were not able to find in the literature -- see~\cite[Theorem 2]{jame} for similar results.
\begin{lemma}\label{Mother}
 For any positive real number $K$ we have
\begin{equation}\label{GREAT}
 K(K+n-1)<\left(K^{\overline{n}}\right)^{{2}/{n}}< \txtb{K(K+n-1)+\frac{(n-1)(n-2)}{6}}
\end{equation} 
for all integer $n$ greater than or equal to three, while when $n$ is two we have identity \txtb{throughout}.
\end{lemma}
\begin{remark}
 Note that $\left(K^{\overline{n}}\right)^{{2}/{n}} = K^2 + (n-1) K + \frac{(n-1)(n-2)}{6} + \bo(K^{-2})$, as $K$ goes to infinity, and
 so the two inequalities refer to the first two and three terms in this expansion.
\end{remark}
\begin{proof} \txtb{We first prove the lower bound.}
For $s\in[0,n-1]$ the function $ \theta(s)=(K+s)(K+n-1-s)$ achieves its minimum at the end points of the interval,
namely $\theta(0)=\theta(n-1)=K(K+n-1)$, and its maximum at $s_+=(n-1)/2$ with $\theta(s_+)=(K+s_+)^2$.
For odd $n$ write  $n=2m+1$. Then $s_+=m\in \mathbb{N}$ and
\[
\begin{array}{lll}
K^{\overline{n}} & = & K(K+1)\ldots(K+m)\ldots(K+n-2)(K+n-1)\eqskip
&=&[K(K+n-1)][(K+1)(K+n-2)]\ldots[(K+m-1)(K+m+1)] \cdot [K+m]\eqskip
& > &[K(K+n-1)]^{m}\cdot [K+m]\eqskip
& \geq & [K(K+n-1)]^{m}\sqrt{K(K+n-1)}\eqskip
& = & [K(K+n-1)]^{n/2}.
\end{array}
\]
For $n$ even write $n=2m $. Then $s_+=m-(1/2)$ and

\[
\begin{array}{lll}
K^{\overline{n}}& = & K(K+1)\ldots(K+m-1)(K+m)\ldots(K+n-2) (K+n-1)\eqskip

&=&[K(K+n-1)][(K+1)(K+n-2)]\ldots[(K+m-1)(K+m)]\eqskip
& > & [K(K+n-1)]^{m}\eqskip
& = & [K(K+n-1)]^{n/2}.
\end{array}
\]
In either case we have
$ \left(K^{\overline{n}}\right)^{2}> [K(K+n-1)]^n$ as desired. 

\txtb{To prove the upper bound. we again consider the cases of even and odd $n$ separately. For even $n=2m$, we start from
\[
\begin{array}{lll}
K^{\overline{n}} & = & {\ds\prod_{j=0}^{m-1}}\left[(K+j)(K+n-1-j)\right]\eqskip
&=& {\ds\prod_{j=0}^{m-1}}\left[K(K+n-1)+j(n-1-j)\right]\eqskip
& < & \left[ K(K+n-1) +\fr{1}{m}\dsum_{j=0}^{m-1}j(n-1-j)\right]^m\eqskip
& = & \left[ K(K+n-1) -\fr{1}{6} (m-1) (2 m-3 n+2)\right]^m\eqskip
& = & \left[ K(K+n-1) + \fr{(n-1)(n-2)}{6}\right]^{n/2},
\end{array}
\]
where the inequality follows from the arithmetic-geometric inequality. For $n=2m+1$ we have
\[
\begin{array}{lll}
K^{\overline{n}} & = & {\ds\prod_{j=0}^{m-1}}\left[(K+j)(K+n-1-j)\right]\times (K+m)\eqskip
&=& {\ds\prod_{j=0}^{m-1}}\left[K(K+n-1)+j(n-1-j)\right]\times (K+m).
\end{array}
\]
We estimate $K+m$ from above by
\[
 K+m = K+\fr{n-1}{2} < \fr{K (K + n - 1) + \fr{(n - 1) (n - 2)}{6} + \fr{(n^2 - 1)}{24}}{\sqrt{K (K + n - 1) + \fr{(n - 1) (n - 2)}{6}}},
\]
which follows from the identity
\begin{gather*}
 \left[ K (K + n - 1) + \fr{(n - 1) (n - 2)}{6} + \fr{(n^2 - 1)}{24}\right]^2 - \left[K (K + n - 1) + \fr{(n - 1) (n - 2)}{6}\right] \left(K + \fr{n - 1}{2}\right)^2
 = \fr{(n^2-1)^2}{576}.
\end{gather*}
Again using the arithmetic-geometric inequality, we now obtain
\[
\begin{array}{lll}
K^{\overline{n}} & = & {\ds\prod_{j=0}^{m-1}}\left[(K+j)(K+n-1-j)\right]\times (K+m)\eqskip
& < & \left[ K(K+n-1)+\fr{1}{m+1}\dsum_{j=0}^{m-1}j(n-1-j)+ \fr{1}{m+1}\left( \fr{(n-1)(n-2)}{6} +
\fr{n^2-1}{24}\right)  \right]^{m+1} \eqskip
& & \hspace*{35mm}\times \fr{1}{\sqrt{K (K + n - 1) + \fr{(n - 1) (n - 2)}{6}}}\eqskip
& = & \left[ K(K+n-1) + \fr{(n-1)(n-2)}{6}\right]^{(n+1)/2}\times \fr{1}{\sqrt{K (K + n - 1) + \fr{(n - 1) (n - 2)}{6}}}\eqskip
& = & \left[ K(K+n-1) + \fr{(n-1)(n-2)}{6}\right]^{n/2}.
\end{array}
\]
}
\end{proof}

By a standard application of Taylor's theorem we have the following bounds for $(1+x)^{a/n}$ for $n\geq 2$ and  $a=1,2$, which we use
several times throughout the paper.
\begin{lemma}\label{Taylor}
Let $a=1,2$. For any $x\in [0, +\infty)$, and $l\geq 2$ even,  we have 
\begin{gather*}
\begin{array}{l}
1+\Lambda_1x+ \Lambda_2x^2+\ldots +\Lambda_lx^l~\leq ~ (1+x)^{a/n} 
~\leq ~ 1+\Lambda_1x+\ldots +\Lambda_{l-1}x^{l-1},
\end{array}
\end{gather*}
and for any $ x\in [0,1)$, and any $l\geq 1$ we have
\begin{gather*}
\begin{array}{l}
 (1-x)^{a/n} 
~\leq ~ 1-\Lambda_1x+ \Lambda_2 x^2-\Lambda_3x^3+\ldots +(-1)^l\Lambda_{l}x^{l},
\end{array}
\end{gather*}
where, for $l\geq 0$, 
\begin{gather}
\Lambda_{l+1}=\Lambda_{l+1}(a) := \frac{1}{(l+1)!}\frac{a}{n}\left(\frac{a}{n}-1\right)\left(\frac{a}{n}-2\right)\ldots \left(\frac{a}{n}-l\right).\label{Lambda(a)2}
\end{gather}
\end{lemma}

\section{Explicit expressions for $s_1(m)$, $s_2(m)$ and $s_3(m)$}\label{sec:Combinatorics}

From $\sigma_{1}(x_{1},\dots,x_{m}) = x_{1} + \cdots + x_{m}$ we immediately have
\[ s_1(m)=
 \sigma_1(1,\ldots, m)=\fr{m(m+1)}{2}.
\]
The expressions for $\sigma_{k}, k=2,3$, may now be derived using Newton's formulas~\cite[identity (2.11')]{macd}
\[
 k \sigma_{k}(x_{1},\dots,x_{m}) = \dsum_{i=1}^{k} (-1)^{i-1}p_{i}(x_{1},\dots,x_{m})\sigma_{k-i}(x_{1},\dots,x_{m}),
\]
where the polynomials $p_{k}$ are the $k^{\rm th}$ power sums defined by $p_{k}(x_{1},\dots,x_{m})= \dsum_{i=1}^{m} x_{i}^{k}$.
For $s_{2}(m)$ we have
\begin{gather*}
\begin{array}{lll}
s_{2}(m)= \sigma_{2}(1,2,\dots,m) & = & \fr{1}{2}\dsum_{i=1}^{2} (-1)^{i-1} \sigma_{2-i}(1,2,\dots,m)p_{i}(1,2,\dots,m)  \eqskip
 & = & \fr{1}{2}\sigma_{1}(1,2,\dots,m) (1+2+\cdots +m) -\fr{1}{2}\sigma_{0}(1,2,\dots,m) p_{2}(1,2,\dots,m) \eqskip
 & = & \fr{(m-1)m(m+1)(3m+2)}{24}.
\end{array}
\end{gather*}
In a similar way we may obtain
\[ s_3(m)=
 \sigma_3(1, \ldots, m)=\fr{(m-2)(m-1)m^2(m+1)^2}{48}.
\]

\section{Computations using Mathematica}\label{sec:MATH}
Some of the computations needed in the paper are straightforward but quite fastidious, involving many algebraic manipulations. They are thus quite suited to 
be carried out using a computer package such as Mathematica. Below we collect the main results used in the paper which were carried out in this way.

\subsection{$\mathbb{S}^n_+$}\label{C.1}

\subsubsection{$\tilde{R}_-(3,k)$ and $\hat{R}_-(3,k)$ }\label{tilde{R}minus3}

We prove that $\tilde{R}_-(3,k)$ is positive and decreasing to $2/3$ by proving
that $B(k)+C(k)$ is positive and decreasing to zero.
 Similar reasoning for $A(k)+D(k)$.

 Both $C(k)$ and $B(k)$ converge to  zero
when $k\to + \infty$.
Consider the function $z(t)=\La{(}1+\sqrt{1-3^{-5}t^{-2}}~\La{)}^{1/3}$ with $t\geq 1$, 
and $H(t)=(3t)^{1/3}z(t)$. Note that $z(t)$ is an increasing function and converges to $2^{1/3}$ when $t\to \infty$. Hence, $z(t)\in \left[(1+\sqrt{1-3^{-5}})^{1/3},2^{1/3}\right)$.
 The condition $B(k)+C(k)> 0$, with $k\geq 2$, $t=k-1$, 
is equivalent to 
\[Y(t):=\fr{B(k)+C(k)}{2}=-(6t)^{1/3}+\left(H(t)+\fr{1}{3H(t)}\right) > 0,\]
or equivalently
\[3\times 3^{1/3}6^{1/3}z(T)\left(1-2^{-1/3}z(t)\right)<\fr{1}{t^{2/3}}.\]
Note that $t=\left(3^{5}z(t)^{3}(2-z^3)\right)^{-1/2}$, 
and so, the above inequality
is equivalent to
\[2^{1/3}\left(1-2^{-1/3}z(t)\right)<(2-z(t)^3)^{1/3}.\]
Now, the inequality $(2^{1/3}-z)^3<(2-z^3)$ is valid for of $z\in \left[(1+\sqrt{1-3^{-5}})^{1/3},2^{1/3}\right)$, equality holds for $z=2^{1/3}$.
Therefore, $Y(t)>0$ for $t\geq 1$.

The derivative is given by
\begin{eqnarray*}
 Y'(t) &=& -\fr{2^{1/3}}{3^{2/3} t^{2/3}} - 
\fr{1}{9\times 3^{1/3} (t+\sqrt{-3^{-5}+t^2})^{1/3}\sqrt{-3^{-5}+t^2}} + 
\fr{(t+\sqrt{-3^{-5} + t^2})^{1/3}}{3^{2/3}\sqrt{-3^{-5} + t^2}}\\
&=& -\fr{2^{1/3}}{3^{2/3} t^{2/3}}+\fr{1}{3\sqrt{t^2-3^{-5}}}\left(H(t)-\fr{1}{3H(t)}\right).
\end{eqnarray*}
Then, $Y'(t)<0$ if and only if,
$\fr{1}{3t^{1/3}\sqrt{1-3^{-5}t^{-2}}}\left(3H(t)^2-1\right) < \fr{2^{1/3}}{3^{2/3} }3H(t).$
Using the expression of $\sqrt{1-3^{-5}t^{-2}}$ and  $t$ in terms of $z$ and $H(t)$ in terms of $t$ we get the equivalent inequality 
\[{\left(3 \times 3^{2/3}\times z(t)^2-t^{-2/3}\right)}<  3\left(z(t)^3-1\right)\times
{2^{1/3}}\times 3^{2/3}\times z(t).\]
Taking again the expression of $t$ in terms of $z$  we get
\begin{eqnarray*}
\left(3 \times 3^{2/3}\times z(t)^2-3^{5/3}z(t)(2-z(t)^3)^{1/3}\right)
&<& 2^{1/3}\times 3^{5/3}\left(z(t)^3-1\right)\times z(t),
\end{eqnarray*}
that after simplification is given by
$z(t) <  \left[2^{1/3}(z(t)^3-1) ~ + ~\left(2-z(t)^3\right)^{1/3} \right]$, 
that is,
\[(2-z(t)^3)^{1/3}>z -2^{1/3}(z(t)^3-1).\]
We can check this is true by considering  the polynomial inequality $(2-z^3)>(z-2^{1/3}(z^3-1))^3$, 
and to check to be true for $z\in [(1+\sqrt{1-3^{-5}})^{1/3}, 2^{1/3}]$,  being an equality 
at $z=2^{1/3}$. Thus, $Y' (t)<0$ is true for all $t\geq 1$.  We have then proved
$Y(t)$ is decreasing and so it is $B(k)+C(k)$. The same holds for $A(k)+D(k)$.

It follows that $\hat{R}_-(3,k)$  converges at infinity to the same limit as $\tilde{R_-}(3,k)$ , since 
$0<k^{a/3}-(k-1)^{a/3}\leq \frac{a}{3}
(k-1)^{-(3-a)/3}$  for $a=1,2$.  

Now using identity~\eqref{hatR3}, we will prove that $\hat{R}_-(3,k)$
is increasing by showing the following functions are increasing
\[\begin{array}{l}
DD(k):=D(k)+3^{-2}D(k)^{-1}-(6k)^{2/3},\\
CC(k):=  \sqrt{D}+3^{-1}(\sqrt{D}\,)^{-1}-(6k)^{1/3}.
\end{array}\]
\begin{gather*}
\begin{array}{lll}
DD'(k)&=& -\fr{ 2\left(27 +\fr{27^2(k-1)}{\sqrt{(27(k-1))^2-3}}\right)}{
3\times 3^{2/3}\left(\sqrt{(27(k-1))^2-3}+27(k-1)\right)^{5/3}}\eqskip
&& \hspace*{5mm}+ \fr{ 2\left(27 +\fr{27^2(k-1)}{\sqrt{(27(k-1))^2-3}}\right)}{
9\times 3^{1/3}\left(\sqrt{(27(k-1))^2-3}+27(k-1)\right)^{1/3}} -\fr{2^{5/3}}{3^{1/3}k^{1/3}}\eqskip
&=& \fr{ 2\left( 27 + 
\fr{ 243(k-1)}{\sqrt{(242/3)-162k+81k^2}}\right)
\left[ -3 + 3^{1/3}\left( 27(k-1) +\sqrt{(27(k-1))^2-3}\right)^{4/3}\right]}{
9\times 3^{1/3}\left(\sqrt{(27(k-1))^2-3}+27(k-1)\right)^{5/3}}\eqskip
& & \hspace*{5mm}-\fr{2^{5/3}}{3^{1/3}k^{1/3}}~.
\end{array}
\end{gather*}
Thus, $DD'(k)\geq 0$ if and only if
\begin{gather*}
\fr{ 2\left( 27 + 
\fr{ 729(k-1)}{\sqrt{(27(k-1))^2-3}}\right)
\left[ -3 + 3^{1/3}\left( 27(k-1) +\sqrt{(27(k-1))^2-3}\right)^{4/3}\right]}{
9\times 3^{2/3}\left(\sqrt{(27(k-1))^2-3}+27(k-1)\right)^{5/3}}-
\fr{2^{5/3}}{3^{1/3}k^{1/3}}>0.
\end{gather*}
Multiplying the above inequality by
\[9\times k^{1/3}\times 3^{1/3}\left(\sqrt{(27(k-1))^2-3}+27(k-1)\right)^{5/3}\times \sqrt{-3 +(27(k-1))^2},\]
we get an equivalent inequality, 
\[
\begin{array}{l}
 3 \times  k^{1/3}
\left[\left(\sqrt{(27(k-1))^2-3}+27(k-1) \right)^{4/3}-3^{2/3}\right] \eqskip \hspace*{10mm}>
 2^{2/3}\sqrt{(27(k-1))^2-3}
 \left( \sqrt{(27(k-1))^2-3}+27(k-1) \right)^{2/3},
\end{array}
\]
 that is,
\begin{equation}\label{DD'>0}
\begin{array}{lll}
W(k) & := & \fr{ 3 \times  k^{1/3}}{2^{2/3}\sqrt{(27(k-1))^2-3}}\times \left[
\left(\sqrt{(27(k-1))^2-3}+27(k-1) \right)^{2/3}\right.\eqskip
& & \hspace*{5mm}\left.-\fr{3^{2/3}}{\left(\sqrt{(27(k-1))^2-3}+27(k-1) \right)^{2/3}}\right]\eqskip
& > & 1.
\end{array}
\end{equation}
We prove this inequality is true by showing that $W$ decreases and at infinity it equals $1$.
We have that $W'(k)<0$ is equivalent to 
\begin{eqnarray*}
&&\left[ \left(\sqrt{-3 + 729 (-1 + k)^2} + 27 (-1 + k)\right)^{4/3}-  3^{2/3} \right] 
\times (729 (-1 + k)^2- 3)\\
&&\quad +  54 \left[3^{2/3} + \left(\sqrt{-3 + 729 (-1 + k)^2} + 27 (-1 + k)\right)^{4/3}\right]
\times \sqrt{-3 + 
    729 (-1 + k)^2} k\\
&&- 2187 \left[ \left(\sqrt{-3 + 729 (-1 + k)^2} + 27 (-1 + k)- 3^{2/3}\right)^{
     4/3}\right] \times  ( k-1) k <0,
\end{eqnarray*}
 or
\begin{eqnarray*}
&&\left[ \left(\sqrt{27^2 (k-1)^2-3} + 27 (k -1)\right)^{4/3}-  3^{2/3} \right] 
\times (27^2 (-1 + k)^2- 3)\\
&&\quad +  54 k\sqrt{27^2(k-1)^2-3}
\left[ \left(\sqrt{27^2 (k-1)^2-3} + 27 (k-1)\right)^{4/3} +3^{2/3} \right]\\
&<& 81\times 27\times  k(k-1)\times \left[ \left(\sqrt{27^2 (k-1)^2-3} + 27 (k-1)\right)^{4/3}- 3^{2/3}\right].
\end{eqnarray*}
Let $X(k)=\sqrt{27^2 (k-1)^2-3} + 27 (k-1)$. Then the previous inequality is equivalent to
\[\left(27^2(2k^2-k-1)+3\right)
> 2\times 27\times k \times \sqrt{27^2 (k-1)^2-3}\times\fr{\left(X(k)^{4/3}+3^{2/3}\right)}{
\left(X(k)^{4/3}-3^{2/3}\right)}.\]

Note that
$Z(k):=\fr{\left(X(k)^{4/3}+3^{2/3}\right)}{
\left(X(k)^{4/3}-3^{2/3}\right)}$ is  decreasing, and so $1<Z(k)\leq Z(2)=1.02062107...$,
so we are considering the equivalent inequality
\[F(k):={(27\times (2k^2-k-1)+3)^2}-Z(k)^2{ (2k)^{2}\times ((27)^2 (k-1)^2-3)} >0.\]

The Taylor series for $F(k)$ in powers of $(k-2)$ is given by
$$F(k)=4.487\times 10^6+1.1177\times 10^7(k-2)+8.77212\times 10^6(k-2)^2+ O((k-2)^3).$$
Hence, for $k$ large enough $W(k) $ is an increasing function, and so $W(k)>1$ for $k$ large.
It converges to 1 at infinity ( using Taylor expansions for example).
Hence $DD(k)$ is increasing for $k$ large. Similar for $CC$, proving that $\hat{R}_-(3,k)$ 
is increasing for $k$ large enough.

The values of $\tilde{R}_-(3,k)$ and of  $\hat{R}_-(3,k)$ given in the proposition may be obtained from the corresponding expressions.
In particular, $\hat{R}_-(3,k)$ vanishes for some $k$ between $35$ and $36$.

\subsubsection{$\tilde{R}_-(4,k), \hat{R}_-(4,k)$}\label{tilde{R}minus4}

We may write $\tilde{R}_-(4,k)$  as
\[\tilde{R}_-(4,k)=
\sqrt{5 + 4 \sqrt{1 + 24 ( k-1)}} + \sqrt{1 + 24 (k-1)} -\sqrt{24(k-1)} 
- 2\times (24(k-1))^{1/4}.\]
Since the functions
$S(t):=\sqrt{t+1}-\sqrt{t}$, and  $Q(t):= \sqrt{5+ 4\sqrt{1+t}}-2t^{1/4}$ are
decreasing, it follows that $\tilde{R}_-(4,k)$ is decreasing.

The behaviour of $ \hat{R}_-(4,k)$ can be described using \em Mathematica. \em 
 For example, it is negative at $k=1$, vanishing somewhere on the interval $[400,500]$
and attains a local maximum approx. $0.0322267$ around $k = 6452$ and converges to zero at infinity ( as explained in case $n=3$).

\subsection{$\mathbb{S}^n$}\label{C.2}
\subsubsection{$R_{\mp}(3,k)$}\label{appb21}
It is elementary to prove that $A_i$ and $B_i$  converge to zero when $k\to \infty$.
We prove now  that $A_i+B_i$ are positive decreasing functions.
Let $X(k)=108 +\sqrt{(108)^2-3k^{-2}}$. Then $X(1)>215$ and  $X(k)\nearrow 6^3$ when $k\to +\infty.$
We have
\begin{gather*}
A_1'(k)+B_1'(k) =
\fr{-18\times 3^{1/6}\times X(k)\times k^{4/3}+ 3^{5/6}\times X(k)^{1/3}
+18\times 3^{1/3}(\sqrt{3})^{-1}X(k)^{4/3}\left(-6+X(k)^{1/3}\right)k^2}{3k^{8/3}(X(k))^{4/3}
\sqrt{ 3 \times 108 - k^{-2}}}.
\end{gather*}
Note that $\left(-6+X(k)^{1/3}\right)k^2\to 0$ when $k\to +\infty$.
Moreover, $(-6+X(k)^{1/3})<0$ and so 
\[A_1'(k)+B_1'(k)< \fr{-18\times 3^{1/6}\times X(k)\times k^{4/3}+ 3^{5/6}\times X(k)^{1/3}}{3k^{8/3}(X(k))^{4/3} \sqrt{3 \times 108 - k^{-2}}}<0.\]
Similarly  $A'_2(k)+B'_2(k)<0$.
Since $A_1(1)+B_1(1)=0.0577504$ and $ A_2(1)+B_2(1)= 0.00324951$ and
$A_i+B_i$ decrease to zero, we conclude they are positive for all $k\geq 1$.
Moreover, $A_1(1)+B_1(1)+A_2(1)+B_2(1)= 0.0609999< 7/12$.
This proves $R_-(3,k)$ is negative decreasing to $-7/12$.\\

It follows immediately that $D_i(k)$ and $E_i(k)$ are
 functions converging to zero. Clearly $E_1(k)+E_2(k)<0$. 
Let $V(t)=D_2(t-1)+D_1(t-1)$, where $t=k+1\geq 2$. Then
\begin{gather*}
V(t) = -\sqrt{C_{W,3}}\left[ t^{1/3}-2^{-1/3}\left( t+\sqrt{ t^2-2^{-6} 3^{-5} }\right)^{1/3}\right]
\times\left\{\sqrt{C_{W,3}}\left[ t^{1/3}+2^{-1/3}\left( t+\sqrt{ t^2-{2^{-6} 3^{-5}} }\right)^{{1}/{3}}\right]-1\right\}
\end{gather*}
and
\begin{eqnarray*}
V' (t)&=& \fr{1}{3^{2/3} t^{2/3}} -\fr{2}{3^{1/3} t^{1/3}} 
- \fr{1 + \fr{ t}{\sqrt{t^2-2^{-6} 3^{-5}}} }
{ 2^{1/3} 3^{2/3} (t + \sqrt{t^2-2^{-6} 3^{-5}})^{2/3}} + \fr{2^{
  1/3} \left(1 + \fr{t}{\sqrt{t^2-2^{-6}3^{-5}}}\right)}{3^{1/3}(t + \sqrt{t^2-2^{-6} 3^{-5}})^{
 1/3}}.
\end{eqnarray*}
It follows that $V(t)<0$, that is $D_1(k)+D_2(k)$ is negative. Now
we show that 
$$2^{1/3}(3t)^{2/3}\sqrt{t^2-3^{-5}2^{-6}}\times V'(t)>0,$$
 that is,
\begin{gather*}
(t + \sqrt{t^2-2^{-6} 3^{-5}})^{1/3}t^{2/3}[12^{1/3}(t + \sqrt{t^2-2^{-6} 3^{-5}})^{1/3}-1]
>2^{1/3}\sqrt{t^2-2^{-6} 3^{-5}}[2\times 3^{1/3}t^{1/3}-1].
\end{gather*}
Let us define the constant $\delta= 3^{-5}2^{-6}$ and  the following functions
\[J(t)=\sqrt{1-\delta t^{-2}}, \quad \xi(t)=\left(\fr{J(t)+1}{2}\right)^{1/3},
\quad a(t)=2\times 3^{1/3}t^{1/3}.\]
Previous inequality is equivalent to the following one,
\begin{equation}\label{B.2}
(J(t)+1)^{1/3}\left( 2^{2/3}3^{1/3}t^{1/3}(1+J(t))^{1/3}-1\right)>2^{1/3}J(t)(2\times 3^{1/3}t^{1/3}-1).
\end{equation}
 We will prove this inequality holds for all  $t\geq 2$, 
or equivalently,
\begin{equation}\label{B.3}
\xi(t)\times\left(a(t)\xi(t)-1\right)> \left(2\xi^3(t)-1\right)\times\left(a(t)-1\right),
\end{equation}
where we used in the w.h.s. the identity $J(t)=2\xi(t)^3-1$. Now (\ref{B.2}) is equivalent to
\begin{equation}\label{B.4}
a(t)\times\left(\xi^2(t)-2\xi(t)^3+1\right)>\xi(t)-2\xi(t)^3+1.
\end{equation}
Since $0<\xi<1$ we have  $ \xi+ 1 > \xi^2 +1 > \xi^3+\xi^3=2\xi^3.$
Hence  (\ref{B.4}) is equivalent to 
\begin{equation}\label{B.5}
a(t)\times \nu(\xi(t))>1,\quad\mbox{where}~ \nu(\xi):=\left(\fr{\xi^2-2\xi^3+1}{\xi-2\xi^3+1}\right),
\quad \forall \xi\in[0,1).
\end{equation}
From the   derivative of $\nu(\xi)$, given by
\[ \nu'(\xi)= \fr{2\xi^2-1}{(2\xi^2+ 2\xi+1)^2},\]
we conclude that  $\nu(\xi)$
attains minimum value $m_0=0.792893$  at $\xi = 2^{-1/2}$ 
and $a(t)$ is an increasing function with $a(1)=2\times 3^{1/3}$. Hence, 
$a(t)m_0 \geq a(1)m_0>1$ implying~(\ref{B.5}) holds. 

 We have then proved that
$D_1(k)+ D_2(k)$ is a negative function increasing to zero.

Therefore, the sum  $E_1+E_2+D_1+D_2$ is negative and converges to zero.
Now we prove that  $E_1+E_2$ is also  increasing.
The functions $E_i$ are of the form $E_1(k)= -1/Y(k)$ where $Y(k)$ increases and is positive with $Y(k)>2$, and $E_2(k)=1/(Y(k)^2)$. Hence, 
$$E_1'(k)+E_2' (k)= \fr{Y'(k)(Y(k)-2)}{Y(k)^3}>0.$$
 
Therefore, $R_+(3,k)$ is negative, increasing and converges  to $-7/12$. 

\subsubsection{$R_{\mp}(4,k)$}\label{appc22}

We have $R_-(4,k)=R_1(k)+R_2(k)$ where
\begin{eqnarray*}
R_1(k)&=& \sqrt{12k}\left(\sqrt{1+\fr{1}{48k}}-1\right)=\sqrt{12k+ \fr{1}{4}}
-\sqrt{12 k},\\
R_2(k)&=& (12 k)^{1/4}\left( \sqrt{\fr{1}{2\sqrt{12k}}+\sqrt{1+\fr{1}{48k}}}-1\right)
= \left(\sqrt{12k+\fr{1}{4}}+\fr{1}{2}\right)^{1/2}-(12k)^{1/4}.
\end{eqnarray*}
The derivative of $R_1(k)$ is negative. The same holds for $R_2$ since
$$ t^{3/4}<\left(t+\fr{1}{4}\right)^{1/2 }\left[\left(t+\fr{1}{4}\right)^{1/2 }+\fr{1}{2}\right]^{1/2}.$$
 Note that $R_1(0)= 1/2$, $R_2(0)=1$ and both $R_i$ are decreasing 
 positive everywhere. For $k\geq 1$ $R_-(4,k)\leq R_-(4,1)= -1.32531$. Hence $R_-(4,k)$ is a negative function decreasing to $-3/2$.
Since $R_+(4,k)<R_-(4, k+1)$ then $R_+(4,k)$ is also negative.
 Moreover,
$$R'_+(4,k)=R'_1(k+1)-R'_2(k+1)= \fr{3^{1/4}}{2^{3/2}(k+1)^{3/4}}-\fr{\sqrt{3}}{\sqrt{k+1}}+
\fr{6\left(2\sqrt{1+\sqrt{1+48(k+1)}}-\sqrt{2}\right)}{\sqrt{1+\sqrt{1+48(k+1)}}\sqrt{1+48(k+1)}}.$$
Multiplying by $\sqrt{(k+1)}$ the above expression and the fact that
$$\fr{3^{1/4}}{2^{3/2}(k+1)^{1/4}}+
6(k+1)^{1/2}\left(\fr{2}{\sqrt{1+48(k+1)}}-\fr{\sqrt{2}}{\sqrt{1+\sqrt{1+48(k+1)}}\sqrt{1+48(k+1)}}\right)>\sqrt{3} $$
 we conclude that
 $R'_+(4,k)>0$, that is $R_+(4,k)$ increases to $-2/3$.
It remains to prove the above fact, or equivalently,
\[\fr{12(k+1)^{1/2}}{\sqrt{1+48(k+1)}}\left(1-\fr{1}{\sqrt{2}\sqrt{1+\sqrt{1+48(k+1)}}}\right)
>\sqrt{3}\left(1- \fr{1}{3^{1/4}2^{3/2}(k+1)^{1/4}}\right).\]
From $(12)^2(k+1)>3(1+48(k+1))$ we conclude that
\[\fr{12(k+1)^{1/2}}{\sqrt{1+48(k+1)}}>\sqrt{3},\]
and from $(1+\sqrt{1+48(k+1)})>4 \times 3^{1/2}(k+1)^{1/2}$ we have
\[\fr{1}{2^{3/2}3^{1/4}(k+1)^{1/4}}> \fr{1}{\sqrt{2}\sqrt{1+\sqrt{1+48(k+1)}}},\]
and the stated fact follows.

\section*{Acknowledgements} 
P.F. and I.S. were partially supported by the Funda\c c\~{a}o para a Ci\^{e}ncia e a Tecnologia (Portugal)
through project UIDB/00208/2020.

\end{document}